\documentclass[11pt]{amsart}

\usepackage{anysize} \marginsize{1in}{1in}{1in}{1in}
\usepackage{comment}
\usepackage{listings}
\usepackage{xcolor}
\usepackage{amsmath}
\usepackage{mathtools}
\usepackage{booktabs}
\usepackage[all]{xy}
\usepackage[utf8]{inputenc}
\usepackage{varioref}
\usepackage{amsfonts}
\usepackage{amssymb}
\usepackage{bbm}
\usepackage{esint}
\usepackage{graphicx}
\usepackage{subcaption}
\usepackage{tikz}
\usepackage{empheq}
\usepackage{enumitem}
\usepackage{tikz-cd}
\usepackage{todonotes}
\usepackage{svg}
\usepackage{inconsolata}
\usepackage{upgreek}

\usetikzlibrary{matrix,arrows,decorations.pathmorphing}
\usepackage{mathrsfs}
\usepackage[hypertexnames=false,backref=page,pdftex,
 	pdfpagemode=UseNone,
 	breaklinks=true,
 	extension=pdf,
 	colorlinks=true,
 	linkcolor=blue,
 	citecolor=red,
 	urlcolor=blue,
 ]{hyperref}

\usepackage{minted}
\setminted{
  linenos,
  breaklines,
  autogobble,
  fontsize=\footnotesize,
  frame=single
}

\newcommand{\sC}{{\mathcal C}}

\newcommand{\sE}{{\mathcal E}}

\newcommand{\sL}{{\mathcal L}}

\newcommand{\C}{{\mathbb C}}

\newcommand{\N}{{\mathbb N}}

\newcommand{\R}{{\mathbb R}}

\newcommand{\Z}{{\mathbb Z}}

\newcommand{\dd}{{{\mathrm{d}}}}

\newcommand{\supp}{\operatorname{supp}}

\renewcommand{\to}[1][]{\xrightarrow{\ #1\ }}

\newcommand{\cp}{\mathbb C \mathrm P}
\newcommand{\rp}{\mathbb R \mathrm P}

\newcommand{\ev}{\operatorname{ev}}

\newcommand{\RR}{\mathbb{R}}
\newcommand{\SH}{\operatorname{SH}}
\newcommand{\CF}{\operatorname{CF}}
\newcommand{\HF}{\operatorname{HF}}

\makeatletter
\newcommand*{\da@rightarrow}{\mathchar"0\hexnumber@\symAMSa 4B }
\newcommand*{\da@leftarrow}{\mathchar"0\hexnumber@\symAMSa 4C }
\newcommand*{\xdashrightarrow}[2][]{%
  \mathrel{%
    \mathpalette{\da@xarrow{#1}{#2}{}\da@rightarrow{\,}{}}{}%
  }%
}
\newcommand{\xdashleftarrow}[2][]{%
  \mathrel{%
    \mathpalette{\da@xarrow{#1}{#2}\da@leftarrow{}{}{\,}}{}%
  }%
}
\newcommand*{\da@xarrow}[7]{%
  \sbox0{$\ifx#7\scriptstyle\scriptscriptstyle\else\scriptstyle\fi#5#1#6\m@th$}%
  \sbox2{$\ifx#7\scriptstyle\scriptscriptstyle\else\scriptstyle\fi#5#2#6\m@th$}%
  \sbox4{$#7\dabar@\m@th$}%
  \dimen@=\wd0 %
  \ifdim\wd2 >\dimen@
    \dimen@=\wd2 %
  \fi
  \count@=2 %
  \def\da@bars{\dabar@\dabar@}%
  \@whiledim\count@\wd4<\dimen@\do{%
    \advance\count@\@ne
    \expandafter\def\expandafter\da@bars\expandafter{%
      \da@bars
      \dabar@ 
    }%
  }%
  \mathrel{#3}%
  \mathrel{%
    \mathop{\da@bars}\limits
    \ifx\\#1\\%
    \else
      _{\copy0}%
    \fi
    \ifx\\#2\\%
    \else
      ^{\copy2}%
    \fi
  }%
  \mathrel{#4}%
}
\makeatother

\makeatletter
\newsavebox\myboxA
\newsavebox\myboxB
\newlength\mylenA

\newcommand*\xtilde[2][0.8]{%
    \sbox{\myboxA}{$\m@th#2$}%
    \setbox\myboxB\null
    \ht\myboxB=\ht\myboxA%
    \dp\myboxB=\dp\myboxA%
    \wd\myboxB=#1\wd\myboxA
    \sbox\myboxB{$\m@th\widetilde{\copy\myboxB}$}
    \setlength\mylenA{\the\wd\myboxA}
    \addtolength\mylenA{-\the\wd\myboxB}%
    \ifdim\wd\myboxB<\wd\myboxA%
       \rlap{\hskip 0.5\mylenA\usebox\myboxB}{\usebox\myboxA}%
    \else
        \hskip -0.5\mylenA\rlap{\usebox\myboxA}{\hskip 0.5\mylenA\usebox\myboxB}%
    \fi}

\newbox\usefulbox

\def\getslant #1{\strip@pt\fontdimen1 #1}

\def\xxtilde #1{\mathchoice
 {{\setbox\usefulbox=\hbox{$\m@th\displaystyle #1$}%
    \dimen@ \getslant\the\textfont\symletters \ht\usefulbox
    \divide\dimen@ \tw@ 
    \kern\dimen@ 
    \xtilde{\kern-\dimen@ \box\usefulbox\kern\dimen@ }\kern-\dimen@ }}
 {{\setbox\usefulbox=\hbox{$\m@th\textstyle #1$}%
    \dimen@ \getslant\the\textfont\symletters \ht\usefulbox
    \divide\dimen@ \tw@ 
    \kern\dimen@ 
    \xtilde{\kern-\dimen@ \box\usefulbox\kern\dimen@ }\kern-\dimen@ }}
 {{\setbox\usefulbox=\hbox{$\m@th\scriptstyle #1$}%
    \dimen@ \getslant\the\scriptfont\symletters \ht\usefulbox
    \divide\dimen@ \tw@ 
    \kern\dimen@ 
    \xtilde{\kern-\dimen@ \box\usefulbox\kern\dimen@ }\kern-\dimen@ }}
 {{\setbox\usefulbox=\hbox{$\m@th\scriptscriptstyle #1$}%
    \dimen@ \getslant\the\scriptscriptfont\symletters \ht\usefulbox
    \divide\dimen@ \tw@ 
    \kern\dimen@ 
    \xtilde{\kern-\dimen@ \box\usefulbox\kern\dimen@ }\kern-\dimen@ }}%
 {}}

\newcommand*\xoverline[2][0.75]{%
    \sbox{\myboxA}{$\m@th#2$}%
    \setbox\myboxB\null
    \ht\myboxB=\ht\myboxA%
    \dp\myboxB=\dp\myboxA%
    \wd\myboxB=#1\wd\myboxA
    \sbox\myboxB{$\m@th\overline{\copy\myboxB}$}
    \setlength\mylenA{\the\wd\myboxA}
    \addtolength\mylenA{-\the\wd\myboxB}%
    \ifdim\wd\myboxB<\wd\myboxA%
       \rlap{\hskip 0.5\mylenA\usebox\myboxB}{\usebox\myboxA}%
    \else
        \hskip -0.5\mylenA\rlap{\usebox\myboxA}{\hskip 0.5\mylenA\usebox\myboxB}%
    \fi}

\def\xxoverline #1{\mathchoice
 {{\setbox\usefulbox=\hbox{$\m@th\displaystyle #1$}%
    \dimen@ \getslant\the\textfont\symletters \ht\usefulbox
    \divide\dimen@ \tw@ 
    \kern\dimen@ 
    \overline{\kern-\dimen@ \box\usefulbox\kern\dimen@ }\kern-\dimen@ }}
 {{\setbox\usefulbox=\hbox{$\m@th\textstyle #1$}%
    \dimen@ \getslant\the\textfont\symletters \ht\usefulbox
    \divide\dimen@ \tw@ 
    \kern\dimen@ 
    \xoverline{\kern-\dimen@ \box\usefulbox\kern\dimen@ }\kern-\dimen@ }}
 {{\setbox\usefulbox=\hbox{$\m@th\scriptstyle #1$}%
    \dimen@ \getslant\the\scriptfont\symletters \ht\usefulbox
    \divide\dimen@ \tw@ 
    \kern\dimen@ 
    \xoverline{\kern-\dimen@ \box\usefulbox\kern\dimen@ }\kern-\dimen@ }}
 {{\setbox\usefulbox=\hbox{$\m@th\scriptscriptstyle #1$}%
    \dimen@ \getslant\the\scriptscriptfont\symletters \ht\usefulbox
    \divide\dimen@ \tw@ 
    \kern\dimen@ 
    \xoverline{\kern-\dimen@ \box\usefulbox\kern\dimen@ }\kern-\dimen@ }}%
 {}}
\makeatother

\makeatletter
\newcommand{\mylabel}[2]{#2\def\@currentlabel{#2}\label{#1}}
\makeatother

\makeatletter
\newcommand{\Mac}{}
\DeclareRobustCommand{\Mac}{%
  M%
  \raisebox{\dimexpr\fontcharht\font`M-\height}{%
    \check@mathfonts\fontsize{\sf@size}{0}\selectfont
    c%
  }%
}
\makeatother

\newtheoremstyle{citing}
  {}
  {}
  {\itshape}
  {}
  {\bfseries}
  {\textbf{.}}
  {.5em}
  {\thmnote{#3}}

\theoremstyle{plain}
\newtheorem{theorem}{Theorem}

\newtheorem{lemma}[theorem]{Lemma}
\newtheorem{corollary}[theorem]{Corollary}

\newtheorem{bigthm}{Theorem}

\newtheorem{proposition}[theorem]{Proposition}

\theoremstyle{remark}

\theoremstyle{definition}

\newtheorem{definition}[theorem]{Definition}

\numberwithin{equation}{section}

\theoremstyle{remark}
\newtheorem{remark}[theorem]{Remark}

{\theoremstyle{citing}
}

{\theoremstyle{definition}
}

\title{Hofer--Zehnder capacity as a geodesic selector}

\author{Johanna Bimmermann}
\address{Mathematical Institute, University of Oxford, Andrew Wiles Building, Woodstock Road, Oxford OX2 6GG, United Kingdom}
\email{johanna.bimmermann@maths.ox.ac.uk}

\author{Beomjun Sohn}
\address{Chair for Geometry and Analysis, RWTH Aachen University, Pontdriesch 10-12, DE-52062 Aachen, Germany}
\email{bsohn95@gmail.com}

\let\origmaketitle\maketitle
\def\maketitle{
  \begingroup
  \def\uppercasenonmath##1{} 
  \let\MakeUppercase\relax 
  \origmaketitle
  \endgroup
}

\begin{document}
\thispagestyle{empty}

\begin{abstract}
We compute the Hofer--Zehnder capacity of the unit disk cotangent bundle of every ellipsoid in $\mathbb R^3$. The capacity is determined by the smaller of two distinguished quantities in the geodesic length spectrum: twice the systole and the length of the shortest simple closed geodesic of Morse index $3$. For the lower bound, we use Riemannian billiards on a suitable cut of the ellipsoid. For the upper bounds, we develop two complementary methods. The first combines an argument by Hofer--Viterbo with neck-stretching and yields, more generally, an upper bound for positively curved Riemannian two-spheres in terms of closed geodesics of prescribed index. The second uses the pair-of-pants product in symplectic homology and the Viterbo isomorphism to bound the Hofer--Zehnder capacity of any disk cotangent bundles of Riemannian two-spheres by twice the diastole; for positive curvature, the diastole agrees with the systole.
\end{abstract}

\maketitle

\setlength{\parindent}{1em}
\setcounter{tocdepth}{2}

\tableofcontents

\section{Introduction}

The Hofer--Zehnder capacity \cite{HZ90,HZ94} is a symplectic notion of ``size'' measured through Hamiltonian dynamics. Roughly speaking, it quantifies how much a Hamiltonian can oscillate before the existence of fast periodic orbits is guaranteed. More precisely, let $(M,\omega)$ be a symplectic manifold, possibly with boundary. The Hofer--Zehnder capacity is defined as
\[
c_{HZ}(M,\omega):=\sup\left\{\max H-\min H\mid H\in\mathcal H_{\mathrm{ad}}(M,\omega)\right\}.
\]
Here, a Hamiltonian $H\in C^\infty(M,\mathbb{R}_{\geq 0})$ is called admissible if it vanishes on a nonempty open subset of $M$, is equal to its maximum outside a compact subset of the interior of $M$, and has no nonconstant fast periodic orbits. A nonconstant $T$-periodic orbit of $H$ is a smooth map $\gamma\colon\mathbb{R}/T\mathbb{Z}\to M$ satisfying $\dot\gamma(t)=X_H(\gamma(t))\neq 0$. We say that $\gamma$ is fast if $T\leq 1$. The Hofer--Zehnder capacity is a symplectic capacity and as such it has the following properties:
\begin{itemize}
    \item (\emph{monotone}) if there exists a symplectic embedding $(M,\omega)\hookrightarrow(M',\omega')$ of codimension zero, then $c_{HZ}(M,\omega)\leq c_{HZ}(M',\omega')$;
    \item (\emph{conformal}) for every $a>0$, one has $c_{HZ}(M,a \cdot \omega)=a\cdot c_{HZ}(M,\omega)$
    \item (\emph{normalized}) $c_{HZ}(B^{2n}(1),\omega_0)=\pi=c_{HZ}(Z^{2n}(1),\omega_0)$
\end{itemize}
For the normalisation, $\omega_0:=\sum_{j=1}^n \dd x_j\wedge\dd y_j$ denotes the standard symplectic form on $\mathbb{R}^{2n}\cong \C^n$, $B^{2n}(1):=\lbrace z=x+iy \in\mathbb{C}^{n}\mid |z|<1\rbrace$ is the open unit ball, and $Z^{2n}(1):=\lbrace z=x+iy \in\mathbb{C}^{n}\mid |z_1|<1 \rbrace$  is the standard symplectic cylinder.

\medskip
\noindent\textbf{Computability.}
For bounded convex domains in \(U\subset\mathbb{R}^{2n}\), the Hofer--Zehnder capacity is computable and coincides with the minimal action of a Reeb orbit on the boundary $\partial U$ \cite{HZ90}, i.e.
  \[
  c_{HZ}(U,\omega_0)\;=\;\min\{\mathcal A(\gamma)\;:\; \gamma \text{ is a closed Reeb orbit on } \partial U\}.
  \]
However outside this class little is known. Unlike embedding capacities there is no universal upper bound (like volume). Indeed many (closed) symplectic manifolds have infinite Hofer--Zehnder capacity \cite{U12}. Usher uses the observation, that if a symplectic manifold admits a closed stable hypersurface without closed characteristics, then one can construct a Hamiltonian $H$ supported on a neighborhood of the hypersurface, without period orbits. For any constant $C>0$ the Hamiltonian $C\cdot H$ is admissible, hence $c_{HZ}=\infty$. Vice versa finiteness of the Hofer--Zehnder capacity implies that no such hypersurface exists, which is the Weinstein conjecture. It is therefore no surprise that finding upper bounds to the Hofer--Zehnder capacity requires hard machinery such as pseudoholomorphic curves or Floer theory.

\vspace{1em}
\noindent
\textbf{Codisk bundles.} Another very natural class of symplectic manifolds is given by fiberwise convex domains $K\subset T^*N$ in cotangent bundles, containing the zero-section. To any such domain one can associate a Finsler metric
\[
F:TN\to \R_{\geq 0}, \qquad F_x(v)=\sup_{p\in K} p(v).
\]
Then $K=D^*_FN=D_{F^*}N\subset T^*N$ is the unit disk bundle of the Legendre dual Finsler metric $F^*$, and $\tfrac12(F^*)^2$ generates the co-geodesic flow. If $K\subset T^*N$ is fiberwise an ellipsoid, then $F$ is induced by a Riemannian metric $g$. At present, there is no general result computing, or even bounding, the Hofer--Zehnder capacity of such co-disk bundles in a way that generalizes the minimal action formula for convex domains in $\R^{2n}$. Indeed, explicit values are known only in special examples, particularly when there is a canonical choice of metric. Since this paper is mostly concerned with Riemannian metrics, we use the metric isomorphism $TN\to T^*N$ induced by $g$ to pull back the canonical symplectic form on $T^*N$ to $TN$. We thus identify the unit tangent and cotangent disk bundles and write $D_gN=D_g^*N$. 

\vspace{1em}
\noindent
\textbf{Constant curvature.}
Despite this general lack of computability, explicit results are available when the geometry of the base manifold provides additional symmetry. In particular, if the base manifold is a symmetric space and $g$ is its invariant metric, then the Hofer--Zehnder capacity of the disk bundle can often be computed \cite{Bim23,Bim25,BM24}. We denote by $\mathrm{sys}(g)$ the systole of $g$, that is, the length of its shortest nonconstant closed geodesic. For surfaces of constant curvature, the following cases are known:
\begin{itemize}
    \item If $S^2$ is equipped with the round metric, then $c_{HZ}(DS^2,\dd\lambda)=\mathrm{sys}$.
    \item If $\rp^2$ is equipped with the round metric, then $c_{HZ}(D\rp^2,\dd\lambda)=2\cdot\mathrm{sys}$.
    \item If $T^2$ is equipped with a flat reversible Finsler metric, then $c_{HZ}(DT^2,\dd\lambda)=2\cdot\mathrm{sys}$.
\end{itemize}
We don't know how to extend this list to all closed surfaces of constant curvature. Even for flat Klein bottles $K$, the capacity is unknown in general, while for closed orientable hyperbolic surfaces even its finiteness remains open \cite{OP26}.

\vspace{1em}
\noindent
\textbf{Ellipsoids.}
As a first step beyond the constant-curvature setting, we study metrics on $S^2$ induced by ellipsoids. Ellipsoids form a natural class of deformations of the round sphere for which the geodesic flow remains completely integrable. This is particularly useful for our lower bounds, which rely on Riemannian billiards and therefore require a detailed understanding of the geodesic dynamics. By contrast, our upper bounds are more robust and apply in greater generality. More precisely, we consider the ellipsoid
\[
E(a,b,c):=\left\{(x,y,z)\in\mathbb{R}^3\;\middle|\;
\frac{x^2}{a^2}+\frac{y^2}{b^2}+\frac{z^2}{c^2}=1\right\},\ a\geq b\geq c > 0,
\]
equipped with the Riemannian metric induced by the Euclidean metric on $\mathbb{R}^3$. The complete integrability of the geodesic flow on $E(a,b,c)$ goes back to Jacobi \cite{Jac39} and allows us to describe the closed geodesics relevant to the computation of the Hofer--Zehnder capacity. This leads to our main result.

\begin{bigthm}\label{ellipsoid}
For every ellipsoid $E=E(a,b,c)$ with $a\geq b\geq c>0$, the Hofer--Zehnder capacity of its unit disk bundle is
\[
c_{HZ}(D_1E,\dd\lambda)=\min\left\{\ell(\eta),2\cdot\mathrm{sys}\right\},
\]
where $\mathrm{sys}$ is the length of the shortest closed geodesic and $\eta$ is a shortest closed geodesic of Morse index $3$. The systole is realized by the minor principal ellipse. The geodesic $\eta$ is either the major principal ellipse or, in the oblate case, a simple closed geodesic that winds once around the $z$-axis and intersects the major principal ellipse four times; see Fig.~\ref{dias}.
\end{bigthm}

\begin{figure}[h]
  \centering
  \includegraphics[width=0.7\textwidth]{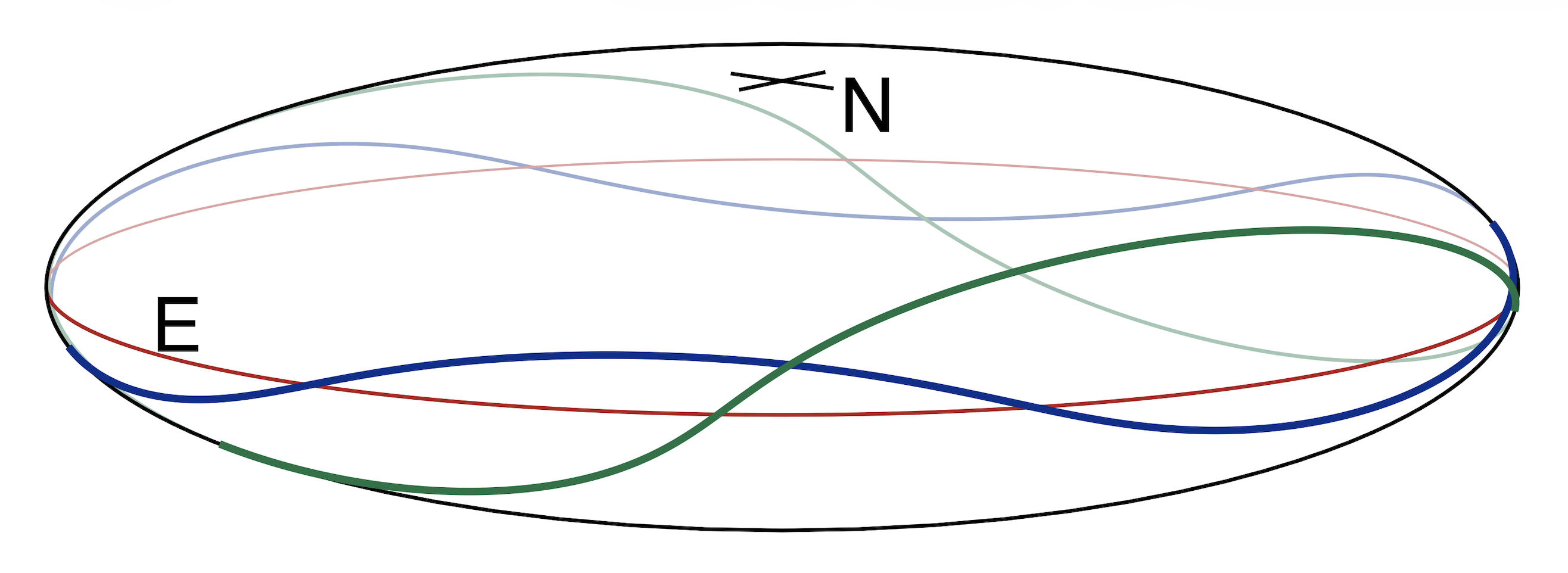} 
  \caption{The figure is taken from \cite{Kar25}. It shows an oblate ellipsoid of revolution $E(a,a,c)$ with two additional simple closed geodesics (green, blue). The green one has index 3 and the blue one has index 5, the mayor base ellipse (red) has odd index $\geq 7$ and is longer than both the green and the blue closed geodesic. For a general ellipsoid $E(a,b,c)$ and $c$ sufficiently small one gets similar additional closed geodesics.}
  \label{dias}
\end{figure}

If $a=b$ or $b=c$, then $E(a,b,c)$ is a surface of revolution, and the additional integral of the geodesic flow is simply the angular momentum about the axis of symmetry. In these cases, the Gromov width of the disk bundle was computed in \cite{FRV23} using embedded contact homology for toric domains, and the resulting value agrees with our computation of the Hofer--Zehnder capacity.\\

To illustrate which of the two quantities in Theorem~\ref{ellipsoid} determines the capacity, we numerically compared $\ell(\eta)$ with $2\mathrm{sys}$ over a large sample of ellipsoids, normalized by setting $a=1$. Figure~\ref{colormap} shows the region in the parameter space $1\geq b\geq c>0$ where the minimum is attained by $2\mathrm{sys}$ and the region where it is attained by $\ell(\eta)$. The values of $\ell(\eta)$ were computed using Charles F.\ F.\ Karney's program \texttt{Geod3Solve} \cite{Kar25,Kar25b}.

\begin{figure}[h]
  \centering
  \includegraphics[width=0.5\textwidth]{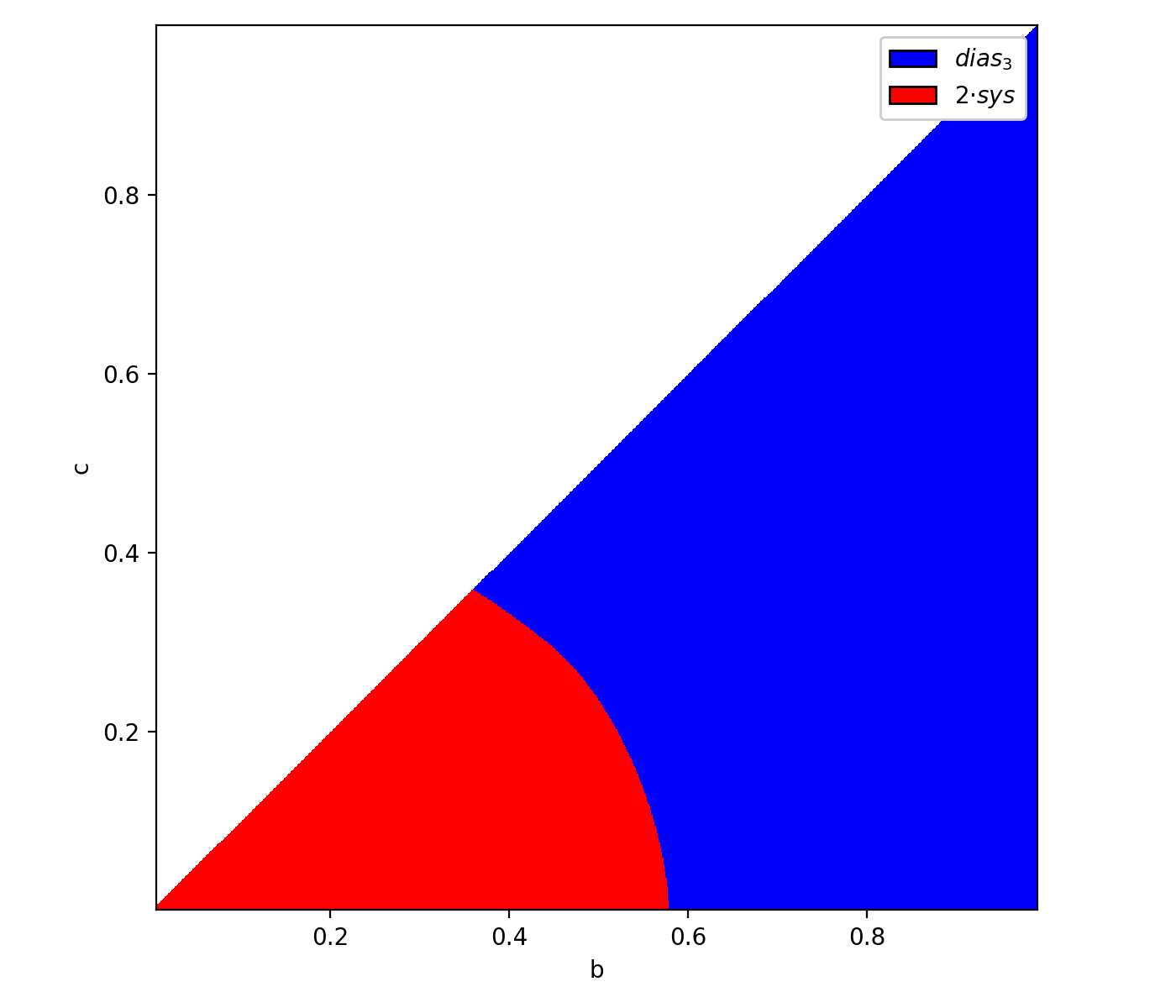} 
  \caption{Any ellipsoid can be scaled as $E(a,b,c)=a\cdot E(1,b/a,c/a)$, so that we may assume $a=1$. The plot shows for which values of $a=1\geq b\geq c$ twice the systole is shorter (red), and for which values $\ell(\eta)$ is shorter (blue). The value of $\mathrm{sys}$ is given by an elliptic integral, namely the perimeter of the minor principal ellipse $y^2/b^2+z^2/c^2=1$, and $\ell(\eta)$ is four times the Riemannian distance between $(a,0,0)$ and $(0,b,0)$, which can be computed using a program written by Karney \cite{Kar25,Kar25b}.}
  \label{colormap}
\end{figure}

\vspace{1em}
\noindent
\textbf{Structure of the proof.}
The proof proceeds in three steps, consisting of one lower bound and two complementary upper bounds. In Step~1, we establish the lower bound using Riemannian billiards. In Step~2, we prove the upper bound $c_{HZ}(DE,\dd\lambda)\leq\ell(\eta)$. To do so, we symplectically embed $DE$ into $\cp^1\times\cp^1$ and combine diagonal neck-stretching with a Hofer--Viterbo-type argument \cite{HV92} that derives capacity bounds from pseudoholomorphic curves. This can be viewed as a first step toward extending the methods of \cite{HV92,Lu06} to the noncompact setting of finite-energy curves. Finally, in Step~3, we use symplectic homology to prove the upper bound $c_{HZ}(DE,\dd\lambda)\leq2\mathrm{sys}$. Each of the three steps extends beyond the setting of ellipsoids and yields a result of independent interest. We therefore briefly outline the three arguments and state the corresponding results in their more general form.

\vspace{1em}
\noindent
\textbf{Billiards, Step 1.}
Our lower bound is based on Riemannian billiards and applies more generally to unit disk tangent bundles of Riemannian manifolds. Let $(N,g)$ be a Riemannian manifold and let $\Omega\subset N$ be an open domain with compact closure and smooth boundary. The main idea is to approximate the billiard dynamics on $(\overline{\Omega},g)$ by smooth Hamiltonian systems whose potentials diverge near $\partial\Omega$. If $(\overline{\Omega},g)$ has no short periodic billiard trajectories, the approximating Hamiltonian systems have no short periodic orbits. They therefore give rise to admissible Hamiltonians and hence to a lower bound for the Hofer--Zehnder capacity. The resulting estimate is expressed in terms of the shortest admissible periodic billiard trajectory.

\begin{theorem}\label{thmbilliard}
Let $(N,g)$ be a Riemannian manifold, and let $\Omega\subset N$ be an open domain with compact closure and smooth boundary. Then
\[
c_{HZ}(D_gN,\dd\lambda)\geq c_{HZ}(D_g\Omega,\dd\lambda)\geq\Omega\text{-}\mathrm{sys}(g),
\]
where
\[
\Omega\text{-}\mathrm{sys}(g):=\inf\left\{\ell_g(\gamma)\mid \gamma\text{ is an admissible periodic billiard trajectory of }(\overline{\Omega},g)\right\}.
\]
\end{theorem}
\noindent
The precise notion of an \emph{admissible periodic billiard trajectory} will be introduced in Section~\ref{secbilliard}.\\

For the proof of our main theorem, we apply this theorem to a billiard table obtained by cutting the ellipsoid along a geodesic segment. More precisely, let $\gamma_y=\lbrace y=0\rbrace$ be the median principal ellipse and let $L\subset\gamma_y$ be the segment joining the two neighboring umbilical points $U_+$ and $U_-$ that contains the north pole. We consider the billiard table $\Omega:=E(a,b,c)\setminus L$. Since $\Omega$ has two nonsmooth boundary points, we approximate it by smooth billiard tables and use a generalized reflection law at $U_+$ and $U_-$. By classifying the resulting admissible periodic billiard trajectories and estimating their lengths, we obtain the lower bound in our main theorem.

\begin{proposition}\label{proplowerbound}
Let $E=E(a,b,c)$ with $a\geq b\geq c>0$, and let $\eta$ be a shortest closed geodesic on $E$ of Morse index $3$. Then\
$$
c_{HZ}(DE,\dd\lambda)\geq\min\lbrace\ell(\eta),2\mathrm{sys}\rbrace.
$$
\end{proposition}

Indeed, for the billiard table $\Omega=E\setminus L$, the analysis of admissible periodic billiard trajectories gives $\Omega\text{-}\mathrm{sys}(g)=\min \lbrace(\ell(\eta),2\mathrm{sys}\rbrace$. Proposition \ref{proplowerbound} therefore follows from Theorem \ref{thmbilliard}.

\vspace{1em}
\noindent
\textbf{Pseudoholomorphic curves, Step 2.}
The Hofer--Viterbo argument \cite{HV92}, later generalized by Lu \cite{Lu06}, bounds the Hofer--Zehnder capacity of a uniruled symplectic manifold in terms of the energy of pseudoholomorphic spheres. To obtain an analogous bound for codisk bundles of $S^2$, we embed the codisk bundle into $\cp^1\times\cp^1$ and combine this argument with neck-stretching. The resulting finite-energy planes are asymptotic to closed Reeb orbits, whose action agrees in our setting with the length of the corresponding closed geodesics on $S^2$. This yields an upper bound given by the length of a distinguished closed geodesic, characterized by its index and an additional action bound.\\

To state the general theorem, let us first fix some notation. Let $\Lambda S^2:=W^{1,2}(S^1,S^2)$ denote the free loop space, and let $\sE_g:\Lambda S^2\to\mathbb{R}$ be the energy functional. For a positively curved Riemannian metric $g$, let $R(S^2,g)$ denote the circumradius, namely, the radius of the smallest  Euclidean ball containing the convex body in $\mathbb{R}^3$ whose boundary is isometric to $(S^2,g)$. We denote by $\mathrm{sys}(g)$ the systole, i.e., the length of its shortest closed geodesic.

\begin{theorem}
\label{Theorem: Upper bound from index 3 - refined + geodesic flow version}
Let $g$ be a Riemannian metric of positive curvature on $S^2$. Then
\[
c_{HZ}(D_g^*S^2,\dd\lambda)
\leq
\sup_{\gamma\in\mathcal{P}_{HZ}(g)}\ell_g(\gamma),
\]
where $\mathcal{P}_{HZ}(g)$ is the set of closed geodesics of Morse index $3$ with respect to $\sE_g$ that satisfy one of the following conditions:
\begin{itemize}
    \item $\gamma$ has an even number of self-intersections and $\ell_g(\gamma)\leq2\pi R(S^2,g)$;
    \item $\gamma$ has an odd number of self-intersections and $\ell_g(\gamma)+\mathrm{sys}(g)\leq2\pi R(S^2,g)$.
\end{itemize}
\end{theorem}

For ellipsoids, there is precisely one ($S^1$-family of) closed geodesics of index 3, satisfying the conditions in Theorem~\ref{Theorem: Upper bound from index 3 - refined + geodesic flow version}, all members of the family have the same length. This gives the first upper bound in our main theorem.

\begin{corollary}\label{cor:ellipsoid-index-three-upper-bound}
Let $E=E(a,b,c)$ with $a\geq b\geq c>0$, and let $\eta$ be a shortest closed geodesic on $E$ of Morse index $3$. Then
\[
c_{HZ}(D_1E,\dd\lambda)\leq\ell(\eta).
\]
\end{corollary}

\vspace{1em}
\noindent
\textbf{Symplectic homology, Step 3.}
Our second upper bound is obtained using the pair-of-pants product in symplectic homology. Under the Viterbo isomorphism, this product corresponds to the Chas--Sullivan loop product on the free loop space. We construct suitable homology classes from sweepouts of $S^2$ and show that their loop product represents the fundamental class of the constant loops. Keeping track of the length filtration then gives an upper bound for the Hofer--Zehnder capacity in terms of twice the maximal length of the sweepout. Minimizing over all such sweepouts leads to the diastole, a min--max invariant of the length functional on the free loop space. More precisely, for a Riemannian metric $g$ on $S^2$, define
\[
\mathrm{dias}(g):=\inf_{\Gamma}\sup_{t\in[-1,1]}\ell_g(\Gamma(t)),
\]
where $\Gamma:([-1,1],\{\pm1\})\to(\Lambda S^2,S^2)$ ranges over families of loops representing a generator of $\pi_1(\Lambda S^2,S^2)\simeq\pi_2(S^2)$. Here, $S^2\subset\Lambda S^2$ denotes the subspace of constant loops.

\begin{theorem}\label{thmSH}
Let $g$ be any Riemannian metric on $S^2$. Then
\[
c_{HZ}(D_g^*S^2,\dd\lambda)\leq2\,\mathrm{dias}(g).
\]
\end{theorem}

For positively curved Riemannian metrics on $S^2$, Abbondandolo--Mazzucchelli showed that $\mathrm{dias}(g)=\mathrm{sys}(g)$ \cite[App.~A]{BK22}. Theorem~\ref{thmSH} therefore gives $c_{HZ}(D_g^*S^2,\dd\lambda)\leq2\,\mathrm{sys}(g)$ and, in particular, yields the second upper bound in our main theorem.

\vspace{1em}
\noindent
\textbf{Acknowledgements.}
We thank Alberto Abbondandolo, Peter Albers, Filip Broćić, Dylan Cant, Charles Karney, Alexandru Oancea, Alexander Ritter, Felix Schlenk and Simon Vialaret for helpful discussions and valuable input. We are also grateful to the organizers of the workshop \emph{Billiards and Quantitative Symplectic Geometry}, held at Heidelberg University, for their hospitality and for providing a stimulating environment. The first author was supported by the Engineering and Physical Sciences Research Council [grant number EP/Z535977/1]. The second author was supported by Basic Science Research Program through the National Research Foundation of Korea(NRF) funded by the Ministry of Education(RS-2025-02317642).

The authors used ChatGPT and Claude during the preparation of this article for proofreading, language refinement, and mathematical exploration. The authors takes full responsibility for all mathematical statements, proofs, and conclusions in this article.

\section{Step 1: Billiards}\label{secbilliard}

Let $\Omega\subset N$ be an open subset with compact closure
$\bar\Omega$ and smooth boundary $\partial\Omega$. Consider Riemannian billiards on
$(\bar\Omega,g)$. Ideally, one would like a statement of the form
\begin{equation}\label{ideal lower bound}
    c_{HZ}(D_gN,\dd\lambda)\geq
    \inf\left\{\ell_g(\gamma)\mid
    \gamma \text{ is a periodic billiard trajectory of }(\bar\Omega,g)\right\}.
\end{equation}
The crux lies in defining a suitable notion of periodic billiard trajectory for which such a lower bound can be proved. The following definition is adapted to the approximation schemes developed in \cite{BG89,AM11,V21}.

\begin{definition}\label{billiard}
A $\tau$-periodic billiard trajectory of $(\bar\Omega,g=\langle\cdot,\cdot\rangle)$
is an element $\Gamma\in H^1(S^1,\bar\Omega)$ such that there exists a finite Borel
measure $\mu$ on
\[
\mathcal C:=\{t\in \R/\tau\Z\mid \gamma(t)\in\partial\Omega\},
\qquad \gamma(t):=\Gamma(t/\tau),
\]
with the following properties:
\begin{enumerate}
    \item for every $\psi\in H^1(S^1,\gamma^*T\bar\Omega)$,
    \[
    \int_0^\tau \langle \dot\gamma,\nabla_{\dot\gamma}\psi\rangle\,\dd t
    =
    \int_{\mathcal C}\langle \nu(\gamma),\psi\rangle\dd\mu(t),
    \]
    where $\nu$ denotes the outer unit normal along $\partial\Omega$;

    \item the curve $\gamma$ is a smooth geodesic on
    $(\R/\tau\Z)\setminus \supp(\mu)$ and has constant energy
    \[
    E(\gamma):=\frac12|\dot\gamma|_g^2;
    \]

    \item $\gamma$ has left and right derivatives $\dot\gamma_\pm$, which are
    left- and right-continuous on $\R/\tau\Z$, respectively. Moreover, at each
    time $t\in\mathcal C$ which is an isolated point of $\supp(\mu)$, the curve
    satisfies the law of reflection
    \[
    \langle \dot\gamma_+,\nu\rangle=-\langle \dot\gamma_-,\nu\rangle\neq 0,\qquad \dot\gamma_+-\langle \dot\gamma_+,\nu\rangle\nu
    =
    \dot\gamma_--\langle \dot\gamma_-,\nu\rangle\nu,
    \]
    where $\nu\in T_{\gamma(t)}N$ is the inner unit normal to $\partial\Omega$
    at $\gamma(t)$.
\end{enumerate}
\end{definition}

\begin{remark}
This notion of billiard trajectory is intentionally weak. It includes not only the
usual piecewise geodesic trajectories with isolated reflections, but also more
degenerate boundary behaviour. The set $\mathcal C$ records all boundary contact times, whereas $\supp(\mu)\subset \mathcal C$ records
those times at which a nontrivial normal reaction occurs. If $\supp(\mu)$ is finite, we call $\gamma$ a \emph{bounce orbit}. In this case,
every time $t\in \supp(\mu)$ is an isolated bounce time and satisfies the usual
(non-glancing) law of reflection. Nevertheless, $\mathcal C$ may contain additional
isolated times in $\mathcal C\setminus \supp(\mu)$ at which $\gamma$ meets $\partial\Omega$
tangentially; we refer to such behaviour as \emph{glancing}. Finally, $\mathcal C$ may contain nontrivial intervals, corresponding to trajectories
that move along the boundary for a positive amount of time, or are even entirely
contained in the boundary. We refer to such behaviour as \emph{gliding} along $\partial\Omega$. If $\gamma$ is smooth on an interval contained in $\partial\Omega$, then $\gamma$ is a constant-speed geodesic of $\partial\Omega$ with respect to the induced metric.
\end{remark}

The idea of proof for the lower bound \ref{ideal lower bound} is using the approximation of billiards by Hamiltonian dynamics introduced by \cite{BG89, AM11, V21}.

\begin{proof}
Consider an approximation scheme as introduced by Vocke \cite{V21}, namely
\[
H_\varepsilon:=E+\varepsilon V:T\bar\Omega\to\R,
\]
where $V:\bar\Omega\to\R$ is constant outside a neighbourhood of $\partial\Omega$ and tends to infinity as one approaches $\partial\Omega$. For the precise construction of $V$, see \cite[Sec.~3.2.1]{V21}. Let $C>0$ be strictly smaller than the minimal length of a periodic billiard trajectory of $(\bar\Omega,g)$. Suppose that there exists a sequence $(\Gamma_\varepsilon,\tau_\varepsilon)$ of periodic orbits for $H_\varepsilon$ of energy $0<\delta<H_\varepsilon(\Gamma_\varepsilon)\leq 1/2$ with $\tau_\varepsilon\leq C$. Then the convergence result for approximate solutions, \cite[Prop.~3.6]{V21}, yields a non-constant $\tau$-periodic billiard
trajectory of period $\tau\leq C$ and speed $\leq 1$, a contradiction. Hence, for $\varepsilon>0$ sufficiently small, every periodic orbit of $H_\varepsilon$ has period larger than $C$. Now define the sublevel set $U_\varepsilon:=\{H_\varepsilon<1/2\}\subset D_g\bar\Omega\subset D_gN.$ The Hamiltonian $K_{C,\varepsilon}:=\left.C\sqrt{2H_\varepsilon}\right|_{U_\varepsilon}$ attains its maximum value $C$ on $\partial U_\varepsilon$ and its minimum $0$ on
the zero section. Although $K_{C,\varepsilon}$ is only continuous near the minimum, one can choose, for every $\delta>0$, a function $f:\lbrack 0,C\rbrack\to\lbrack 0,C-\delta\rbrack$ such that $f\circ K_{C,\varepsilon}$ is admissible; see for instance \cite{Bim23}. Therefore,
\[
C-\delta\leq c_{HZ}(U_\varepsilon,\dd\lambda)\leq c_{HZ}(D_gN,\dd\lambda)
\qquad \forall\,\varepsilon,\delta>0.
\]
Letting $\delta,\varepsilon\to 0$ and choosing $C=C(\varepsilon)$ approaching the minimal length of a periodic billiard trajectory, we obtain the claim.
\end{proof}

A closer inspection of the proof shows that the above lower bound can be improved.
Indeed, if a periodic billiard trajectory is not the limit of a sequence of
approximate solutions, then the argument does not detect it. This motivates the
following notion of admissibility.

\begin{definition}
A billiard orbit $\gamma$ is called \emph{admissible} if there exists a sequence of
periodic orbits $(\Gamma_\varepsilon,T_\varepsilon)\in H^1(S^1,\bar\Omega)\times\R_{>0}$ for the approximation Hamiltonians $H_\varepsilon$ converging, as $\varepsilon\to 0$, to $(\Gamma,T)$ such that $\gamma(t)=\Gamma(t/T)$.
\end{definition}
Up to this point we have only considered domains $\Omega$ with smooth boundary $\partial\Omega$. However, later we will need to treat billiard tables whose boundary is not smooth. In that situation, we approximate $\Omega$ by a family of open subsets $\Omega_\delta\subset \Omega$ with smooth boundary such that
\[
\bigcup_{\delta>0}\Omega_\delta=\Omega.
\]
We then call a billiard orbit of $\Omega$ \emph{admissible} if it arises as a limit, as $\delta\to 0$, of admissible billiard orbits in $\Omega_\delta$. Note that for a domain with non-smooth boundary, the reflection law in Definition~\ref{billiard} is in general not well defined at singular boundary points, since there need not be a unique outer normal vector. In the specific situation considered here, we will address this issue by requiring the normal vector to belong to a suitable normal cone, which we will define explicitly (Sec.\ \ref{altreflection}).
\begin{remark}
The notion of admissibility may depend on the chosen approximation scheme. Throughout this paper, for domains with smooth boundary we use the approximation scheme of \cite{V21}. For domains with non-smooth boundary, we fix in each case a specific approximating family of smooth domains, which will be described explicitly when needed. At present, we do not know whether, or in what sense, the resulting notion of admissibility is independent of these choices.
\end{remark}
It seems plausible that all bounce orbits are admissible and that non-admissibility
can occur only for certain gliding trajectories, but the authors are unaware of a
general result to this effect. A better understanding of non-admissible billiard
orbits would be very useful. At present, we can only exclude such trajectories by
hand. For example, the lower bound $c_{HZ}(D\rp^n,\dd\lambda)=2\cdot\mathrm{sys}$ from \cite{Bim23} is obtained by considering billiards on a hemisphere, viewed as
the complement of a totally geodesic copy of $\rp^{n-1}$. Thus the boundary of the
billiard table in $\rp^n$ is $\rp^{n-1}$, and it contains many gliding orbits of
length $\mathrm{sys}$, but none of them is admissible. A similar argument was used in
\cite{BM24} to show that non-contractible geodesics on lens spaces $L(p,1)$ for $p\geq 3$ odd, are invisible to the Hofer--Zehnder capacity. More precisely we computed $c_{HZ}(DL(p,1),\dd\lambda)=p\cdot\mathrm{sys}$.\\

The above proof immediately implies the following theorem.

\begin{theorem}[Thm.~\ref{thmbilliard}]
Let $\Omega\subset N$ be an open domain with compact closure and smooth boundary
$\partial\Omega$. Then
\[
c_{HZ}(D_gN,\dd\lambda)\geq c_{HZ}(D_g\Omega,\dd\lambda)\geq \Omega\text{-}\mathrm{sys}(g),
\]
where
\[
\Omega\text{-}\mathrm{sys}(g):=
\inf\left\{\ell_g(\gamma)\mid
\gamma \text{ is an admissible periodic billiard trajectory of }(\bar\Omega,g)\right\}.
\]
\end{theorem}

\vspace{1em}

\subsection{Geodesic flow on ellipsoids}

To prove the lower bound in Theorem~\ref{ellipsoid} using Theorem~\ref{thmbilliard}, we need to choose a suitable billiard table and analyze its admissible periodic trajectories. We therefore begin by recalling the relevant features of the geodesic flow on ellipsoids, following Klingenberg \cite{Kl95} and Moser \cite{Mos80}. We use Karney's notation and refer to \cite{Kar25b} for a complementary overview. We further recommend the reader to have a look at Karney's Wikipedia article \href{https://en.wikipedia.org/wiki/Geodesics_on_an_ellipsoid}{\emph{Geodesics on an ellipsoid}}.
\\

\noindent
For three real numbers $ a \geq b \geq c >0$ define the ellipsoid
$$
E = E(a,b,c) := \left\{ \frac{x^2}{a^2} + \frac{y^2}{b^2} + \frac{z^2}{c^2} = 1 \right\}\subset \R^3
$$
and equip it with the metric induced by the Euclidean metric on $\mathbb{R}^3$. 
The intersections with the three coordinate planes yield three simple closed geodesics, called \emph{principal ellipses}. 
We denote them by $\gamma_i \equiv E(a,b,c) \cap \{ i = 0 \}$ for $i = x, y, z$.\\

\noindent
\textbf{Ellipsoidal coordinates.}
Assume for the moment that $a>b>c>0$, and define
\[
k:=\frac{\sqrt{b^2-c^2}}{\sqrt{a^2-c^2}},
\qquad
k':=\frac{\sqrt{a^2-b^2}}{\sqrt{a^2-c^2}},
\qquad
k^2+k'^2=1.
\]
Following Karney \cite{Kar25b}, we parametrize the ellipsoid by the ellipsoidal latitude $\beta$ and longitude $\omega$:
\[
\begin{pmatrix}
x(\beta,\omega)\\
y(\beta,\omega)\\
z(\beta,\omega)
\end{pmatrix}
=
\begin{pmatrix}
\displaystyle a\cos\omega\sqrt{k^2\cos^2\beta+k'^2}\\[0.6ex]
\displaystyle b\cos\beta\sin\omega\\[0.6ex]
\displaystyle c\sin\beta\sqrt{k^2+k'^2\sin^2\omega}
\end{pmatrix}.
\]
The curves on which either $\beta$ or $\omega$ is constant are lines of curvature. The coordinate system is singular at the four umbilics, characterized by $\cos\beta=\sin\omega=0$. The parameter torus $(\omega,\beta)\in[-\pi,\pi]\times[-\pi,\pi]$ covers the ellipsoid twice. Equivalently, one may use $[-\pi,\pi]\times[-\tfrac{\pi}{2},\tfrac{\pi}{2}]$ as a principal sheet, with branch cuts along $\beta=\pm\tfrac{\pi}{2}$. Figure~\ref{coordinates} illustrates the resulting coordinate grid.

\begin{figure}[ht]
  \centering
  \includegraphics[width=0.5\textwidth]{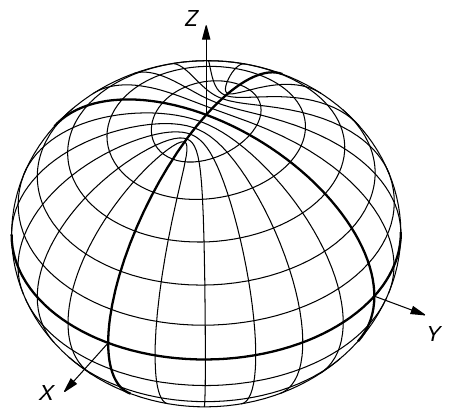}
  \caption{The ellipsoidal coordinate grid, from \cite{Kar25b}. The thick curves are the principal ellipses: the minor principal ellipse $\gamma_x$, given by $x=0$ or $\cos\omega=0$; the median principal ellipse $\gamma_y$, given by $y=0$ or $\cos\beta\sin\omega=0$; and the major principal ellipse $\gamma_z$, given by $z=0$ or $\sin\beta=0$.}
  \label{coordinates}
\end{figure}

\noindent
\textbf{Integrability.}
The geodesic flow on the triaxial ellipsoid $E=E(a,b,c)$ is Liouville integrable. Let $e_\beta(p)$ and $e_\omega(p)$ denote the unit tangent vectors in the directions of increasing $\beta$ and $\omega$, respectively. Thus, $e_\beta$ is tangent to the coordinate curve $\omega=\mathrm{const}$, while $e_\omega$ is tangent to the coordinate curve $\beta=\mathrm{const}$. Following Karney \cite{Kar25b}, for a unit vector $v\in T_pE$ we define its azimuth $\alpha\in\mathbb{R}/2\pi\mathbb{Z}$ by
\[
v=e_\omega(p)\sin\alpha+e_\beta(p)\cos\alpha.
\]
Equivalently, $\sin\alpha=\langle v,e_\omega(p)\rangle$ and $\cos\alpha=\langle v,e_\beta(p)\rangle$. On the unit tangent bundle $T^1E$, an additional first integral is given by
\[
F(p,v)
=
k^2\cos^2\beta\sin^2\alpha
-
k'^2\sin^2\omega\cos^2\alpha,
\qquad
F:T^1E\longrightarrow[-k'^2,k^2].
\]
Equivalently, its homogeneous quadratic extension to $TE$ is
\[
\widetilde F(p,v)
=
k^2\cos^2\beta\,\langle v,e_\omega(p)\rangle^2
-
k'^2\sin^2\omega\,\langle v,e_\beta(p)\rangle^2.
\]
The functions $\tfrac{1}{2}| v|^2$ and $\widetilde F$ Poisson commute and are functionally independent on an open dense subset of $TE$. Hence the geodesic flow is Liouville integrable.\\
\ \\
\noindent
\noindent
\textbf{Invariant tori.}
We now describe the regular invariant sets of the geodesic flow \cite[Thm.~3.5.7]{Kl95}. Fix the unit energy level $\tfrac12|v|_g^2=\tfrac12$ and let $f$ be a regular value of $F$. For both $f\in(-k'^2,0)$ and $f\in(0,k^2)$, the joint level set $\{\tfrac12|v|_g^2=\tfrac12,\ F=f\}$ is the disjoint union of two embedded invariant $2$-tori.

Let $N:=(0,0,c)$ be the north pole, and let $U_+$ and $U_-$ be the two umbilical points adjacent to $N$ along the median principal ellipse $\gamma_y=E(a,b,c)\cap\{y=0\}$. We denote by $L\subset\gamma_y$ the closed geodesic segment with endpoints $U_+$ and $U_-$ whose interior contains $N$. In particular, $L$ intersects the minor principal ellipse $\gamma_x=E(a,b,c)\cap\{x=0\}$ precisely at $N$.

For $f\in(-k'^2,0)$, the geodesics wind monotonically around the $x$-axis, while the longitude $\omega$ oscillates between two turning lines characterized by $\sin^2\omega=-f/k'^2$. The two invariant tori correspond to the two possible orientations of the winding around the $x$-axis. As $f\to-k'^2$, they degenerate to the two orientations of $\gamma_x$. Following Karney \cite{Kar25b}, we call these geodesics \emph{transpolar}. Every transpolar geodesic intersects the interior of $L$.

For $f\in(0,k^2)$, the geodesics wind monotonically around the $z$-axis, while the latitude $\beta$ oscillates between the turning lines $\beta=\pm\beta_f$, where $\cos^2\beta_f=f/k^2$. Again, the two invariant tori correspond to the two possible orientations of the winding. As $f\to k^2$, they degenerate to the two orientations of the major principal ellipse $\gamma_z$. We call these geodesics \emph{circumpolar}. No circumpolar geodesic intersects the interior of $L$.

Finally, the singular level $F=0$ consists of the \emph{umbilical} geodesics. Every such geodesic passes repeatedly through a pair of opposite umbilical points \cite[Thm.~3.5.16]{Kl95}. The median principal ellipse $\gamma_y$ lies in this level and, up to orientation and iteration, is its only closed geodesic.\\

\noindent
\textbf{Closed geodesics.}
Closed geodesics other than the principal ellipses can be described using the period map introduced in \cite[Def.~3.5.9]{Kl95}. We express this map in terms of Karney's first integral and write it as $f\mapsto\upomega(f)$, where $F=f$.\footnote{Klingenberg denotes both the period map and its parameter differently. We use Karney's first integral $F$ as the parameter and write $\upomega$ for the period map to distinguish it from the ellipsoidal longitude $\omega$.} We recall the precise definition and the relation between the two conventions in Appendix~\ref{appendix: period map}. Roughly speaking, $\upomega(f)$ measures the ratio between the winding and oscillation frequencies of a geodesic on the invariant tori with $F=f$. In suitable coordinates, the geodesic flow on each such torus is conjugate to a linear flow of slope $\upomega(f)$ on the flat square torus \cite[Thm.~3.5.10]{Kl95}. A geodesic is closed precisely when $\upomega(f)$ is rational and simple precisely when $\upomega(f)$ is an integer. The period map extends continuously across $f=0$ and is strictly increasing, with $\upomega(0)=1$. Thus $\upomega(f)>1$ for circumpolar geodesics and $\upomega(f)<1$ for transpolar geodesics. It follows that nonprincipal simple closed geodesics occur on the circumpolar side $f>0$. More precisely, such geodesics exist only if $\lim_{f\to k^2}\upomega(f)>2$. The first such family is characterized by $\upomega(f)=2$ and consists of simple closed geodesics winding once around the $z$-axis and intersecting the major principal ellipse four times. Since the limiting value of the period map is bounded above by $a/c$, an ellipsoid with $a/c<2$ has no simple closed geodesics other than its principal ellipses. For fixed pinching $a/c$, this limiting value, and hence the number of possible nonprincipal simple closed geodesics, is equal to $a/c$ in the oblate limit $a=b$ and smallest in the prolate limit $b=c$.\\

\noindent
\textbf{Morse indices.}
On a triaxial ellipsoid $E(a,b,c)$ with $a>b>c>0$, the minor and major principal ellipses $\gamma_x$ and $\gamma_z$ are stable closed geodesics, whereas the median principal ellipse $\gamma_y$ is unstable \cite[Thm.~3.5.8]{Kl95}. The minor principal ellipse $\gamma_x$, which realizes the systole, has Morse index $1$, while $\gamma_y$ has Morse index $2$. The Morse index of the major principal ellipse $\gamma_z$ is odd and at least $3$ \cite[Lem.~3.5.17]{Kl95}.\\

As the ellipsoid is deformed, the index of $\gamma_z$ increases by two whenever a new family of simple circumpolar closed geodesics bifurcates from $\gamma_z$. In terms of Karney's first integral, these bifurcations occur when the limiting value $\lim_{f\to k^2}\upomega(f)$ passes through an integer.

\begin{lemma}\label{index closed geodesics}
The minor principal ellipse $\gamma_x$ is the only simple closed geodesic of Morse index $1$. If $\lim_{f\to k^2}\upomega(f)<2$, then the major principal ellipse $\gamma_z$ has Morse index $3$. If $\lim_{f\to k^2}\upomega(f)>2$, then there is an $S^1$-family of simple circumpolar closed geodesics characterized by $\upomega(f)=2$. Each member of this family winds once around the $z$-axis, intersects $\gamma_z$ four times, and has Morse index $3$ and nullity $1$. We denote by $\eta$ one such geodesic in either case. Moreover, the length of such a geodesic satisfies $\ell(\eta) \le \ell(\gamma_z)$.
\end{lemma}

This lemma follows from Klingenberg's description of the geodesic flow on ellipsoids together with the index computation recalled in Appendix~\ref{appendix: morse indices}.  The required length estimates for the latter case, when $\eta\neq\gamma_z$, are established in Appendix~\ref{appendix: period map}.

\subsection{Proof of Proposition~\ref{proplowerbound}} Denote by $L\subset \gamma_y$ the geodesic segment joining the two neighboring umbilical points $U_+$ and $U_{-}$ containing the north pole (see Figure \ref{Omega}). Our billiard table will be $\Omega:=E(a,b,c)\setminus L$. Note that $\Omega$ has no smooth boundary, we therefore need to generalize our notion of billiard trajectory in Definition \ref{billiard}, by clarifying what the reflection law at the tips $U_{+}$ and $U_{-}$ of $L$ should be\footnote{The easiest way would be to require no reflection law and allow any in- and out-going angle, but we need a slightly stronger generalized reflection law to prove our lower bound.}.

\subsubsection{Generalized reflection law}\label{altreflection} We say that a billiard trajectory in $\Omega$ satisfies the generalized reflection law at the tips if the conditions of Definition \ref{billiard} hold with (3) replaced by:

\begin{itemize}
  \item[(3)'] At each time $t\in\mathcal C$ which is an isolated point of $\mathrm{supp}(\mu)$ and for which $\gamma(t)=U_{\pm}$, the curve $\gamma$ admits left and right derivatives $\dot\gamma_{\pm}$ that are left- and right-continuous on $\R/T\Z$, respectively, and obey the reflection law
  \[
  \langle \dot\gamma_+,\nu\rangle=-\langle \dot\gamma_-,\nu\rangle \neq 0,\qquad
  \dot\gamma_+-\langle \dot\gamma_+,\nu\rangle\nu=\dot\gamma_--\langle \dot\gamma_-,\nu\rangle\nu
  \]
  for some unit vector $\nu\in H_{\mathrm{in}}\cap H_L$. Here $H_{\mathrm{in}}:=\{v\in T_{U_{\pm}}E\mid v\cdot \dot\gamma_+\le 0\}$ and $H_L:=\{v\in T_{U_{\pm}}E\mid v\cdot l\ge 0\}$, where $l\in T_{U_{\pm}}E$ is the unit tangent of $\mathrm{int}\,L$ pushed to the tip $U_{\pm}$ and pointing out of $L$. The half-spaces $H_{\mathrm{in}}$, $H_L$, and the cone $H_{\mathrm{in}}\cap H_L$ are shown in Figure \ref{reflection}.
\end{itemize}

\begin{figure}[h]
  \centering
  \includegraphics[width=0.6\textwidth]{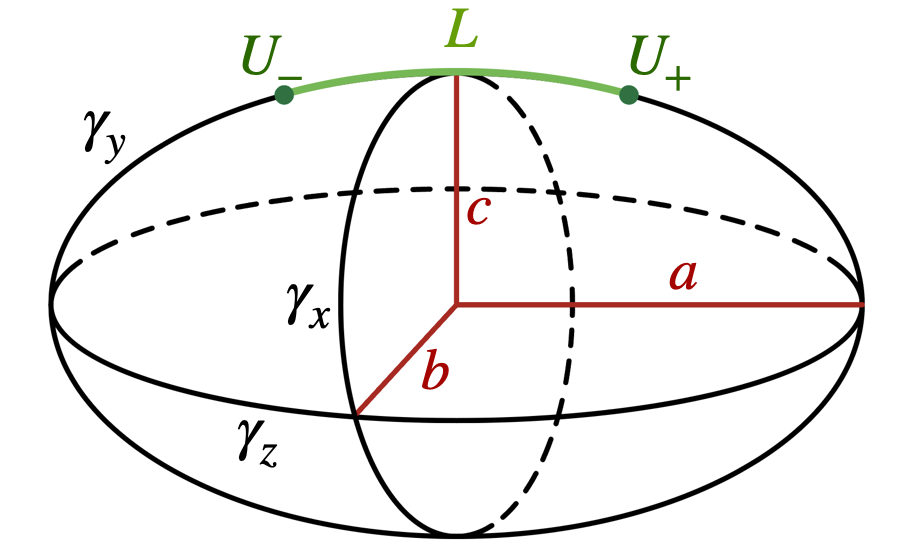} 
  \caption{A sketch of the ellipsoid $E(a,b,c)$ with $a\geq b\geq c\geq 0$ and the geodesic segment $L$ connecting the umbilical points $U_{\pm}$. }
  \label{Omega}
\end{figure}

\begin{proof} The idea is to approximate $\Omega$ by a sequence of smooth billiard tables $\Omega_\delta\subset \Omega$ with smooth boundary. We then show that if a sequence of periodic $\Omega_\delta$-billiard trajectories converges, the limit is a generalized billiard trajectory of $\Omega$, hence satisfies (3)'.\\

\noindent
\textit{Choice of $\Omega_\delta$.} We define a $\delta$-thickening of $L$ as the interior of the $\cos(\beta)=\delta$ -parameter line (see Figure \ref{1bounce}). 
Now its complement $\Omega_\delta$ has smooth boundary 
$\partial \Omega_\delta = \lbrace \cos(\beta)= \delta\rbrace$
hence the approximation scheme in \cite{V21,AM11} applies. Sending $\delta \to 0$ recovers the original billiard table $\Omega = E(a,b,c) \setminus L.$ \\

\noindent
\textit{Compactness.} For each $\delta > 0,$ let $(\Gamma_\delta, T_\delta)$ be a periodic bounce orbit\footnote{Note that $\Omega_\delta$ is the complement of a strongly geodesically convex set, hence all billiards trajectories are bounce orbits.} of $\Omega_\delta$ with energy $E_\delta \le 1$ and period $T_\delta < C$, for some constant $C>0$ Clearly, $\nabla_{\dot\Gamma_\delta}\dot \Gamma_\delta$ is uniformly bounded in $L^1$, hence $\Gamma_\delta$ is uniformly bounded in $W^{2,1}$. The compact embedding
$$
W^{2,1}(S^1, E)\hookrightarrow W^{1,2}(S^1, E)=H^1(S^1, E)\footnote{View $E\subset \R^3$ isometrically to define the Sobolev spaces.}
$$
yields, after passing to a subsequence, $\Gamma_\delta\to \Gamma$ in $H^1$ as $\delta\to 0$. All bounce orbits are longer than some positive constant and their velocity is bounded as $E_\delta<1$, this implies $T_\delta>c$ for some constant $c>0$. Together with the assumption $T_\delta<C$, up to taking a subsequence, we may assume $T_\delta\to T,\ c\leq T\leq C$. The sequence of finite Borel measures $\mu_\delta$ supported on $\sC_\delta=\lbrace t\in S^1\mid \Gamma_{\delta}(t)\in\partial\Omega_\delta\rbrace$ converges weakly to $\mu$ supported on $\sC=\lbrace t\in S^1\mid \Gamma_{\delta}(t)\in L\rbrace$. \\

\noindent
\textit{Convergence of outer normals.} For every 
sequence of points $p_\delta \in \partial \Omega_\delta$ that converges to a point $p\in\mathrm{int} L$, the outer normals of $\Omega_\delta$ at $p_\delta$ converge to the outer normal of $p\in \mathrm{int} L$ at $\Gamma(t)$. This outer normal is well defined up sign. For a sequence of bounce points this sign is determined by the fact that $\langle \lim_{t\uparrow t_o}\dot \Gamma(t),\nu(\Gamma(t_o))\rangle>0$. Hence, at $t\in \sC$ such that $\Gamma(t)\in \mathrm{int}(L)$, $\Gamma$ satisfies the standard reflection law.\\

\noindent
\textit{Reflection law at the tips.} In order to prove (3)' at $U_{\pm}$, note that when a sequence of points $p_\delta \in \partial \Omega_\delta$ converges to one of the tips of $L$, their normals $\nu(p_\delta, \Omega_\delta)$ converge (up to taking a subsequence) to an element in the half-space $H_L$. Further, for any sequence of bounce points $p_\delta=\gamma_\delta(t_\delta)$, $t_\delta\in\mathrm{supp}(\mu_\delta)$, the normal always lies in $H_{\mathrm{in}}^\delta:=\lbrace v\in T_{p_{\delta}}E \mid v\cdot \dot\gamma_\delta ^+\leq 0\rbrace$, hence their normals need to converge to an element in $H_{\mathrm{in}}$. Together we find $\nu_\delta(\Gamma_\delta(t_\delta)) \to \nu(\Gamma(t))$ up to taking subsequences, and $\nu(\Gamma(t)) \in H_{\mathrm{in}} \cap H_L$.
\end{proof}

\begin{figure}[h]
  \centering
  \includegraphics[width=1\textwidth]{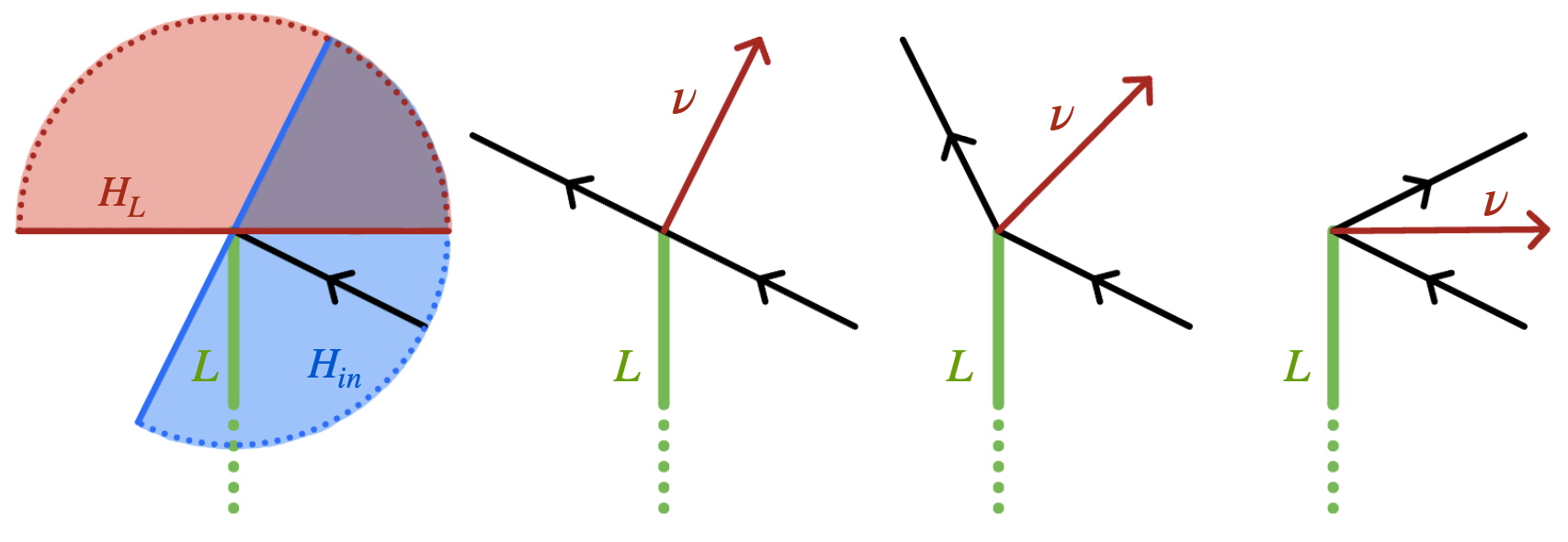} 
  \caption{Denote by $H_L=\lbrace v\in T_{U_{\pm}}E \mid v\cdot l\geq 0\rbrace$ the outer half space at one of the tips ${U_{\pm}}$ of $L$ and $H_{in}=\lbrace v\in T_{U_{\pm}}E \mid v\cdot \dot\gamma_+\leq 0\rbrace$. The left picture shows the half spaces $H_{in}$ and $H_L$, the other pictures show possible reflections with $\nu\in H_{in}\cap H_L$, hence allowed under the generalized reflection law (3)'.}
  \label{reflection}
\end{figure}

\subsubsection{Admissible closed billiard trajectories.} We now go through all types of billiard trajectories (transpolar, circumpolar, umbilical) to show that 
$$
\Omega\text{-}\mathrm{sys}=\min\lbrace 2\mathrm{sys},\ell(\eta)\rbrace.
$$

\noindent
\textit{Circumpolar billiard trajectories.} If a billiard trajectory contains a \emph{circumpolar} segment, then all its segments are circumpolar, as these always hit the interior of \(L\) and thus satisfy the standard reflection law. As \(\Omega\) is invariant under reflection in the \(xz\)-plane, every periodic circumpolar billiard orbit unfolds to a closed circumpolar geodesic with winding number around the \(z\)-axis at least \(2\). The shortest among these covers \(\gamma_z\) twice; hence its length is $2\mathrm{sys}$.\\

\noindent
\textit{Transpolar billiard trajectories.} Transpolar geodesics do not intersect \(L\) at all. Hence all closed transpolar geodesics are closed billiard trajectories. The shortest among these is either \(\gamma_x\) (if \(\omega(a)\le 2\)) or \(\eta\) (if \(\omega(a)\ge 2\)).\\

\noindent
\textit{Umbilical billiard trajectories.}
All umbilical geodesics pass through the tips \(U_\pm\) of \(L\). Indeed, every
geodesic joining two antipodal umbilical points has the same length. In particular,
every geodesic loop based at an umbilical point has the same length as \(\gamma_y\) \cite{Kl95}.

\begin{lemma}
Every umbilical billiard trajectory that does not cover a segment of \(\gamma_y\)
contains at least two geodesic loops based at \(U_+\) or \(U_-\). In particular, its length is at least the length $2\,\ell(\gamma_y)$, and hence at least $2\,\mathrm{sys}$.
\end{lemma}

\begin{proof}
The only periodic umbilical geodesic is \(\gamma_y\). Every other umbilical geodesic
passes repeatedly through antipodal pairs of umbilical points, so any umbilical
billiard trajectory is obtained by concatenating geodesic loops based at one of the tips
\(U_\pm\) of \(L\). Each such loop has length \(\ell(\gamma_y)\). Consider one such loop, given by a segment of an umbilical geodesic \(\xi\). As
illustrated in Figure~\ref{umbilical}, the two ends of the loop determine the
incoming and outgoing directions of a possible one-bounce trajectory at a tip of
\(L\). Now continue \(\xi\) beyond the incoming branch. Since \(\gamma_x\) has Morse
index \(1\), the second intersection of \(\xi\) with \(\gamma_x\) occurs farther
along \(\gamma_x\) in the same direction. Equivalently, the angle \(\alpha\) between \(\dot\xi_{\mathrm{out}}\) and \(L\) at the tip of \(L\) is smaller than the angle \(\alpha'\) between \(\dot\xi_{\mathrm{in}}\) and \(L\). Indeed, this angle decreases monotonically under successive returns to the tip of \(L\); for a precise description of this angle change, see \cite{Har49}. It follows that the change in momentum between the incoming and outgoing directions does not lie in the admissible cone. Hence a single loop cannot satisfy the generalized reflection law~(3)'. Therefore, any umbilical billiard trajectory not covering a segment of \(\gamma_y\) must contain at least two such loops. Consequently, its length is at least $2\,\ell(\gamma_y)\ge 2\,\mathrm{sys}$. In fact, the same argument can be iterated to show that no periodic concatenation of
umbilical geodesic loops satisfies the reflection law at all bounce points. Hence
there are no periodic umbilical billiard trajectories that do not cover a
segment of \(\gamma_y\).
\end{proof}

\begin{figure}[h]
  \centering
  \includegraphics[width=0.6\textwidth]{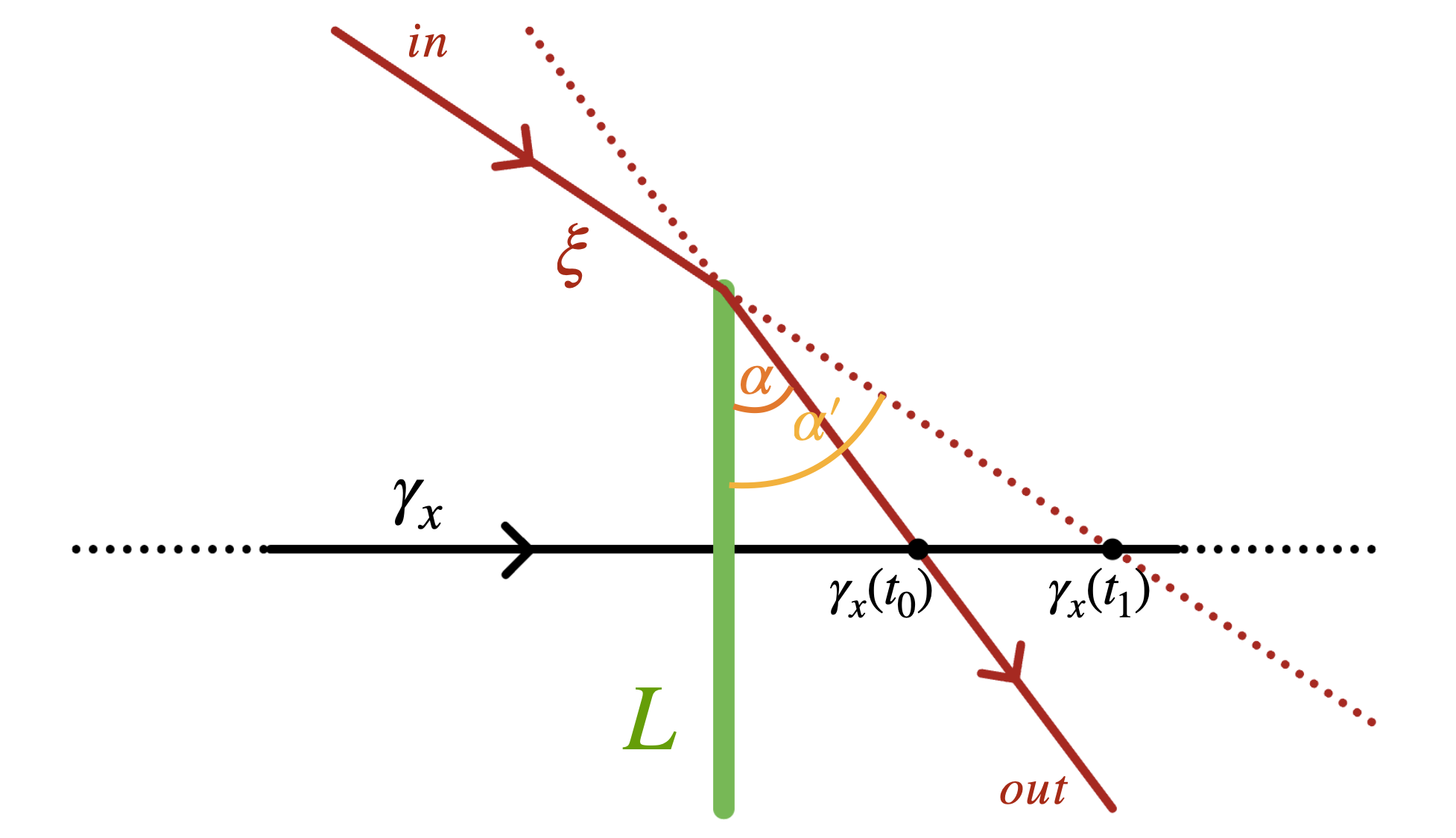} 
  \caption{A one-bounce umbilical loop based at a tip of \(L\). The incoming and outgoing
branches, labeled \(in\) and \(out\), are segments of the same umbilical geodesic
\(\xi\). The dashed continuation indicates that \(\xi\) meets \(\gamma_x\) again
farther along in the same direction, so the angle at the tip is incompatible with
the generalized reflection law.}
  \label{umbilical}
\end{figure}

\noindent
Finally we need to deal with periodic billiard orbits covering (parts of $\gamma_y)$. There is the gliding orbit covering $\gamma_y$ precisely once and the 2-bounce orbit traversing $\gamma_y\setminus L$ twice. The former has length strictly shorter than $\eta$, and violates our lower bound. We claim that this orbit is not admissible, hence can be ignored for the lower bound.

\begin{lemma}
    The gliding orbits that cover $\gamma_y$ exactly once are not admissible.
\end{lemma}

\begin{proof}
The idea is to use the approximating billiard tables $\Omega_\delta\to \Omega$ with smooth boundary and show that there is no sequence of billiard trajectories for $\Omega_\delta$ converging to a gliding orbit that covers $\gamma_y$ exactly once.

Assume, for contradiction, that there exists such a sequence $\Gamma_\delta$. For $\delta>0$ small enough, the curve $\Gamma_\delta$ avoids $(0,\pm b,0)$ and therefore has a well-defined winding number $\pm1$ with respect to the $y$-axis. By our particular choice of $\Omega_\delta$, the value of the first integral $F$ is preserved at each reflection. Hence periodic bounce orbits can again be divided into the same three classes as geodesics, namely \emph{transpolar}, \emph{circumpolar}, and \emph{umbilical}.

Now transpolar and umbilical periodic bounce orbits necessarily have at least two bounces, and therefore cannot have winding number $\pm1$ with respect to the $y$-axis. Thus the only remaining possibility is that $\Gamma_\delta$ is a circumpolar billiard orbit with $F(\gamma_\delta,\dot\gamma_\delta)\to 0$. Moreover, the complement of $\Omega_\delta$ is geodesically convex, so a geodesic segment cannot return to $\partial\Omega_\delta$ without winding once around the $y$-axis. It follows that $\Gamma_\delta$ must be a geodesic loop, that is, a closed curve with exactly one non-smooth point $\Gamma_\delta(t)\in \partial\Omega_\delta$, and that these loops converge in $H^1$ to the gliding orbit $\xi$ covering $\gamma_y$ once.

Such geodesic loops do indeed exist; see Figure~\ref{1bounce}. A circumpolar geodesic with $F=k^2\delta^2$ never enters the region $\{\cos\beta<\delta\}$ and intersects the parameter line $\{\cos\beta=\delta\}$ tangentially. Since $\upomega(F)>1$ for circumpolar geodesics, it oscillates faster than it winds, and successive tangency points with $\{\cos\beta=\delta\}$ wind around the north pole in the direction opposite to the winding of $\Gamma_\delta$. Thus the resulting geodesic loop has a single non-smooth point, but it fails the billiard reflection law there; see Figure ~\ref{1bounce}. Hence it is not a billiard trajectory of $\Omega_\delta$, a contradiction.
\end{proof}

\begin{figure}[h]
  \centering
  \includegraphics[width=0.6\textwidth]{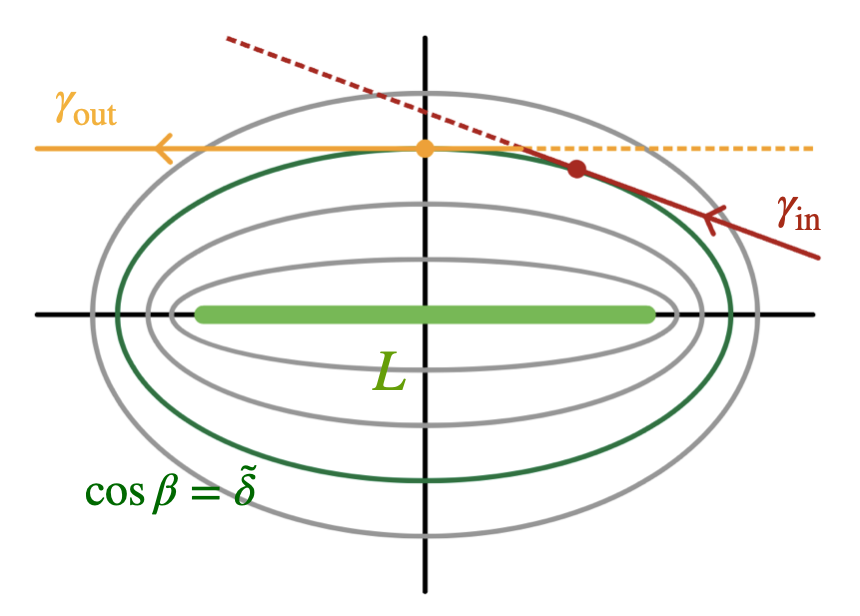} 
  \caption{The $\beta$-parameter lines (shown in grey) bound geodesically convex neighborhoods of $L$ (light green). A circumpolar geodesic with $F=k^2\tilde\delta^2$ gives rise to a geodesic loop (yellow and red are two ends of the same geodesic loop) with a single non-smooth point (the intersection of the red and yellow curves), but this loop does not satisfy the billiard reflection law.}
  \label{1bounce}
\end{figure}

Finally, we need to deal with the $2$-bounce orbit that traverses
$\gamma_y\setminus L$ twice. Its length is $2\ell(\gamma_y\setminus L)$.
The length estimate proved in Lemma~\ref{Lemma: 2-bounce orbit vs index-3 orbit} in Appendix~\ref{appendix: period map} shows that
\[
\ell(\gamma_y\setminus L)\geq \frac{1}{2}\ell(\eta).
\]
Indeed, in the notation of Appendix~\ref{appendix: period map}, the half-length of the shortest closed geodesic $\eta$ of Morse index $3$ is $T(\gamma_0)$, while
Lemma~\ref{Lemma: 2-bounce orbit vs index-3 orbit} gives
$\ell(\gamma_y\setminus L)\geq T(\gamma_0)$. Hence the $2$-bounce orbit has length at least $\ell(\eta)$. This concludes the proof of the lower bound 
$$
c_{HZ}(D^*E(a,b,c),\dd\lambda)\geq\min \lbrace 2\mathrm{sys},\ell(\eta)\rbrace.
$$

\newpage

\section{Step 2: Pseudoholomorphic curves}

In this section, we prove the following theorem on the upper bound of the Hofer--Zehnder capacity. It is a refined version of Theorem~\ref{Theorem: Upper bound from index 3 - refined + geodesic flow version} and we will deduce Theorem \ref{Theorem: Upper bound from index 3 - refined + geodesic flow version} in Section~\ref{Section: Positive curvature}.

\begin{theorem}
\label{Theorem: Upper bound from index 3 - nondegenerate W}
Let $W \subset \operatorname{int}(D^*_{g_{\rm 0}}S^2)$ be a fiberwise starshaped domain in $T^*S^2$ whose Reeb flows on the boundary is nondegenerate. Then, the Hofer--Zehnder capacity satisfies
\begin{equation}
    c_{\rm HZ} (W) \le \max_{\gamma \in \mathcal{P}_{\rm HZ}(W)} \mathcal{A}(\gamma),
\end{equation}
where $\mathcal{P}_{\rm HZ}(W)$ denotes the set of closed Reeb orbits of Conley--Zehnder index $3$ satisfying either
\begin{itemize}
    \item $\gamma$ is non-contractible in $\partial W$ and satisfies $\mathcal{A}(\gamma) \le 2 \pi$, or
    \item $\gamma$ is contractible in $\partial W$ and
    there exists a non-contractible closed Reeb orbit $\gamma'$ satisfying
    \[{\rm CZ}(\gamma')=1 \quad\text{and}\quad\mathcal{A}(\gamma) + \mathcal{A}(\gamma') \le 2\pi.\]
\end{itemize}
\end{theorem}

Here, the Conley--Zehnder index is computed with respect to a global symplectic trivialisation of the contact manifold $(\partial W,\ker\lambda_{\rm taut})$. More precisely, the contact boundary $\partial W$ is contactomorphic to $(\partial D^*_{g_0}S^2,\ker\lambda_{\rm taut})$ via the fiberwise radial map, which is naturally identified with $(\mathbb{RP}^3,\xi_{\rm univ})$, where $\xi_{\rm univ}$ denotes the universally tight contact structure. That is, the pullback of $\xi_{\rm univ}$ under the double covering $S^3\to\mathbb{RP}^3$ is the standard tight contact structure on $S^3$. Consequently, the contact distribution admits a global symplectic trivialisation
\[
    \Phi:(\mathbb{R}^2,\omega_0)\times \partial W \longrightarrow
    \bigcup_{y\in \partial W} (\xi_y,d\lambda_{\rm taut}|_{y}).
\]

With respect to this trivialisation, the Conley--Zehnder index of a nondegenerate closed Reeb orbit $\gamma$ is defined as the Conley--Zehnder index of the associated path of $2\times2$ symplectic matrices obtained from the linearised Reeb flow. For degenerate closed Reeb orbits, there are two possible extensions of the Conley--Zehnder index: the lower semicontinuous extension $\mu^-$ and the upper semicontinuous extension $\mu^+$. When $\gamma$ is a closed geodesic, $\mu^-$ agrees with the Morse index of $\gamma$ with respect to the energy functional, while $\mu^+$ agrees with the sum of the Morse index and the nullity. We will use these facts without further discussion and refer the reader to \cite{Salamon-Zehnder92} for the definition of the Conley--Zehnder index, and to \cite{Biran-Polterovich-Salamon03,Weber06} for its relation to the Morse index of closed geodesics with respect to the energy functional $\mathcal{E}_g$; see also \cite{Duistermaat76}. Theorem~\ref{Theorem: Upper bound from index 3 - nondegenerate W} extends to the degenerate case as follows.

\begin{theorem}
\label{Theorem: Upper bound from index 3 - degenerate W}
Let $W \subset \operatorname{int}(D^*_{g_{\rm 0}}S^2)$ be a fiberwise starshaped domain in $T^*S^2$ whose Reeb flows on the boundary is possibly degenerate. Then, Hofer--Zehnder capacity satisfies
\begin{equation}
    c_{\rm HZ} (W) \le \sup_{\gamma \in \mathcal{P}_{\rm HZ}(W)} \mathcal{A}(\gamma),
\end{equation}
where $\mathcal{P}_{\rm HZ}(W)$ denotes the set of closed Reeb orbits $\gamma$ of $\mu^-(\gamma) \le 3 \le \mu^+(\gamma)$ satisfying either
\begin{itemize}
    \item $\gamma$ is non-contractible in $\partial W$ and
    $\mathcal{A}(\gamma) \le 2\pi$, or
    \item $\gamma$ is contractible in $\partial W$ and there exists a
    non-contratible closed Reeb orbit $\gamma'$ satisfying
    \[
        \mu^-(\gamma') \le 1 \le \mu^+(\gamma')
        \quad\text{and}\quad
        \mathcal{A}(\gamma) + \mathcal{A}(\gamma') \le 2\pi .
    \]
\end{itemize}
\end{theorem}

We first sketch the main ideas of the proof. The full proofs of the two theorems will be given at the end of this section. The rough idea is to view $W$ as a subdomain of $\cp^1 \times \cp^1$ which is disjoint from the diagonal $\Sigma$. Based on the ideas of \cite{HV92}, we pass from pseudoholomorphic spheres in $\cp^1 \times \cp^1$ to Floer spheres, from $p$ to $\Sigma$, for a suitable perturbation of a HZ-admissible Hamiltonian (Corollary \ref{Corollary: existence of Floer spheres}).

We then apply a neck-stretching argument in a geometric setup similar to that of \cite{HWZ03}. This produces Floer planes, from $p$ to a closed Reeb orbit $\gamma$, in the symplectic completion $\widehat W$ (Proposition~\ref{Proposition: Bottom level}). Theorem~\ref{Theorem: Upper bound from index 3 - nondegenerate W} is obtained by analyzing the $L^2$-energy together with the Fredholm indices of the curves arising as SFT-limits. Theorem~\ref{Theorem: Upper bound from index 3 - degenerate W} then follows by taking a limit of non-degenerate contact forms.

\subsection{From pseudo-holormophic spheres to Floer spheres}

The aim of this subsection is to prove Proposition \ref{Proposition: moduli space of Floer spheres}, which establishes the existence of a Floer cylinder associated to a given HZ-admissible Hamiltonian. We begin by setting up the geometric framework.

Throughout this section, we assume that $W$ is a fiberwise starshaped domain in $T^*S^2$, contained in the disk cotangent bundle $D^*_{g_0} S^2$ of the round sphere of radius $1$. By embedding the interior ${\rm int}(D_{g_0}^* S^2)$ to $\cp^1 \times \cp^1$, or equivalently, after applying a symplectic cut \cite{ler} along the boundary $\partial D_{g_0}^* S^2$, such domain $W$ symplectically embeds into $\cp^1 \times \cp^1$ via
\[
(W, \omega_{\rm std}) \hookrightarrow ({\rm int} (D^*_{g_0} S^2), \omega_{\rm std}) \to (\cp^1 \times \cp^1, 2\omega_{\rm FS} \oplus 2\omega_{\rm FS}).
\]
Hence, we may regard $(W, \omega_{\rm std})$ as a subdomain of $(\cp^1 \times \cp^1, 2 \omega_{\rm FS} \oplus 2 \omega_{\rm FS})$. We will often abbreviate the symplectic manifold $(\cp^1 \times \cp^1, 2\omega_{\rm FS} \oplus 2\omega_{\rm FS})$ by $(M,\omega)$, and the restriction of $\omega|_W$ to the subdomain $W$ also simply by $\omega$.

The boundary $\partial W$ then embeds as a separating hypersurface of contact type. More precisely, there exists a collar neighborhood $U_{\partial W}$ of $\partial W$ with a canonical identification
\begin{equation}\label{Equation: nbhd of partial W}
(U_{\partial W}, 2 \omega_{\rm FS}\oplus 2 \omega_{\rm FS}) \cong \left( (-\epsilon, \epsilon)_r \times \partial W , d(e^r \cdot \lambda_{\rm taut}|_{\partial W}) \right).
\end{equation}
Here, seperating means that, denoting $\cp^1 \times \cp^1 \setminus \rm{int}(W)$ by $V$, we have
\[
\cp^1 \times \cp^1 \setminus \partial W = {\rm int}(W) \sqcup\, {\rm int}(V).
\]

The contact hypersurface $\partial W$ is a convex contact boundary of the symplectic manifold $W$ and a concave contact boundary of the symplectic manifold $V$. The $0$-section embeds as an anti-diagonal Lagrangian on the interior of $W$, and symplectic submanifold $\Sigma$ defined as a diagonal lies on the interior of $V$. 

In this subsection, we consider two moduli spaces associated with the line class
\begin{equation}
\label{Equation: Line class A}
    A := [\mathbb{CP}^1\times\{\mathrm{pt}\}]
    \in H_2(\mathbb{CP}^1\times\mathbb{CP}^1).
\end{equation}
For a given Hofer--Zehnder admissible Hamiltonian, the first moduli space consists of pseudoholomorphic spheres passing the fixed minimum point $p$ to the divisor $\Sigma$ (problem (M-J)), while the second consists of Floer spheres with the same asymptotic conditions (problem (M)). We begin with the pseudoholomorphic spheres.\\

\noindent \textbf{The problem (M-J)}: For an $\omega$-compatible almost complex structure $J$ on $(M,\omega)$, denote by $\mathcal{M}(p,\Sigma;J,A)$ the moduli space of $J$-holomorphic spheres representing the line class $A$ and passing through
a point $p\in W$ and the divisor $\Sigma$. More precisely, it consists of maps $u:(\cp^1,i)\to(M,J)$ satisfying
\begin{equation}
\label{Equation: Cauchy-Riemann equation - sphere vers}
    du + J \circ du \circ i = 0\quad\text{and}\quad[u]=A
\end{equation}
together with the point constraints
\begin{equation}
u(0) = p \quad \text{and} \quad u(\infty) \in \Sigma.
\end{equation}

Since the line class $A$ is $\omega$-\textit{minimal}, i.e., 
\begin{equation}
\label{Equation: omega-minimality}
\omega(A) > 0 \quad \text{and} \quad \omega(A) = \inf \lbrace{w(B)>0:B \in H_2(M;\mathbb{Z})}\rbrace,
\end{equation}

any $J$-holomorphic sphere representing $A$ cannot be multiply-covered, or equivalently is \textit{somewhere injective}, i.e., there exists a point $z \in \cp^1$ such that 
\begin{equation}
\label{Equation: Somewhere injecitivty}
du(z) \ne 0 \qquad \text{and} \qquad u^{-1}(u(z))=\lbrace z \rbrace.
\end{equation}
Such a point $z$ is called a \textit{somewhere injective point}. The set of somewhere injective points is an open dense subset of $\cp^1$.

The virtual dimension of the moduli space is given by
\begin{equation}
\label{Equation: vdim constrained}
    \operatorname{vdim}\,\mathcal{M}(p,\Sigma;J,A)
    = 2c_1(A) + \dim M - \operatorname{codim}(p) - \operatorname{codim}(\Sigma) = 2.
\end{equation}
Since every solution is somewhere injective, a generic $\omega$-compatible almost complex structure $J$ is \textit{regular}. Recall that $J$ is called regular if, for every $u \in \mathcal{M}(p,\Sigma;J,A)$, the associated linearised operator $D_u$ is surjective. For such a regular $J$, the moduli space $\mathcal{M}(p,\Sigma;J,A)$ is a smooth $2$-dimensional manifold by the implicit function theorem.

\begin{remark}[Automatic transversality]
\label{Remark: Automatic transversality}
For $u \in \mathcal{M}(p, \Sigma;J,A)$, the adjunction formula 
\[
2\cdot \delta(u) - \chi(S^2) \le A \cdot A - c_1(A),
\]
implies that $u$ is always an embeddeding. Then, it follows from \cite[Thm.~1]{HLS97} that every $\omega$-compatible $J$ is regular. 
For the remainder of this section, we only require somewhere injectivity, and hence regularity for generic $J$, together with the fact that $i\oplus i$ on $\cp^1 \times \cp^1$ is regular for the line class $A$, \cite[Cor.~3.3.5]{MS12}
\end{remark}

The moduli space $\mathcal{M}(p, \Sigma;J,A)$ admits a $\mathbb{C}^*$-action defined by
\begin{equation}
\label{Equation: C*-action on B}
\lambda \cdot u(z) \to u(\lambda z)    
\end{equation}
where we identify $\mathbb{CP}^1$ with $\mathbb{C}\cup\{\infty\}$. Since this action fixes both $0$ and $\infty$, the map $\lambda\cdot u$ also belongs to
$\mathcal{M}(p,\Sigma;J,A)$. The $\mathbb{C}^*$-action free due to the somewhere injectivity, and thus the unparametrised moduli space
\[
\widehat{\mathcal{M}}(p, \Sigma;J,A) = \mathcal{M}(p, \Sigma;J,A) / \mathbb{C}^*
\]
is a smooth $0$-dimensional manifold for generic $\omega$-compatible $J$.

Moreover, $\widehat{\mathcal{M}}(p,\Sigma;J,A)$ is compact with respect to the smooth topology by the $\omega$-minimality of the line class $A$.
This follows from Gromov compactness for pseudoholomorphic curves together with an argument using the minimality of the $L^2$-energy. Recall that the $L^2$-energy of a curve $u:\mathbb{CP}^1\to M$ is defined as
\begin{equation}
\label{equation: L2-energy of the holomorphic curves}
E(u):=\frac{1}{2}\int_{\cp^1}\|du(z)\|_J^2\,{\rm dvol},
\end{equation}
where $\|du(z)\|_J$ denotes the operator norm with respect to some metric on $\cp^1$ and the metric $g_J$ on $M$ induced by $J$ as
\begin{equation}
\label{Equation: Metric from compatible almost complex structure}
    g_J(\cdot,\cdot):=\omega(\cdot,J\cdot),
\end{equation}
The energy is independent of the choice of auxiliary metric. In particular, when $u$ is a pseudoholomorphic curve, it only depends on the homology class that $u$ represents, i.e.
\[E(u)= \omega([u]).\]

The discussion above can be summarized in the following lemma.
\begin{lemma}
\label{Lemma: unparametrised moduli space}
For a generic $\omega$-compatible almost complex structure $J$, the unparametrized moduli space of $J$-holomorphic spheres,
\[
\widehat{\mathcal{M}}(p,\Sigma;J,A),
\]
is a compact $0$-dimensional manifold. In particular, it consists of finitely many points.
\end{lemma}

\begin{remark}[Orientation]
Note that the lemma can be strengthened by asserting that $\widehat{\mathcal{M}}(p,\Sigma;J,A)$ is a compact oriented $0$-dimensional manifold. This allows us to consider the oriented count $\#\widehat{\mathcal{M}}(p,\Sigma;J,A).$ However, to prove the existence of pseudoholomorphic curves, it suffices to consider only the mod $2$ count $\#_2\widehat{\mathcal{M}}(p,\Sigma;J,A)$.
\end{remark}

The moduli space may apriori be empty. We next prove that it is not the case.

\begin{proposition}
\label{Proposition: moduli space of J-holomorphic spheres}
For any regular almost complex structure $J$, the mod 2 count
\[
\#_2\widehat{\mathcal{M}}(p,\Sigma;J,A)=1.
\]
In particular, they are not empty.
\end{proposition}

When the moduli space $\mathcal{M}(p,\Sigma;J,A)$ is empty, the almost complex structure $J$ is regular by definition. For completeness, we state the following corollary on the existence of pseudoholomorphic spheres, although it will not be used in the rest of the paper. As discussed in Remark~\ref{Remark: Automatic transversality}, this corollary can also be deduced from Proposition~\ref{Proposition: moduli space of Floer spheres} by automatic transversality.

\begin{corollary}
\label{Corollary: existence of J-holomorphic spheres}
The moduli space $\widehat{\mathcal{M}}(p,\Sigma;J,A)$ is non-empty for any choice of $J$.
\end{corollary}

\begin{proof}[Proof of Proposition \ref{Proposition: moduli space of J-holomorphic spheres}]
    For the product complex structure $i\oplus i$, there exists a unique unparametrized holomorphic sphere representing the line class $[\cp^1 \times \lbrace {\mathrm pt} \rbrace ]$ that passes through $p$ and intersects $\Sigma$, namely,
    \[
    \operatorname{pr}_1^{-1} \circ \operatorname{pr}_1 (p))
    \]
    where $\operatorname{pr}_1:\cp^1\times\cp^1\to\cp^1$ denotes the projection onto the first factor. Therefore, the moduli space $\widehat{\mathcal{M}}(p,\Sigma; i \oplus i,A)$ is a singleton.

    Consider a smooth homotopy $\mathcal{J}=\{J_{\tau}\}_{\tau\in[0,1]}$ connecting the two regular almost complex structures $J_0=i\oplus i$ and $J_1=J$. For a generic choice of $\mathcal{J}$, the associated parametrised moduli space
    \[
        \mathcal{M}(p,\Sigma;\mathcal{J},A) :=
        \{(\tau,u)\mid \tau \in [0,1],~u\in\mathcal{M}(p,\Sigma;J_\tau,A)\}
    \]
    is a smooth manifold with boundary. The boundary is then given by
    \[
        \partial\mathcal{M}(p,\Sigma;\mathcal{J},A)
        = \mathcal{M}(p,\Sigma;i\oplus i,A) \sqcup \mathcal{M}(p,\Sigma;J,A).
    \]
    
    As in the case of $\widehat{\mathcal{M}}(p,\Sigma;J,A)$, the parametrized moduli space $\mathcal{M}(p,\Sigma;\mathcal{J},A)$ admits a $\mathbb{C}^*$-action given by \eqref{Equation: C*-action on B}. The $\omega$-minimality of the class $A$ also implies that the parametrised moduli space of unparametrised pseudoholomorphic spheres,
    \[
    \widehat{\mathcal{M}}(p,\Sigma;\mathcal{J}, A) = \mathcal{M}(p,\Sigma;\mathcal{J}, A) / \mathbb{C}^*
    \]
    is a compact smooth $1$-dimensional manifold with boundary. Again, the $\omega$-minimality implies that the boundary is precisely
    \[
    \widehat{\mathcal{M}}(p,\Sigma;i \oplus i, A) \sqcup \widehat{\mathcal{M}}(p,\Sigma;J, A)
    \]
    Since the mod $2$ count of the boundary points of a compact $1$-dimensional manifold is always even, we conclude that
    \[
    \#_2\widehat{\mathcal{M}}(p,\Sigma;J,A) =\#_2\widehat{\mathcal{M}}(p,\Sigma;i\oplus i,A) = 1.
    \]
    This completes the proof.
\end{proof}

So far, we have considered almost complex structures without any additional assumptions. For our purposes, we impose several assumptions, starting with (J-1). We denote by $\mathcal{J}(\Sigma)$ the set of $2\omega_{\rm FS}\oplus 2\omega_{\rm FS}$-compatible almost complex structures $J$ on $\mathbb{CP}^1\times\mathbb{CP}^1$ satisfying the following condition:
\begin{itemize}
    \item[(J-1)] $J$ coincides with the complex structure $i\oplus i$ on a
    neighborhood $U_\infty$ of $\Sigma$.
\end{itemize}

The neighborhood $U_\infty$ of $\Sigma$ is fixed once and for all. For example, we choose $U_\infty$ as a tubular neighborhood of $\Sigma$, identified with a neighborhood of the zero section in its complex normal line bundle. This choice does not affect the arguments that follow.

Under (J-1), $\Sigma$ defines a $J$-holomorphic sphere representing the diagonal class
\[
[\mathbb{CP}^1 \times \{\mathrm{pt}\}] + [\{\mathrm{pt}\}\times \mathbb{CP}^1] \in H_2(\cp^1 \times \cp^1).
\]
Since the algebraic intersection number between $u$ and $\Sigma$ is given by
\[
u \cdot \Sigma = [\mathbb{CP}^1 \times \{\mathrm{pt}\}] \circ 
\left(
[\mathbb{CP}^1 \times \{\mathrm{pt}\}]
+
[\{\mathrm{pt}\}\times \mathbb{CP}^1]
\right)
=1,
\]
every curve $u\in \mathcal{M}(p,\Sigma;J,A)$ intersects $\Sigma$ at exactly one point, and the intersection is transverse by positivity of intersections; see, for example, \cite[App~E.2]{MS12}.

Similarly, for $u_1,u_2\in \mathcal{M}(p,\Sigma;J,A)$, the algebraic intersection number between them is zero. This implies that either the images of $u_1$ and $u_2$ are disjoint or the two curves agree. Since both curves pass the point $p$ at $z=0$, it follows that there exists a holomorphic automorphism $\phi:\mathbb{CP}^1\to \mathbb{CP}^1$ satisfying
\[u_1 = u_2 \circ \phi.\]
Furthermore, since $u_1(\infty)$ and $u_2(\infty)$ are the unique intersection points with $\Sigma$, the map $\phi$ fixes $z=\infty$. As $u_1$ and $u_2$ are embeddings by the adjunction formula, the map $\phi$ also fixed $z=0$. The reparametrization $\phi$ is then given by the $\mathbb{C}^*$-action. This in fact implies that, for a regular $J$ satisfying {\rm (J-1)}, not only
the mod $2$ count but also the unoriented count is equal to one, i.e.,
\[
|\widehat{\mathcal{M}}(p,\Sigma;J,A) |=1.
\]
Nevertheless, in our argument, it suffices to consider only the mod $2$ count.

Lastly, we note that the moduli space $\mathcal{M}(p, \Sigma;J, A)$ is naturally identified with the moduli space of $J$-holomorphic cylinders with prescribed asymptotic behavior:
\begin{equation}
\label{Equation: Cauchy-Riemann equation - cylinder vers}
M_Z(p,\Sigma;J,A):=\left\{
u:\mathbb{R}\times S^1 \to M~\middle|~~
\begin{aligned}
&\quad \partial_s u + J \partial_t u = 0, \qquad[\bar{u}]=A\\ &\lim_{s \to- \infty} u(s,t)=p, \quad \lim_{s \to \infty} u(s,t)\in \Sigma,
\end{aligned}
\right\}.
\end{equation}
Here $\bar{u}$ denotes the extension of $v$ to a map $\bar{u}:\mathbb{CP}^1\to M$ that follows from the asymptotic behavior. By the removal of singularities, $\bar{u}$ defines a sphere map in $\mathcal{M}(p,\Sigma;J,A)$, and this construction gives a one-to-one correspondence between $J$-holomorphic cylinders and spheres. Accordingly, we will freely regard an element $u\in\mathcal{M}(p,\Sigma;J,A)$ as either a cylinder or a sphere, and use the same notation for the moduli space of (unparametrised) pseudoholomorphic cylinders, omitting the subscript $Z$.

\vspace{1cm}

\noindent \textbf{The problem (M)}: To describe the problem, we begin with a  deformation of a Hofer--Zehnder admissible Hamiltonian, corresponding to the $C^0$-small perturbation appearing in \cite[Prop.~3.3]{Irie14}. The purpose of this perturbation is to keep the almost complex structure autonomous for Floer equations as in \cite[Sec.~7]{FHS95}, rather than considering a compactly supported $S^1$-invariant perturbation as in \cite{HV92}.

By slight abuse of notation, we also denote by $H: M \to \mathbb{R}_{\le 0}$\footnote{Here, unlike in the introduction, we adopt the convention that admissible Hamiltonians are non-positive,  so that the maximum is replaced by the minimum. This is more convenient for the Floer-theoretic setup.} an extension of the Hofer-Zehnder admissible Hamiltonian $H: W \to \mathbb{R}_{\le 0}$ compactly supported in the interior $\operatorname{int}(W)$, namely
\begin{equation}
H(x) =
\begin{cases}
H(x) & x \in W,\\
0 & x \in V.
\end{cases}
\end{equation}

For a given HZ--admissible Hamiltonian $H$, let $U_{\rm min}$ be a Darboux ball centered at a point $p$ where $H$ attains its minimum value and is constant, i.e.,
\[ H(x)=\min H \qquad \forall\,x\in U_{\rm min}. \]
Fix a Darboux chart
\begin{equation}
\label{Equation: Darboux Chart i}
    \iota:(B^4(r_0),\omega_0)\longrightarrow (U_{\rm min},\omega|_{U_{\rm min}}).
\end{equation}
Using this chart, we perturb $H$ inside $U_{\rm min}$ by
\begin{equation}
\label{Equation: C0-perturbation inside Darboux Ball}
H^\prime(z) = 
\begin{cases}
    H(z) & \text{if}~z\notin U_{\rm min},\\
    H(z)+ F \circ \iota^{-1}(z) &\text{if}~z \in U_{\rm min},
\end{cases}
\end{equation}
where $F:B^4(r_0)\to \RR_{\le 0}$ is a smooth function such that there exist positive constants $0<r_-<r_+<r_0$ satisfying
\[
F(x)=
\begin{cases}
S(x)-b, & x\in B^4(r_-),\\
0, & x\notin B^4(r_+),
\end{cases}
\]
for some positive-definite quadratic form $S:\mathbb{R}^4\to\mathbb{R}$ and some $b>0$. We also identify $S$ with its associated positive-definite symmetric matrix and impose further conditions on this matrix; see \eqref{Equation: Admissible equation}. At this point, it suffices to note that the perturbed Hamiltonian $H'$ can be chosen to satisfy the following properties, as illustrated in \cite{Irie14}, for any such symmetric matrix $S$ with eigenvalues less than $\pi$:
\begin{itemize}
    \item Any nonconstant contractible periodic orbit of $X_{H^\prime}$ has period larger than 1.
    \item $\min H^\prime < \min H$
    \item $\min H^\prime$ is isolated in the set of critical values of $H^\prime$.
    \item $\min H^\prime$ is attained at the point $p$, the center of a Darboux ball $U_0$.
\end{itemize}

The constant loop at $p$ is nondegenerate as a $1$-periodic orbit of $X_{\tau H^\prime}$ for $\tau \in (0,1]$ and the Conley--Zehnder index of $p$ is precisely
\[
\mu_{\rm CZ}(p)= 2 - \mu_{\rm Morse}(p) = 2.
\]
The index can also be verified directly, since the associated path of symplectic matrices is generated by the positive definite symmetric matrix with eigenvalues less than $\pi$.

For an $\omega$-compatible almost complex structure $J$, denote by $\mathcal{M}(p,\Sigma;J,H^\prime,A)$ the moduli space of Floer cylinder asymptotic to $p, \Sigma$ representing the line class $A$. More
precisely, it consists of maps $u:\mathbb{R}\times S^1\to M$ satisfying
\begin{equation}
    \label{Equation: Floer equation}
    \partial_s u + J(u(s,t))\partial_t u = \nabla H^\prime(u(s,t)).
\end{equation}
with asymptotic conditions
\[
\lim_{s\to-\infty}u(s,t)=p,
\qquad
\lim_{s\to+\infty}u(s,t)\in\Sigma,
\]
and such that the continuous extension $\bar{u}$ obtained by identifying
$\mathbb{R}\times S^1\cup\{-\infty,+\infty\}$ with $\mathbb{CP}^1$ represents the line class $A$. The gradient $\nabla H^\prime$ is taken with respect to the Riemannian metric $g_J$ defined as \eqref{Equation: Metric from compatible almost complex structure}, or equivalently, $\nabla H^\prime = JX_{H^\prime}$.

We refer to such
solutions as \textit{Floer spheres} from $p$ to $\Sigma$. 
Since $H'$ vanishes near $\Sigma$, the extension $\bar{u}$ is smooth at $+\infty$ and, by the removal of singularities theorem, defines a $J$-holomorphic disk near $+\infty$. More precisely, identifying $+\infty$ with the north pole of $\mathbb{CP}^1$, the extension defines a $J$-holomorphic disk on a neighborhood of the north pole.

\medskip

One can view the Floer equation as a perturbed Cauchy-Riemann equation in the case of \eqref{Equation: Cauchy-Riemann equation - cylinder vers}, and moduli spaces of Floer spheres can also be understood as moduli space associated to  Hamiltonian perturbations in the sense of \cite[Ch.\ 8]{MS12}. Note that the nondegeneracy of $p$ as $1$-periodic orbit of $H^\prime$ implies that the moduli space is a Fredholm problem. In particular, its virtual dimension is given by
\[
    \operatorname{vdim}\mathcal{M}(p,\Sigma;J,{H^\prime},A)=
    \dim_{\mathbb{C}}M\cdot\chi(\mathbb{CP}^1\setminus\{0\})
    +2c_1(A)-\mu_{\mathrm{CZ}}(p)-\operatorname{codim}(\Sigma)
    =2.
\]

As in \eqref{equation: L2-energy of the holomorphic curves}, the $L^2$-energy of a Floer spheres $u \in \mathcal{M}(p,\Sigma;J,H^\prime, A)$ is given by
\begin{equation}
E(u):=\frac{1}{2}\int_{\mathbb{R}\times S^1}\left\|\frac{\partial u}{\partial s} \right\|_J^2 + \left\|\frac{\partial u}{\partial t}-X_{H^\prime}(u) \right\|_J^2\,{\rm ds\,dt},
\end{equation}
where the norm is taken with respect to the metric $g_J$ on $M$ induced by $J$. As before, the energy is independent of the choice of $J$, and satisfies
\begin{equation}
\label{Equation: L2-energy of Floer spheres}
\begin{aligned}
    E(u) = \omega(A) + H^\prime(p).
\end{aligned}
\end{equation}
This follows from the fact that the Floer equation \eqref{Equation: Floer equation} is equivalent to the negative gradient flow of the Hamiltonian action functional along the capped loops $\{(\gamma_\sigma, \bar{u}_\sigma)\}_{\sigma \in \mathbb{R}}$
\[
\mathcal{A}_{H^\prime}(\gamma_\sigma,\bar{u}_\sigma)
=
\int_{\bar{u}_\sigma}\omega
-
\int_{S^1}H^\prime(\gamma_\sigma(t))\,dt,
\]
where the loop $\gamma_\sigma$ is defined as
\[
\gamma_\sigma(t)=u(\sigma,t),
\]
and $\bar{u}_\sigma$ is the capping disk obtained by extending $u|_{\{s\le\sigma\}}$ by the point $\bar{u}(-\infty)=p$. Since $H^\prime(p)<0$, the $L^2$-energy of every Floer sphere is strictly smaller than the minimal energy of a nonconstant $J$-holomorphic sphere, namely,\[ E(u)<\omega(A). \]

Unlike the case of pseudoholomorphic spheres in the line class or that of Floer cylinders for time-dependent almost complex structures, achieving transversality here can be more subtle. In this setting, the framework developed in \cite[Sec.~7]{FHS95} applies and provides abundant \textit{somewhere injective regular points}, namely, points $z$ in the domain of a Floer cylinder $u$ satisfying somewhere injectivity \eqref{Equation: Somewhere injecitivty} and additionally satisfying that the four tangent vectors
\[
\partial_su,\qquad \partial_tu,\qquad X_{H^\prime}(u),\qquad \nabla {H^\prime}(u)
\]
are linearly independent. Since the autonomous setting is less standard, we briefly recall the relevant details for completeness.

We say Hamiltonian $H^\prime$ is \textit{regular} at the critical point $p$ if there exists a compatible almost complex structure $J$ such that the symmetric matrix
\begin{equation}
\label{Equation: Symmetric matrice on the local chart}
S = J_0 \circ \Phi^{-1}\circ dX_H(p)\circ \Phi    
\end{equation}
for a unitary frame $\Phi:(\mathbb{R}^4 , i) \to (T_pM,J)$ satisfies the following:
for any four non trivial tuple of real numbers $(\alpha,\beta,\hat{\alpha},\hat\beta)$, there is no nonzero solution $\zeta \in \mathbb{R}^4$ of the equations
\begin{equation}
\label{Equation: Admissible equation}
\begin{aligned}
    &(SJ_0 - J_0 S - \alpha - \beta J_0)\zeta = 0, \qquad \text{and} \\
    &(SJ_0 - J_0 S - \alpha - \beta J_0)S\zeta -\hat \alpha \zeta - \hat \beta J_0 \zeta= 0.
\end{aligned}
\end{equation}
Note that this definition is independent of the choice of unitary frame at $p$. Moreover, the set of symmetric matrices satisfying the above conditions is
open and dense in the space of all symmetric matrices
\cite[Thm.~6.1]{FHS95}. Conversely, if $H'$ is regular at $p$, we say that $J\in\mathcal{J}(M)$ is \textit{admissible} if there exists such a unitary frame. 

For $J$ that agrees with the standard complex structure at $p$ with respect to $\iota$, the symmetric matrix in the unitary frame \eqref{Equation: Symmetric matrice on the local chart} agrees with the symmetric matrix chosen in the construction. We always assume that the symmetric matrix $S$ is chosen so as to ensure that $H'$ is regular at $p$. The almost complex structure $J$ is automatically admissible under the following additional assumption:
\begin{itemize}
    \item[(J-2)] The pullback $\iota^*J$ agrees with the standard complex structure $i$ at the origin in the Darboux chart $\iota:(B^4(r_0),\omega_0)\to(U_{\rm min},\omega|_{U_{\rm min}})$.
\end{itemize}
We denote by $\mathcal{J}(p,\Sigma)$ the set of compatible almost complex structures satisfying (J-1), (J-2).

\medskip

Under these assumptions, applying \cite[Lem.~7.8]{FHS95} yields the following lemma, which provides an analogue of the abundance of somewhere injective points for Floer cylinders with autonomous Floer data $(H',J)$.

\begin{lemma}
\label{Lemma: Somewhere injectivity of Floer spheres}
Let $H$ be a Hofer--Zehnder admissible Hamiltonian compactly supported in the
interior $\operatorname{int}(W)$. With respect to the Darboux chart $\iota$ in \eqref{Equation: Darboux Chart i} centered at the minimum $p$, let $H'$ be the
perturbation of $H$ as in \eqref{Equation: C0-perturbation inside Darboux Ball}, and let $J\in\mathcal{J}(p,\Sigma)$. Then there exists a neighborhood $U=U(H', J)\subset U_{\rm min}$ of the center $p$ such that, for every Floer cylinder $u\in\mathcal{M}(p,\Sigma;J,{H^\prime},A)$, the set of regular somewhere injective points $z\in u^{-1}(U)$ is open and dense in $u^{-1}(U)$.
\end{lemma}
\begin{proof}
    Since a curve representing the line class $A$ cannot be multiply covered, every solution $u$ is \textit{simple} in the sense of \cite[Sec.~7]{FHS95}; that is, for every integer $m>1$, there exists a point $(s,t)\in\mathbb{R}\times S^1$ such that
    \[ u(s,t+1/m)\neq u(s,t). \]

    Indeed, suppose that $u$ is not simple. Then there exists $m>1$ such that the map
    \[
        \check{v}:\mathbb{R}\times[0,1]\to M, \qquad \check{v}(s,t)=u(s/m,t/m),
    \]
    descends to a cylinder map $v:\mathbb{R}\times S^1\to M$. The limits of $v$ at $\pm\infty$ agree with those of $u$. Hence, the extensions to $\mathbb{CP}^1$ satisfy
    \[ \bar{u}=\bar{v}\circ\phi \]
    where $\phi:\mathbb{CP}^1\to\mathbb{CP}^1$ is given by $\phi(z)=z^m$. Therefore, we have
    \[ [\bar{u}]=m[\bar{v}] \in H_2(M;\mathbb{Z}).\]
    This is impossible since $[\bar{u}]$ cannot be represented as a multiple of another class. Hence every solution is simple, and the conclusion follows by applying \cite[Lem.~7.8]{FHS95}.
\end{proof}

As in the case of $J$-holomorphic spheres, an $\omega$-compatible almost complex structure will be called \textit{regular for a Hamiltonian $H^\prime$}, or simply \textit{$H^\prime$-regular}, if the associated linearised operator is surjective for Floer cylinders in $\mathcal{M}(p,\Sigma;J,H',A)$. For $H'$-regular $J$, the moduli space $\mathcal{M}(p,\Sigma;H^\prime,J,A)$ is a smooth $2$-dimensional manifold carrying a free smooth $\mathbb{C}^*$-action, and the moduli space of unparametrised Floer cylinders
\[
\widehat{\mathcal{M}}(p, \Sigma;J,H^\prime,A) := \mathcal{M}(p, \Sigma;J,H^\prime,A) / \mathbb{C}^*
\]
is a smooth $0$-dimensional manifold, i.e., a discrete collection of points.

As a consequence of Lemma \ref{Lemma: Somewhere injectivity of Floer spheres}, the notion of $H^\prime$-regular is generic on $\mathcal{J}(p,\Sigma)$, in the sense that the set of $H^\prime$-regular $J$ is residual in $\mathcal{J}(p,\Sigma)$. However, compactness and non-emptiness of the moduli space have not yet been established.  Both properties essentially follow from the arguments of \cite{HV92}. Nevertheless, our moduli space is slightly different from the one considered there, where Floer equations for $(H,J)$ are considered with $\mathbb{R}$-parametrised $J$ satisfying only the assumption (J-1). We therefore establish both properties as the following proposition.

\begin{proposition}
\label{Proposition: moduli space of Floer spheres}
Let $H$ be a Hofer--Zehnder admissible Hamiltonian compactly supported in the interior $\operatorname{int}(W)$. With respect to the Darboux chart $\iota$ in \eqref{Equation: Darboux Chart i} centered at the minimum $p$, let $H'$ be the
perturbation of $H$ as in \eqref{Equation: C0-perturbation inside Darboux Ball}, and let $J\in\mathcal{J}(p,\Sigma)$. Then, the moduli space $\widehat{\mathcal{M}}(p,\Sigma;J, H^\prime, A)$ is compact. Moreover if $J$ is $H^\prime$-regular, then $\widehat{\mathcal{M}}(p,\Sigma;J, H^\prime, A)$ is a smooth compact $0$-dimensional manifold satisfying
\[
\#_2 \widehat{\mathcal{M}}(p,\Sigma;J, H^\prime, A) = 1.
\]
\end{proposition}

Since $J$ is automatically $H^\prime$-regular whenever $\mathcal{M}(p,\Sigma;J,H^\prime,A)$ is empty, the proposition implies the
existence of a Floer cylinder for every $J \in \mathcal{J}(p,\Sigma)$:
\begin{corollary}
\label{Corollary: existence of Floer spheres}
In the setting of Proposition~\ref{Proposition: moduli space of Floer spheres}, the moduli space $\widehat{\mathcal{M}}(p,\Sigma;J,H^\prime,A)$ is non-empty for every $J\in\mathcal{J}(p,\Sigma)$, not necessarily $H^\prime$-regular.
\end{corollary}

\begin{proof}[Proof of Proposition \ref{Proposition: moduli space of Floer spheres}]
We first prove the compactness of $\widehat{\mathcal{M}}(p,\Sigma;J,H^\prime,A)$ following Gromov-Floer compactification. Denote by $\lbrace u_n \rbrace_{i \in \mathbb{N}}$ a sequence of Floer cylinders with normalised as
\begin{equation}
\label{Equation: Normalisation for bottom building}
u_n(0,0) \in \partial U_0 \qquad \text{and} \qquad u_n(\mathbb{R}_{<0}\times S^1) \subset U_0,
\end{equation}
where $U_0$ is a Darboux ball centered at $p$ where $H'$ is quadratic.

Since the $L^2$-energy \eqref{Equation: L2-energy of Floer spheres} of each $u_n$ is strictly smaller than $\omega(A)$, bubbling cannot occur along the sequence. Consequently, 
up to passing to a subsequence, $\{u_n\}$ converges in the $C^\infty_{\mathrm{loc}}$-topology to a Floer cylinder $u_\infty$ satisfying
\[
u_\infty(0,0) \in \partial U_0\qquad \text{and} \qquad u_\infty(\mathbb{R}^{<0}\times S^1) \subset \overline{U}_0.
\]
Moreover, the $L^2$-energy of $u_\infty$ is finite, as given by
\[
E(u_\infty) \le \liminf_{i\to \infty} E(u_i) = \omega(A) + H^\prime (p).
\]

Finiteness of the $L^2$-energy implies that for any sequence $\{s_i\}_{i\in\mathbb{N}}$ converging either to $+\infty$ or $-\infty$, there exists a subsequence such that the loop
\begin{equation}
\gamma_i:S^1 \to M\qquad \gamma_i(t):=u_\infty(s_i, t)
\end{equation}
converges to a $1$-periodic orbit of $X_{H^\prime}$. For a sequence converging to $-\infty$, since the critical point $p$ is the only $1$-periodic orbit of $X_{H^\prime}$ in the closure of the Darboux ball $\bar{U}_0$, the only possible limit is the minimum $p$. Hence, we conclude that $u_\infty$ converges uniformly to $p$ at $-\infty$:
\begin{equation}
\label{Equation: negative asymptotic of limit curve}
    \lim_{s\to-\infty}u_\infty(s,t)=p.
\end{equation}
Since $u_\infty(0,0) \in \partial U$, we also conclude that the limit is nonconstant. 

\medskip

Next, we consider the asymptotic limit at $+\infty$. For each $u_n$ converges to some point in $\Sigma$, there exists $\sigma_n \in \mathbb{R}$ such that 
\begin{equation}
\label{Equation: Normalisation for top building}
u_n(\sigma_n,\tau_n) \in \partial U_\infty\qquad \text{and} \qquad u_n(\mathbb{R}_{> \sigma_n}\times S^1) \subset U_\infty.
\end{equation}
holds for some $\tau_n \in S^1$. The constant $\sigma_n$ is positive since $U_0$ and $U_\infty$ does not intersect, but may diverge to $+\infty$ as $n\to\infty$.

Suppose that $\sigma_i$ has a finite limit. Then, for sufficiently large $\sigma_0$, the restriction of the Floer cylinder $u_n|_{s \ge \sigma_0}$ extends to a $J$-holomorphic disk $w_n$ with $w_n(0)\in\Sigma$. We can set the extensions to have a common domain. More precisely, we may take the upper hemisphere of $\mathbb{CP}^1$ bounded by the circle $\log|z|=\sigma_0$, or equivalently, a disk of radius $1/ \log (|\sigma_0|)$ in $\mathbb{C}$ equipped with the complex structure $-i$.

The $L^2$-energy of $w_n$, defined analogously to \eqref{equation: L2-energy of the holomorphic curves}, is bounded above by $E(u_n)$, and hence is strictly smaller than $\omega(A)$. Therefore, bubbling can only occur at boundary points, i.e., points corresponding to $\{ \sigma_0 \} \times S^1$ for $u_n$. Since we have already shown that bubbling cannot occur for Floer cylinders, after passing to a subsequence, $\{w_n \}$ converges to a $J$-holomorphic disk $w_\infty$ in the $C^\infty$-topology. This implies
\[ u_\infty(+\infty) = w_\infty(0) \in \Sigma,\]
as $w_\infty(0)$ is the limit of a subsequence of points $w_n(0)$ that lies on $\Sigma$.

Therefore, when $\sigma_n$ has a finite limit, there exists a subsequence such that $u_n$ converges to a Floer cylinder $u_\infty$ in the $C^\infty_{\mathrm{loc}}$-topology, which also connects $p$ to $\Sigma$. Its extension $\bar{u}_\infty$ is the $C^0$-limit of the extensions $\bar{u}_n$ and hence represents the same homology class $A$. Consequently, we have
\[
    u_\infty\in\mathcal{M}(p,\Sigma;J,H',A).
\]

Hence, it remains to show that $\sigma_n$ does not tend to $+\infty$. The following lemma which corresponds to \cite[Thm.~4.3.(iii)]{HV92} is crucial for this step.
\begin{lemma}
\label{Lemma: Minimal escape energy}
    There exists a constant $E_->0$ such that every Floer cylinder $u\in \mathcal{M}(p,\Sigma;J,H^\prime,A)$ normalised as
    \begin{equation}
    \label{Equation: Proof of Minimal escape - 1}
        u(0,0)\in\partial U_0 \qquad\text{and}\qquad u(\mathbb{R}_{<0}\times S^1)\subset U_0
    \end{equation}
    satisfies
    \[
     E_- < \frac{1}{2}\int_{(-\infty,0] \times S^1}\left\|\frac{\partial u}{\partial s} \right\|^2 + \left\|\frac{\partial u}{\partial t}-X_{H^\prime}(u) \right\|^2\,{\rm ds\,dt}.
    \]
    
\end{lemma}
\begin{proof}
Assume that there exists a sequence of Floer cylinders $\{u_n\}\subset\mathcal{M}(p,\Sigma;J,H',A)$ such that the $L^2$-energy of half-cylinder
\[
E_n:=\frac{1}{2}\int_{(-\infty,0] \times S^1}\left\|\frac{\partial u_n}{\partial s} \right\|^2 + \left\|\frac{\partial u_n}{\partial t}-X_{H^\prime}(u_n) \right\|^2\,{\rm ds\,dt}
\]
tends to $0$, where each $u_n$ is normalised as \eqref{Equation: Proof of Minimal escape - 1}. 

Let $u_\infty$ be the $C^\infty_{\rm loc}$-limit of $\{u_n\}$. From the
discussion above, the limit satisfies
\[
u_\infty(0,0) \in \partial U_0\qquad \text{and} \qquad u_\infty(\mathbb{R}^{<0}\times S^1) \subset \overline{U}_0.
\]
and, in particular,
\[
\lim_{s\to-\infty} u_\infty(s,t) = p.
\]
Moreover, the $L^2$-energy of the negative half-cylinder $u_\infty|_{(-\infty,0]\times S^1}$ vanishes:
\[
\frac{1}{2}\int_{(-\infty,0]\times S^1}
\left\|\frac{\partial u_\infty}{\partial s}\right\|^2
+
\left\|\frac{\partial u_\infty}{\partial t}-X_{H^\prime}(u_\infty)\right\|^2
\,{\rm d}s\,{\rm d}t
\le
\liminf_{n\to\infty}E_n
=0.
\]

Since Floer cylinder with vanishing $L^2$-energy is constant in the
$\mathbb{R}$-coordinate, this contradicts
\[
u_\infty(0,0)\in\partial U
\qquad\text{and}\qquad
\lim_{s\to-\infty}u_\infty(s,t)=p.
\]
This completes the proof.
\end{proof}

Similarly, there exists a $E_+>0$ such that every $u \in \mathcal{M}(p,\Sigma;J,H',A)$ normalised as
\begin{equation}
\label{Equation: Proof of Minimal escape - 2}
    u(0,0) \in \partial U_\infty \qquad \text{and} \qquad u(\mathbb{R}_{>0}\times S^1) \subset U_\infty.
\end{equation}
satisfies
\[
     E_+ < \frac{1}{2}\int_{[0,\infty) \times S^1}\left\|\frac{\partial u}{\partial s} \right\|^2 + \left\|\frac{\partial u}{\partial t}-X_{H^\prime}(u) \right\|^2\,{\rm ds\,dt}.
\]
We then take
\[
E_\infty=\min\{E_-,E_+\}>0.
\]

Together with this constant, we prove the remaining step for the compactness.

\medskip

\noindent \textbf{Claim.} Sequence $\{ \sigma_n\}$ does not tends to $+\infty$.

\begin{proof}[Proof of Claim]
Assume the contrary. Passing to a subsequence, we may assume that
\[
\sigma_1<\sigma_2<\cdots
\]
with $\sigma_n\to+\infty$. Since the $L^2$-energy
\[
\frac{1}{2}\int_{[0, \sigma_n]\times S^1}\left\|\frac{\partial u_n}{\partial s} \right\|^2 + \left\|\frac{\partial u_n}{\partial t}-X_{H^\prime}(u_n) \right\|^2\,{\rm ds\,dt} \]
is bounded above as $\sigma_n\to+\infty$, there exists a sequence $\{ s_n \in [0,\sigma_n] \}$ such that loops
\[
\gamma_{n}:S^1 \to M\qquad \gamma_{n}(t):=u_n(s_n, t),
\]
satisfy
\[
\| \dot\gamma_n - X_H^{\prime}(\gamma_n) \|_{L^2} \to 0.
\]

After passing to a subsequence, $\gamma_n$ converges to a $1$-periodic orbit of $X_{H^\prime}$. By the Hofer--Zehnder admissibility of $H$, this orbit must be a constant orbit corresponding to a critical point $q$ of $H^\prime$. Note that the critical point $q$ may be degenerate and need not be uniquely determined, since it may depend on the choice of subsequence.

\medskip

For each $\gamma_n$, we define disk capping ${v}_{n}$ given by the completion of $u_n|_{\lbrace s \le s_n \rbrace}$ by adding $u_n(-\infty)=p$. 
Combining Lemma~\ref{Lemma: Minimal escape energy}, with the formula
\begin{equation}
\label{Equation: Proof of compactness - 1}
\mathcal{A}_{H^\prime}(\gamma_{n}, {v}_{n}) = \frac{1}{2}\int_{(-\infty, s_n]\times S^1}\left\|\frac{\partial u_n}{\partial s} \right\|^2 + \left\|\frac{\partial u_n}{\partial t}-X_{H^\prime}(u_n) \right\|^2\,{\rm ds\,dt} - (H^\prime(p)),
\end{equation}
the Hamiltonian action of $(\gamma_n, {v}_n)$ satisfies
\begin{equation}
\label{Equation: Proof of compactness - 2}
    E_\infty - H^\prime(p)< \mathcal{A}_{H^\prime}(\gamma_{n}, v_{n}) < \omega(A)-E_\infty.
\end{equation}

Finally, note that for sufficiently large $n$, one may assume that $\gamma_n$ admits a capping $v'_n$ contained in a Darboux ball centered at $q$, whose symplectic area is smaller than $E_\infty/2$. Moreover, we may assume that $n$ is sufficiently large so that
\[
|H^\prime (q) - H^\prime (\gamma_n(t)))|<E_\infty/2.
\]
for every $t\in S^1$. This implies that, for sufficiently large $n$, we have
\begin{equation}
\label{Equation: Proof of compactness - 3}
    \left|\mathcal{A}_{H^\prime}(\gamma_n, \check{v}_n)-\left(\omega([v_n \# v'_n])-H^\prime(p_\infty)\right)\right| < E_\infty
\end{equation}
Together with \eqref{Equation: Proof of compactness - 2}, we conclude that
\begin{equation}
\label{Equation: Proof of compactnes - 4}
    H'(p_\infty) -H^\prime(p) < \omega([v_n \# v'_n]) < \omega(A)+H^\prime(p_\infty).
\end{equation}

Since $p$ is the minimum of the non-positive Hamiltonian $H^\prime$, this
contradicts the $\omega$-minimality of the class $A$, as the sphere $v_n\#v'_n$ satisfies
\[
0<\omega([v_n\# v'_n])<\omega(A).
\]
This completes the proof of the Claim.
\end{proof}

Hence, the sequence $\{\sigma_n\}$ has no divergent subsequence, and this
proves the compactness of the moduli space $\widehat{\mathcal{M}}(p,\Sigma;J,H',A)$. Therefore, for $H^\prime$-regular $J$, we conclude that the moduli space is a smooth compact $0$-dimensional manifold, and its mod $2$ count is
well-defined. It remains to show that
\[
\#_2 \widehat{\mathcal{M}}(p,\Sigma;J, H^\prime, A) = 1.
\]

This follows the same idea as Lemma~\ref{Lemma: unparametrised moduli space}.
Let $\mathcal{J}=\lbrace J_{\tau} \rbrace_{\tau \in [0,1]}$ be a smooth homotopy of almost complex structures in $\mathcal{J}(p,\Sigma)$ from a regular almost complex structure $I\in\mathcal{J}(p,\Sigma)$ to the given $H^\prime$-regular almost complex structure $J$. We denote the associated parametrised moduli space by
\[
    \mathcal{M}(p,\Sigma;\mathcal{J}, H^\prime, A) \subset [0,1] \times W^{k,p}(\mathbb{R}\times S^1, M),
\]
consisting of pairs $(\tau,u)$ such that $u:\mathbb{R}\times S^1\to M$  satisfying the Floer equation
\[\partial_s u + J_\tau(u(s,t))\partial_t u = \tau \cdot \nabla H^\prime(u(s,t)),\]
with asymptotic conditions
\[
\lim_{s\to-\infty}u(s,t)=p,
\qquad
\lim_{s\to+\infty}u(s,t)\in\Sigma,
\]
and whose continuous extension $\bar u$ satisfies $[\bar{u}]=A$. For $\tau=0$, the Floer equation reduces to the Cauchy--Riemann equation for the almost complex structure $I$, while for $\tau=1$, it becomes the Floer equation associated with $(H^\prime,J)$.

The parametrised moduli space $\mathcal{M}(p,\Sigma;\mathcal{J},H^\prime,A)$ also admits a $\mathbb{C}^*$-action given by
\[
\lambda\cdot(\tau,u)=(\tau,\lambda\cdot u).
\]
The $\omega$-minimality in the case $\tau=0$, and Hofer-Zehnder admissibility of
$\tau H$ in the case $0<\tau \le 1$, implies the compactness of the moduli space of unparametrised curves
\[
\widehat{\mathcal{M}}(p,\Sigma;\mathcal{J},H^\prime,A)
= \mathcal{M}(p,\Sigma;\mathcal{J},H^\prime,A)/\mathbb{C}^*
\]
Here we use the fact that $p$ is the unique minimum of $\tau H'$ for every $\tau\in(0,1]$.

Moreover, for every $\tau\in(0,1]$, the Hamiltonian $\tau H'$ is regular at the critical point $p$, and Lemma~\ref{Lemma: Somewhere injectivity of Floer spheres} implies the existence of a somewhere injective regular point for every Floer cylinder at $\tau$. Thus, every curve $(\tau,u)\in\mathcal{M}(p,\Sigma;\mathcal{J},H',A)$ with $\tau \ne 0$ admits a somewhere injective regular point, while for $\tau=0$ it admits a somewhere injective point. This guarantees the surjectivity of the linearised operator for the parametrised moduli problem for a generic homotopy $\mathcal{J}$ from $I$ to $J$. Here we use the assumptions that $I$ is regular and $J$ is $H'$-regular.

Consequently, for generic $\mathcal{J}$, the parametrised moduli space is a smooth $1$-dimensional manifold with boundary, and its boundary is precisely 
\[
    \widehat{\mathcal{M}}(p,\Sigma;I, A) \sqcup \widehat{\mathcal{M}}(p,\Sigma;J,H^\prime, A).
\]

Since the mod $2$ count of boundary points of a compact $1$-dimensional manifold is always even, we conclude that
\[
\#_2\widehat{\mathcal{M}}(p,\Sigma;H^\prime,J,A) =\#_2\widehat{\mathcal{M}}(p,\Sigma;I,A) = 1,
\]
where the second equality follows from Lemma~\ref{Lemma: unparametrised moduli space}. This completes the proof.
\end{proof}

We note that proving compactness for a sequence $[(\tau_n,u_n)]\in \widehat{\mathcal{M}}(p,\Sigma;\mathcal{J},H',A)$ with
\[
\tau_n \searrow 0
\]
is slightly different from the compactness argument for $\widehat{\mathcal{M}}(p,\Sigma;J,H',A)$.  In particular, the negative asymptotic limit does not automatically converge to $p$ from the normalisation \ref{Equation: Normalisation for bottom building} alone for the limit curve. Nevertheless, the contradiction in this case follows from the same type of argument using the $\omega$-minimality of the line class for pseudoholomorphic curves. The argument is essentially the same as that in \cite[Thm.~3.4]{HV92}, and we omit the details here.

\noindent

\noindent \textbf{Apriori upper bound.} As a corollary of Proposition~\ref{Proposition: moduli space of Floer spheres},
we obtain an a priori bound for the Hofer--Zehnder capacity
$c_{\rm HZ}(W)$. This follows already from the monotonicity and the computation of $c_{\rm HZ}(D_{g_0}^* S^2)$ \cite{Bim23}; nevertheless, we state the result here for completeness.

\begin{corollary}
\label{Corollary: Apriori bound on HZ-capacity}
    Let $W$ be a fiberwise starshaped domain in $T^*S^2$, contained in the disk cotangent bundle $D^*_{g_0}S^2$ of the round sphere of radius $1$. Then the Hofer--Zehnder capacity $c_{\rm HZ}(W)$ satisfies
    \[
    c_{\rm HZ}(W) < 2\pi.
    \]
\end{corollary}
\begin{proof}
    Let $H$ be a Hofer--Zehnder admissible Hamiltonian on $W$. Regarding $W$ as a subdomain of $(\mathbb{CP}^1\times\mathbb{CP}^1, 2\omega_{\rm FS} \oplus 2\omega_{\rm FS})$, and $H$ as a Hamiltonian defined on $\cp^1 \times \cp^1$. Corollary \ref{Corollary: existence of Floer spheres} implies that, for any almost complex structure $J\in\mathcal{J}(p,\Sigma)$, there exists a Floer sphere $u\in\mathcal{M}(p,\Sigma;J,H',A)$. The $L^2$-energy of $u$ satisfies
    \[
    E_{H',J}(u)=\omega(A)+H'(p),
    \]
    and hence, we obtain
    \[
    -\min H<-\min H'<\omega(A)=2\pi.
    \]
    This proves the upper bound for $c_{\rm HZ}(W)$.
\end{proof}

\vspace{.5cm}

\noindent\textbf{Discussion of almost complex structures.} We finish the subsection with some discussion on almost complex structures. First, the assumption (J-1) ensures that the intersection of pseudoholomorphic/Floer spheres with the symplectic divisor $\Sigma$ is transverse and occurs precisely at a single point. Such condition has not played any role so far in the existence results for
pseudoholomorphic or Floer spheres, such as Corollary \ref{Corollary: existence of Floer spheres}. For example, in the proof of Proposition~\ref{Proposition: moduli space of Floer spheres}, the choice of the product complex structure $i\oplus i$ was not essential and can be replaced by an arbitrary almost complex structure $J$. There, we only used the fact that $H$ vanishes near $\Sigma$, so that Floer spheres locally define $J$-holomorphic disks near the positive end and Gromov compactness can be applied in this region. This condition will be used in the next step, where the fact that the intersection with $\Sigma$ occurs at a single point will play a crucial role. We also note that, in this case, the version of Lemma \ref{Lemma: Minimal escape energy} used to obtain the energy bound $E_+>0$ for the normalization \eqref{Equation: Proof of Minimal escape - 2} can be made explicit, as in the proof of \cite[Prop.~19]{Shelukhin22}.

\medskip
Next, assumption(J-2) can alternated as follows:
\begin{itemize}
    \item[(J-2')] $J$ agrees with a fixed $\omega_p$-compatible almost complex
    structure $I_p$ on $T_pM$.
\end{itemize}

Indeed, for any choice of Darboux chart around $p$, the pullback of $I_p$ defines an $\omega_0$-compatible almost complex structure $J_0$ on $\mathbb{R}^4$ with $I_0=\iota^*I_p$. By \cite[Prop.~2.5.4]{MS17},  there exists a linear symplectomorphism $\Psi:(\mathbb{R}^4,\omega_0)\to(\mathbb{R}^4,\omega_0)$ such that $\Psi^*I_0=i$.

After restricting to a sufficiently small $4$-ball $B^4(r)$, we may assume that $\Psi(B^4(r))$ is contained in the domain of the original Darboux chart. We can then compose the Darboux chart with $\Psi$ to obtain a new Darboux chart for which the pullback of $I_p$ at the origin is the standard complex structure $i$. For such a $J$ satisfying (J-2'), alternatively, one may choose the Darboux ball $U_0$ adapted to $J_p$ and make a $C^0$-perturbation of $H$.

Next, in order to apply the Sard--Smale theorem to the space of almost complex structures $J(p,\Sigma)$ and prove that a generic $J$ is regular (or $H^\prime$-regular), we need the space of almost complex structures to admit a Banach manifold structure. This can still be done under the requirement that a given almost complex structure agrees with a fixed almost complex structure on a prescribed subset  allowing us to achieve transversality under (J-1) and (J-2).

\newpage
\subsection{From Floer spheres in $\cp^1 \times \cp^1$ to Floer planes in $\widehat{W}$.}

In the second step, we state Proposition \ref{Proposition: Bottom level}, which describes the existence of a Floer plane in the symplectic completion $\widehat{W}$ arising as a limit of Floer spheres in $\mathbb{C}P^1\times\mathbb{C}P^1$. More precisely, we regard the closed contact hypersurface $(\partial W,\lambda_{\rm taut}|_{\partial W})$ as a closed contact manifold $(Y,\alpha)$. Then there exists a collar neighborhood $U_{\partial W}$ with the canonical identification
\[
    (U_{\partial W},\omega|_{U_{\partial W}}) \cong
    \left((-\varepsilon,\varepsilon)_r\times Y, d(e^r \cdot \alpha) \right).
\]
We always assume that $\varepsilon>0$ is sufficiently small so that $U_{\partial W}$ is disjoint from $\Sigma$, and disjoint from the support of the given Hofer--Zehnder admissible Hamiltonian.

For the almost complex structures, we impose the following additional condition:
\begin{itemize}
    \item[(J-3)] $J$ is SFT-like on the collar neighborhood $U_{\partial W}$.
\end{itemize}
Here, being SFT-like means that $J$ is invariant in the $r$-coordinate and, with respect to the splitting on the neck
\begin{equation}
\label{Equation: Spliting of the tangent space}
    T_{(r,p)}M
    \cong T_r\mathbb{R}\oplus T_pY
    \cong T_r\mathbb{R}\oplus \mathbb{R}\langle R_\alpha(p)\rangle\oplus\xi_p,
\end{equation}
satisfies
\begin{equation}
\label{Equation: SFT-like J}
    J(\partial_r)=R_\alpha,
    \qquad
    J({\xi})=\xi
\end{equation}
where $R_\alpha$ is the Reeb vector field and $\xi=\ker\alpha$ is the contact structure of $(Y,\alpha)$. We denote by $\mathcal{J}(p,Y,\Sigma)$ the subspace of compatible almost complex structures on $(M,\omega)$ satisfying all three conditions (J-1), (J-2), and (J-3).

\medskip

With this setup, we introduce the neck-stretching process, sketch how a
Floer plane may arise as a limit, and be used to prove Theorem
\ref{Theorem: Upper bound from index 3 - nondegenerate W}. For each real number $R>0$, take a smooth strictly increasing function
\[ \varphi_R:(-R-\varepsilon, R+\varepsilon)_a \to (-\varepsilon, \varepsilon)_r\] satisfying
\begin{equation*}
    \varphi_R(a) =\begin{cases}
        a+R & \text{ if } a \in (-R-\varepsilon, -R-\varepsilon/2),\\
        a-R & \text{ if } a \in (R+\varepsilon/2, R+\varepsilon).
    \end{cases}
\end{equation*}
Such a function defines a neck-stretched symplectic manifold and the
corresponding almost complex structure, which we denote by $(M_R,\omega_R)$
and $J_R$, respectively. More precisely, using the identification
$U_{\partial W}\cong(-\varepsilon,\varepsilon)\times Y$, we define $M_R$ as follows:
\[
M_R := \Big( M \setminus (-\varepsilon/2,\varepsilon/2) \times Y \Big) \bigcup_{\varphi_R} \Big( (-R-\varepsilon,\, R+\varepsilon)\times Y \Big),
\]
where the gluing is defined by identifying $U_{\partial W}\setminus\big((-\varepsilon/2,\varepsilon/2)\times Y\big)$ with the corresponding subset of $(-R-\varepsilon,R+\varepsilon)\times Y$ using the function $\varphi_R$. We denote the induced diffeomorphism by \[\Phi_R:M_R \to M.\]

We define the symplectic form $\omega_R$ on $M_R$ by pulling back $\omega$ by $\Phi_R$, i.e.,
\[
\omega_R:=\Phi_R^*\,\omega.
\]
The symplectic form $\omega_R$ then coincides with $\omega$ on $M\setminus U_{\partial W}$, and it is given by
\[
\omega_R = d\big(e^{\varphi_R(a)}\cdot \alpha \big) = e^{\varphi_R(a)} \left(d \alpha + \varphi^\prime _R(a) \cdot da \wedge \alpha\right),
\]
on the neck. Here, we regard $\alpha$ and $d\alpha$ as differential forms defined on the neck via the splitting of the tangent space \eqref{Equation: Spliting of the tangent space}. Under condition (J-3), we define an almost complex structure $J_R$ on the symplectic manifold $(M_R,\omega_R)$. More precisely, we define $J_R$ to agree with $J$ outside the neck
\[
M\setminus U_{\partial W}
\cong_{\Phi_R}
M_R\setminus\big((-R-\varepsilon,R+\varepsilon)\times Y\big),
\]
and extend it over the neck $(-R-\varepsilon,R+\varepsilon)\times Y$ using \eqref{Equation: SFT-like J}, i.e.,
\begin{equation}
\label{Equation: Almost complex structure on the neck}
    J_R(\partial_a)=R_\lambda,\qquad
    J_R|_{\xi}=J|_{\xi}.
\end{equation}
In particular, one readily checks that $J_R$ is $\omega_R$-compatible. We also denote by \[H_R:M_R\to\mathbb{R}\] the Hamiltonian obtained by extending $H'$ to $M_R$.

\medskip

By Corollary \ref{Corollary: existence of Floer spheres}, one can take a sequence of Floer spheres for $\left(H, (\Phi_n^{-1})^* J_n\right)$. This induces a sequence of Floer spheres for $(H_n, J_n)$ on $M_n$. After passing to a subsequence, with a normalisation \eqref{Equation: Normalisation for bottom building}, Floer spheres then converge to a Floer plane \[ u_\infty : \mathbb{R} \times S^1 \to \widehat{W}\] asymptotic to $p$ at $-\infty$, and converges to a closed Reeb orbit $\gamma$ at $\infty$. Then, as in Corollary~\ref{Corollary: Apriori bound on HZ-capacity}, the positivity of the $L^2$-energy $E(w_\infty)$ implies that
\[
-\min H < \mathcal{A}(\gamma).
\]

Note that apriori the the limit $u_\infty$ has a domain with possible non-empty set of punctures. Studying the full SFT limit, we will prove that this is not the case. Also, considering Fredholm index related to the curves in limit building, we can narrow down the conditions that $\gamma$ satisfies and  as a result we would prove Theorem \ref{Theorem: Upper bound from index 3 - nondegenerate W}.

\medskip

We now introduce moduli space problems arising as the SFT limits of the Floer spheres $\{u_n\}$. The domains of the curves appearing in a limit building are contained in one of the following symplectic manifolds:
\begin{itemize}
    \item the symplectic completion of $W$
    \[ \widehat{W} := W \cup_Y \big([0,\infty)\times Y\big), \]
    \item the symplectization of $Y$
    \[ \mathbb{R}\times Y, \]
    \item the symplectic completion of $V$
    \[ \widehat{V} := \big((-\infty,0]\times Y\big)\cup_Y V. \]
\end{itemize}
Using the collar neighborhood
\[
    U_{\partial W}\cong\big((-\varepsilon,\varepsilon)\times Y,\; d(e^a \alpha)\big),
\]
the symplectic forms on the completions $\widehat{W}$ and $\widehat{V}$ are obtained by extending $\omega|_W$ and $\omega|_V$ with $d(e^a \alpha)$ along the cylindrical ends. The symplectisation $\mathbb{R}\times Y$ is equipped with the symplectic form $d(e^a\alpha)$. We denote the symplectic forms by $\widehat{\omega}_\bigstar$ for $\bigstar\in\{V,W,Y\}$, and omit the subscript when it is clear from the context.

SFT-like almost complex structure $J$ on $(M,\omega)$ extends, as in \eqref{Equation: Almost complex structure on the neck}, to compatible almost complex structures $J_\bigstar$ on each of these symplectic manifolds. We also omit the subscript when it is clear from the context. Hamiltonian compactly supported in $\operatorname{int}(W)$ also extends to $\widehat{W}$. With such Hamiltonian and almost complex structures, we consider the following moduli spaces problems:
\begin{itemize}
    \item Problem (W): Floer spheres in $\widehat{W}$
    \item Problem (W-J): pseduoholomorphic spheres in $\widehat{W}$
    \item Problem (Y): pseudoholomorphic spheres in $\RR \times Y$
    \item Problem (V): pseudoholomorphic spheres in $\widehat{V}$.
\end{itemize}
Here, by spheres we mean that the domain is a punctured $\mathbb{CP}^1$ equipped with the canonical complex structure $i$.

\medskip

Before describing each of these moduli space problems, we make a few technical remarks. First, throughout this subsection, we assume that the Reeb flow $R_\alpha$ is nondegenerate. This assumption will be used at several points; for example, it ensures the uniqueness of asymptotic limits. See \cite{Siefring17} for an example showing that such uniqueness may fail in the case of a degenerate Reeb flow. Although the extended Hamiltonian necessarily possesses degenerate critical points, precise moduli space problem appearing among Problem (W) would still be a Fredholm problem since all the asymptotic limits are nondegenerate; see Lemma \ref{Lemma: Nondegeneracy for limit curves}.

The Fredholm theory underlying these moduli space problems is slightly different from that in the previous subsection due to the noncompactness of the symplectic manifolds. For example, one has to consider Sobolev spaces with exponential weight. We also emphasize that, for pseudoholomorphic spheres, the Fredholm operator is defined using the Banach manifold of unparametrized immersed surfaces.

In the following, we recall some fundamental properties of pseudoholomorphic and Floer spheres in symplectisation and completions. The study of pseudoholomorphic curves in these symplectic manifolds was initiated by Hofer \cite{Hofer93} and further developed in numerous works, for example \cite{HWZ-1,HWZ-3, HWZ03}; for the Floer case, see \cite{BO13}. In particular, we refer to \cite[Sec.~7.4]{AEK24} for Problem (W). We begin with Problem (W).

\newpage

\noindent \textbf{The problem (W)}: Let $H: \widehat{W} \to \mathbb{R}_{\le 0}$ be an extension of a Hofer-Zehnder admissible Hamiltonian compactly supported in the interior $\operatorname{int}(W)$. Let $U_{\rm min}$ be a Darboux ball centered at $p$, where $H$ attains its minimum value and is constant, and denote by
\[H':\widehat{W} \to \mathbb{R}_{\le 0}\] the $C^0$-perturbation supported in $U_{\rm min}$ as in \eqref{Equation: C0-perturbation inside Darboux Ball}. We consider finite energy punctured Floer curves in $\widehat{W}$ for $J$,
namely maps \[ u:\mathbb{R}\times S^1 \setminus \Gamma \to \widehat{W}, \] where $\Gamma$ is a finite subset of $\mathbb{R}\times S^1$, satisfying
\begin{equation}
    \label{Equation: Floer equation on W}
    \partial_s u + J(u(s,t))\partial_t u = \nabla H^\prime(u(s,t)).
\end{equation}
and 
\begin{equation}
\label{Equation: Finite Hofer-energy condition on W}
0<E(u)<\infty.
\end{equation}
Since $H'$ vanishes on the cylindrical end $[0,+\infty)_a\times Y$, the Floer equation reduces to the Cauchy--Riemann equation there. The finite energy condition is defined by the Hofer energy
\begin{equation}
\label{Equation: Hofer Energy in W}
E(u):=\sup_{\varphi} \left(\int_{\mathbb{R} \times S^1} \| \partial_s u \|_{J, \varphi}^2 \right)
\end{equation}
where the supremum is taken over all strictly increasing smooth functions
\begin{equation}
\label{Equation: Profile function on W}
    \varphi:(-\varepsilon,\infty) \to (-1,0),\quad\lim_{a \to \infty}\varphi(a)=0
\end{equation}
such that the induced symplectic form
\[
\omega_\varphi := d(e^\varphi \alpha)
\]
agrees with $\widehat{\omega}$ on $(-\varepsilon,-\varepsilon/2) \times Y \subset W$. The norm $\| \cdot \|_{J, \varphi}$ is induced by the metric $g_{J,\varphi}(\cdot ,\cdot) = \omega_\varphi(\cdot , J\cdot)$.

The finiteness of the Hofer energy implies that the positive asymptotic limit is either a $1$-periodic orbit of Hamiltonian vector field $X_{H'}$ or a closed Reeb orbit. More precisely, for a sequence of real numbers $\{s_i\}$ with $s_i\to \pm\infty$, the loops $ \gamma_i(t) = u(s_i, t) $
have a subsequence converging to one of the following:
\begin{itemize}
    \item a $1$-periodic orbit of $X_{H'}$ or
    \item a closed Reeb orbit on $Y$. In this case, for sufficiently large $i$, we have \[\gamma_i \in  \widehat{W}\setminus W.\] Writing $\gamma_i=(a_i,\bar{\gamma}_i)$,  there exists a closed Reeb orbit $\gamma_\infty$ and $t_0\in S^1$ such that
    \[
        \lim_{i\to\infty}a_i(t)=+\infty \quad\text{and}\quad
        \lim_{i\to\infty}\bar{\gamma}_i(t)=\gamma_\infty\bigl(\pm T_{\infty}(t-t_0)\bigr)
    \]
    where $T_\infty$ denotes the period of the Reeb orbit
    $\gamma_\infty$. 
\end{itemize}

In the first case, the limiting Hamiltonian orbit may depend on the choice of the sequence
$\{s_i\}$ with $s_i\to+\infty$ and of the subsequence, due to the degeneracy of the extended Hamiltonian $H'$.
We nevertheless use the notation $u(+\infty)\in \widehat{W}$ to indicate that the positive end remains in the compact part of $\widehat{W}$, without implying that
$u$ has a unique limit as $s\to+\infty$.\\

On the other hand, in the second case, the asymptotic Reeb orbit is uniquely determined as a consequence of the nondegeneracy of the Reeb flow. The choice of the sequence ${s_i}$ also does not affect the limit. More precisely, writing $u=(a,\bar u)$ on a domain near $+\infty$, we have, 
\begin{equation}
\label{Equation: Description at positive puncture for problem W}
\begin{aligned}
    \lim_{s \to +\infty} a(s,t) = +\infty\quad \text{and} \quad
    \lim_{s \to +\infty} \bar{u}(s,t) =\gamma_\infty (T_{\infty}(t+t_0)).
\end{aligned}
\end{equation}
We denote this case by \[ u(+\infty)=\gamma_\infty, \] and refer to $+\infty$ as a \textit{positive puncture}.

At each puncture $z\in\Gamma$, identifying $\mathbb{R}\times S^1$ with $\mathbb{C}^*$ in the usual way, the loops \[\gamma_{z,s}(t)=u(z+ \exp(-2 \pi (s+it))\] either converge to a point in $\widehat{W}$ or to a Reeb orbit $\gamma_z$ in the sense of \eqref{Equation: Description at positive puncture for problem W}, as $s \to -\infty$. We refer to these two cases as a \textit{removable puncture} and a \textit{positive puncture}, respectively, and denote the latter case by $u(z)=\gamma_z$. We note that, unlike the discussion at $+\infty$, the limit at a removable puncture is well-defined, and the Floer curve extends smoothly over the puncture and satisfies the Floer equation at $z$ as well.

The Hofer energy can be written as
\begin{equation}
\label{Equation: Hofer energy on W in action}    
E(u)= \mathcal{A}_{H'}(u(+\infty))-\mathcal{A}_{H^\prime}(u(-\infty))+\sum_{z \in \Gamma^+} \mathcal{A}(\gamma_z)
\end{equation}
where $\Gamma^+$ is the set of positive punctures. The action $\mathcal{A}(\gamma_z)$ is the period of $\gamma_z$. The Hamiltonian action at infinity is defined as
\begin{equation}
\label{Equation: Hamiltonian action at +infty}
    \mathcal{A}_{H'}(u(+\infty)) = \begin{cases}
    \mathcal{A}_{H'}(p_\infty) &\text{if $u(+\infty) \in \widehat{W}$}\\
    \mathcal{A}(\gamma_\infty) &\text{if $u(+\infty)=\gamma_\infty$}.
    \end{cases}
\end{equation}
The action $\mathcal{A}_{H'}(u(-\infty))$ is defined similarly, except that
in the second case it is given by $-\mathcal{A}(\gamma_\infty)$. Here, the limit $p_\infty$ in the case $u(+\infty)\in\widehat{W}$ may not be well-defined, but note that the Hamiltonian action of any limit along a subsequence is well-defined.

We note that the Hamiltonian action functional depends on the choice of
the function $\varphi$, while by exactness it is independent of the choice
of a capping. It can be written as
\[
    \mathcal{A}_{H'}(\gamma) = \int_{\gamma} e^\varphi \cdot \alpha - \int_{S^1} H'(\gamma(t)).
\]
Nevertheless, the Hofer--Zehnder admissibility of $H$ implies that every $1$-periodic orbit of $H'$ is constant, and hence
\[
    \mathcal{A}_{H'}(p_\infty)=-H'(p_\infty),
\]
which is independent of the choice of $\varphi$. The Hamiltonian action of the asymptotic orbit is also consistent with this expression, since $\varphi(\infty)=0$ with vanishing $H'$.

\medskip

We now focus on the case when
\begin{equation*}
    \quad u(+\infty)=\gamma, \quad u(-\infty) = p , \quad \text{and}\quad 
    \Gamma = \varnothing.
\end{equation*}
In this case, the convergence to $p$ is indeed a limit, since $p$ is nondegenerate, i.e., \[ \lim_{s \to -\infty} u(s,t)=p. \] We denote by $\mathcal{M}(p,\gamma;J,H')$ the moduli space of these finite-energy Floer cylinders.\

\medskip

\noindent\textbf{Fredholm index.} As in Problem (M), the moduli space admits a natural $\mathbb{C}^*$-action, but in this case the action need not be free. More precisely, there may exist $\tau\in S^1$ such that
\[ u(s,t)=u(s,t+\tau) \]
for every $(s,t)\in\mathbb{R}\times S^1$. In this case, $\tau$ can be chosen to be a rational number $1/m$ for some natural number $m>1$, and $\gamma$ is an $m$-fold cover of a closed Reeb orbit $\gamma'$, which is not necessarily a prime orbit. Then the curve $v:\RR \times S^1 \to \widehat{W}$ defined by \[ v(s,t)=u\left( \frac{s}{m},\frac{t}{m}\right) \] defines a cylinder map which also solves the Floer equation \eqref{Equation: Floer equation on W} but for a Hamiltonian $H'/m$ with asymptotic conditions
\[
    v(+\infty)=\gamma' \quad \text{and} \quad v(-\infty) = p.
\]
We say that $u$ is multiply covered in this case, and otherwise call it \textit{simple}, as in Problem (M). Recall that such a case does not occur for curves in $\mathcal{M}(p,\Sigma;J,H',A)$, since the line class $A$ cannot be written as a multiple of another homology class.

Denote the subset of simple Floer cylinder by 
\[\mathcal{M}^*(p,\gamma;J,H') \subset \mathcal{M}(p,\gamma;J,H').\]
The virtual dimension of the moduil space is given as
\begin{equation}
    \operatorname{vdim}\mathcal{M}^*(p,\gamma;J, {H^\prime})
    = \mu_{\mathrm{CZ}}(\gamma)+1 - \mu_{\mathrm{CZ}}(p)= \mu_{\mathrm{CZ}}(\gamma)-1.
\end{equation}
Here, we have used the fact that \[ c_1(\widehat{W}) = 0 \quad \text{and} \quad \chi(\mathbb{R}\times S^1)=0.\] The number $+1$ comes from the freedom of $t_0$ in \eqref{Equation: Description at positive puncture for problem W}. 

Lemma \ref{Lemma: Somewhere injectivity of Floer spheres} also applies in this setting. More precisely, for every simple Floer plane $u\in\mathcal{M}^*(p,\gamma;J,H')$ for a closed Reeb orbit $\gamma$, there exists a neighborhood $U=U(H',J)$ of $p$ whose inverse image $u^{-1}(U) \subset \mathbb{R}\times S^1$ contains regular somewhere injective points. Therefore, for a generic $J\in\mathcal{J}(p,Y,\Sigma)$, the moduli space $\mathcal{M}^*(p,\gamma;J,H')$ is a smooth manifold of dimension $\mu_{\rm CZ}(\gamma)-1$ with a smooth free $\mathbb{C}^*$-action. We say that $J\in\mathcal{J}(p,Y,\Sigma)$ is \textit{$H'$-regular on $\widehat{W}$} if, for every Reeb orbit $\gamma$ and every simple Floer plane $u\in\mathcal{M}^*(p,\gamma;J,H')$, the linearised operator $D_u$ is surjective.

Since the contact form $\lambda$ is nondegenerate, there are countably many choices of $\gamma$. Therefore, since a countable intersection of residual sets is still residual, the set of $H'$-regular almost complex structures forms a residual subset of $\mathcal{J}(p,Y,\Sigma)$. For $H'$-regular $J$, the moduli space $\mathcal{M}^*(p,\gamma;J,H')$ is non empty only if
\begin{equation}
\label{Equation: Obstruction on simple Floer cylinders in W - 1}
    \mu_{\rm CZ}(\gamma)-1 \ge 2.    
\end{equation}

We also note that Lemma~\ref{Lemma: Minimal escape energy} also holds for normalisation given by \eqref{Equation: Normalisation for bottom building}. From the energy identity \eqref{Equation: Hofer energy on W in action}, we conclude that the moduli space is non empty only if
\begin{equation}
\label{Equation: Obstruction on simple Floer cylinders in W - 2}
\mathcal{A}(\gamma) > -\min H'.
\end{equation}
This conclusion does not require the curve to be simple.

\medskip

As in the case of Floer spheres in $(M, \omega)$, the following proposition gives the existence of a finite-energy Floer plane in $\widehat{W}$ for every $J\in\mathcal{J}(p,Y,\Sigma)$,  without assuming that $J$ is $H'$-regular.

\begin{proposition}
\label{Proposition: Bottom level}
    For every $J \in \mathcal{J}(p,Y,\Sigma)$, there exists a finite-energy Floer plane asymptotic to the minimum $p$ at $-\infty$ and to a closed Reeb orbit at $+\infty$. More precisely, there exists a closed Reeb orbit $\gamma$ such that
    \[
        \mathcal{M}(p,\gamma;J,H')\neq\varnothing.
    \]
\end{proposition}

The proof will be given in the section~\ref{section: SFT-limit}. As in the sketch, such a
Floer plane arises as a limit of a sequence of Floer spheres
\[
    u_n\in\mathcal{M}(p,\Sigma;J_n,H_n,A_n)
\]
parametrised by the normalisation \eqref{Equation: Normalisation for bottom building}, although the limiting curve may apriori have punctures. Later, in the proof of Theorem \ref{Theorem: Upper bound from index 3 - nondegenerate W}, we will in fact show that, for generic $J\in\mathcal{J}(p,Y,\Sigma)$, the resulting curve is simple and its asymptotic Reeb orbit $\gamma$ has Conley--Zehnder index $3$. In other words, for generic $J$, we obtain
\[
    \mathcal{M}^*(p,\gamma;J,H')\neq\varnothing.
\]

To prove the proposition, or equivalently, to describe the full SFT-limit building, we now review the remaining moduli problems satisfied by the limit curves, with Problem (W) having been discussed above. The equations under consideration are genuine Cauchy--Riemann equations, since the Hamiltonian $H'$ vanishes outside $W$, while the Hamiltonian term is scaled out during the bubbling analysis.

\vspace{.5cm}

\noindent \textbf{The problem (W-J)}: Let $J$ be an almost complex structure on $\widehat{W}$ that is SFT-like on the cylindrical end $\mathbb{R}_{\geq 0}\times Y$. Consider finite-energy punctured $J$-holomorphic spheres in $\widehat{W}$, namely maps \[ u:\cp^1 \setminus \Gamma \to \mathbb{R}\times \widehat{W}, \] where $\Gamma$ is a finite subset of $\cp^1$, satisfying
\begin{equation}
    \label{Equation: J-holomorphic equation on W-J}
    du + J \circ du \circ i = 0,
\end{equation}
and
\begin{equation}
\label{Equation: Finite Hofer-energy condition on W-J}
0<E(u)<\infty.
\end{equation}
As in problem (W), the finite energy condition is defined in terms of the Hofer energy
\begin{equation}
\label{Equation: Hofer Energy in W-J}
E(u):=\sup_{\varphi} \left(\frac{1}{2} \int_{\cp^1 \setminus \Gamma} \| d u \|_{J, \varphi}^2 \right)
\end{equation}
where the supremum is taken over all strictly increasing smooth functions
\[\varphi:[-\varepsilon,\infty) \to (-1,0), \quad \lim_{a \to \infty} \varphi(a)=0\]
such that the induced symplectic form $\omega_\varphi := d(e^\varphi \alpha)$. The operator norm $\|du\|_{J,\varphi}$ is taken with respect to the induced metric $g_{J,\varphi}(\cdot,\cdot)=\omega_\varphi(\cdot,J\cdot)$.

\medskip

As the behaviors near punctures for Problem (W), the finite energy condition implies that one of the following mutually exclusive cases holds on each puncture $z \in \Gamma$:
\begin{itemize}
    \item Positive puncture: Every sequence $\{z_n\}_{n\in\mathbb{N}}$ converging to $z$ eventually leaves every compact subset of $\widehat{W}$. Equivalently, writing $u=(a,\bar{u})$, we have
    \[ \lim_{z'\to z}a(z')=+\infty. \]
    \item Removable puncture: Every sequence
    $\{z_n\}_{n\in\mathbb{N}}$ converging to $z$ is eventually contained in
    a compact subset of $\widehat{W}$.
\end{itemize}
In the positive puncture case, the map $u$ converges to a closed Reeb orbit $\gamma_z$ on $Y$ as in \eqref{Equation: Description at positive puncture for problem W}. In the removable puncture case, $u$ converges to a point in $\widehat{W}$ and extends smoothly over $z_0$ as a $J$-holomorphic curve by the removal of singularities.

By the exactness of the symplectic form, there are no nonconstant $J$-holomorphic planes without positive punctures. Also the Hofer energy can be written as
\begin{equation}
\label{Equation: Hofer energy on W-J in action}    
E(u)= \sum_{z \in \Gamma^+} \mathcal{A}(\gamma_z)
\end{equation}
where $\Gamma^+$ is the set of positive punctures.

\medskip

\noindent\textbf{Fredholm index.} As in Problem (W), a finite energy $J$-holomorphic curve $u$ may be multiply covered. Nevertheless, for a generic $J$, the linearized Fredholm operator associated to every simple $J$-holomorphic curve is surjective. Consequently, the moduli space of simple $J$-holomorphic curves with the same set of punctures and the same asymptotic conditions is a smooth manifold whose dimension at $u$ is equal to the Fredholm index
\begin{equation}
\label{Equation: Fredholm index of J-holomorphic curve in W}
    \operatorname{Ind}(u) = \sum_{z\in \Gamma^+} \mu_{\rm CZ}(\gamma_z) - \# \Gamma^+,
\end{equation}
where $\Gamma^+$ denotes the set of positive punctures. 

In the case of a single positive puncture, i.e., for $J$-holomorphic planes, we have
\begin{equation}
\label{Equation: Obstruction of J-holomorphic plane in W - 1}
    \mu_{\rm CZ}(\gamma_\infty)\geq 1
\end{equation}
for generic $J$. Indeed, for a simple $J$-holomorphic plane, the Fredholm index is non-negative for generic $J$, which gives \eqref{Equation: Obstruction of J-holomorphic plane in W - 1}. Every non-simple $J$-holomorphic plane is a multiple cover of a simple $J$-holomorphic plane. Its asymptotic orbit is therefore a multiple cover of a Reeb orbit with positive Conley--Zehnder index, and hence the same conclusion holds for non-simple planes.

\newpage

\noindent \textbf{The problem (Y)}: Let $J$ be a $\mathbb{R}$-invariant almost complex structure on $\mathbb{R}\times Y$. Consider finite energy punctured $J$-holomorphic sphere in $\RR \times Y$, namely maps \[ u:\cp^1 \setminus \Gamma \to \mathbb{R}\times Y, \] where $\Gamma$ is a finite subset of $\cp^1$, satisfying
\begin{equation}
    \label{Equation: J-holomorphic equation on Y}
    du + J \circ du \circ i = 0,
\end{equation}
and
\begin{equation}
\label{Equation: Finite Hofer-energy condition on Y}
0<E(u)<\infty
\end{equation}
Here, the finite energy condition is defined in terms of the Hofer energy
\begin{equation}
\label{Equation: Hofer Energy in Y}
E(u):=\sup_{\varphi} \left(\frac{1}{2} \int_{\cp^1 \setminus \Gamma} \| d u \|_{J, \varphi}^2 \right)
\end{equation}
where the supremum is taken over all strictly increasing smooth functions
\[\varphi:(-\infty,\infty) \to (-\infty,0)\]
such that the induced symplectic form $\omega_\varphi := d(e^\varphi \alpha)$. The operator norm $\|du\|_{J,\varphi}$ is taken with respect to the induced metric $g_{J,\varphi}(\cdot,\cdot)=\omega_\varphi(\cdot,J\cdot)$.

Writing $u = (a, \bar{u})$, the finite Hofer energy condition implies that one of the following mutually exclusive cases holds on each puncture $z \in \Gamma$:
\begin{itemize}
    \item positive puncture: $\lim_{z' \to z} a(z') = +\infty$
    \item negative puncture: $\lim_{z' \to z} a(z') = -\infty$
    \item removable puncture: $\lim_{z' \to z} a(z') = a(z)$ exists in $\RR$.
\end{itemize}
In the third case, one can show that $u$ smoothly extends over $z$. For the first case, there exists a closed Reeb orbit $\gamma_{z}$ with the convergence is given in the sense of \eqref{Equation: Description at positive puncture for problem W}. The negative puncture case is similar, and the convergence is described by
\begin{equation}
\label{Equation: Description at negative puncture for problem Y}
    \lim_{s \to -\infty} \bar{u}\left(z_0+ \exp ({2\pi (s+i t})\right) = \gamma_{z_0} (T_{\gamma_{z_0}}(t+t_0))~\text{ for some }t_0\in S^1,
\end{equation}
where the addition is taken with respect to a holomorphic coordinate on $\mathbb{C}P^1$ near $z$. Similarly, the Hofer energy can be written as
\begin{equation}
\label{Equation: Hofer energy on Y in action}    
    E(u)= \sum_{z \in \Gamma^+} \mathcal{A}(\gamma_z)
\end{equation}
where $\Gamma^+$ is the set of positive punctures. For the symplectisation $\mathbb{R}\times Y$, there is another type of $L^2$-energy, called the $d\alpha$-energy. Although $d\alpha$ is not symplectic, it still defines a pseudonorm $d\alpha(\cdot,J\cdot)$ on tangent spaces. We define the $d\alpha$-energy by
\[
E_{d\lambda}(u) := \frac{1}{2}\int_{\cp^1 \setminus \Gamma} \| d u \|^2.
\]
By Stokes' theorem, we have
\begin{equation}
\label{Equation: dlambda energy on Y in action}
    E_{d\alpha}(u)
    =\sum_{z\in\Gamma^+}\mathcal{A}(\gamma_z)
    -\sum_{z\in\Gamma^-}\mathcal{A}(\gamma_z),
\end{equation}
where $\Gamma^-$ is the set of negative punctures. In particular, the non-negativity of the $d\alpha$-energy implies that either $\Gamma^+$ is non-empty or both $\Gamma^\pm$ are empty. The latter case is again excluded by the exactness of $\omega_\varphi$, and we conclude that $\Gamma^+$ is never empty.

We note that the $d\alpha$-energy vanishes if and only if the tangent planes of the curve $u$ are contained in $T_r\mathbb{R}\oplus\mathbb{R}\langle R_\alpha(p)\rangle$ with respect to the splitting \eqref{Equation: Spliting of the tangent space}, i.e.,
\[
    \pi_\xi\circ du=0,
\]
where
\[
    \pi_\xi:T_{(a,p)}(\mathbb{R}\times Y)\to\xi_p
\]
is the projection onto the contact plane with respect to this splitting. In this case, the image of the curve is contained in a trivial cylinder
\[
    u_\gamma(s,t)=(Ts,\gamma(Tt))
\]
over a prime closed Reeb orbit $\gamma$ of period $T$. In other words, $u$ is a branched cover of a trivial cylinder $u_\gamma$: there exists a holomorphic branched covering map
\[
    \Phi:(\cp^1\setminus\Gamma , i )\to (\mathbb{R}\times S^1, i)
\]
such that
\[
    u=u_\gamma\circ\Phi.
\]

Again, note that $J$-holomorphic curves $u$ with finite Hofer energy may be multiply covered, and may even be a branched cover of a trivial cylinder
$u_\gamma$. Nevertheless, for a generic choice of SFT-like almost complex structure on $\mathbb{R}\times Y$, the linearised Fredholm operator associated to every simple $J$-holomorphic curve is surjective. Consequently, the moduli space of simple $J$-holomorphic curves with the same set of punctures and the same asymptotic conditions is a smooth manifold whose dimension at $u$ is equal to the Fredholm index
\begin{equation}
\label{Equation: Fredholm index of J-holomorphic curve in Y}
    \operatorname{Ind}(u) = \sum_{z\in \Gamma^+} \mu_{\rm CZ}(\gamma_z) - \sum_{z \in \Gamma^-} \mu_{\rm CZ}(\gamma_z) -2+ \# \Gamma.
\end{equation}

If $u_\gamma$ is the trivial cylinder over a prime closed Reeb orbit, then its Fredholm index is zero. Up to the automorphisms of the cylinder, the trivial cylinder is the only simple $J$-holomorphic curve for which every $\mathbb{R}$-translation represents the same unparametrised curve. Hence, in all other cases, $\mathbb{R}$-translation of a simple $J$-holomorphic curve $u$ produces a distinct $J$-holomorphic curve in $\mathbb{R}\times Y$, and therefore
\begin{equation*}
    \operatorname{Ind}(u) \ge 1.
\end{equation*}
If there is precisely one positive puncture, namely $\Gamma^+=\{\infty\}$, then the index can be rewritten as
\begin{equation}
\label{Equation: Fredholm index of J-holomorphic curve in Y - single + pucture}
    \operatorname{Ind}(u) = (\mu_{\rm CZ}(\gamma_\infty)-1) - \sum_{z \in \Gamma^-}(\mu_{\rm CZ}(\gamma_z) -1).
\end{equation}

\medskip

\noindent \textbf{The problem (V)}: Let $J$ be a complex strucure on $\widehat{V}$ that agrees with product complex structure $i \oplus i$ near $\Sigma$ and that is $\mathbb{R}$-invariant on the cylindrical ends
\[
(-\infty, \epsilon)\times Y \subset \widehat{V}.
\]
Consider finite energy punctured $J$-holomorphic sphere in $\widehat{V}$, namely maps \[ u:\cp^1 \setminus \Gamma \to \widehat{V}, \] where $\Gamma$ is a finite subset of $\cp^1$, satisfying
\begin{equation}
    \label{Equation: J-holomorphic equation on V}
    du + J \circ du \circ i = 0,
\end{equation}
and
\begin{equation}
\label{Equation: Finite Hofer-energy condition on V}
0<E(u)<\infty.
\end{equation}

The energy condition is defined by the Hofer energy
\begin{equation}
\label{Equation: Hofer Energy in V}
E(u):=\sup_{\varphi} \left( \frac{1}{2}\int_{\cp^1 \setminus \Gamma} \| d u \|_{J, \varphi}^2 \right)
\end{equation}
where the supremum is taken over all strictly increasing smooth functions
\[
\varphi:(-\infty,\varepsilon] \rightarrow (0,1),\quad \lim_{a \to -\infty}\varphi(a)=0,
\]
such that the induced symplectic form $\omega_\varphi := d(e^\varphi \alpha)$ agrees with $\widehat{\omega}$ on $(\varepsilon/2,\varepsilon) \times Y \subset \widehat{V}$. The operator norm $\|du\|_{J,\varphi}$ is taken with respect to the induced metric \[g_{J,\varphi}(\cdot,\cdot)=\omega_\varphi(\cdot,J\cdot).\]

The finite energy condition implies that one of the following mutually exclusive cases holds on each puncture $z \in \Gamma$:
\begin{itemize}
    \item negative puncture: Every sequence $\{z_n\}_{n\in\mathbb{N}}$ converging to $z$ eventually leaves every compact subset of $\widehat{V}$. Equivalently, writing $u=(a,\bar{u})$, we have
    \[ \lim_{z'\to z}a(z')=-\infty. \]
    \item Removable puncture: Every sequence
    $\{z_n\}_{n\in\mathbb{N}}$ converging to $z$ is eventually contained in
    a compact subset of $\widehat{V}$.
\end{itemize}
At the negative puncture, there exists a closed Reeb orbit $\gamma_{z}$ with the convergence is given in the sense of \eqref{Equation: Description at negative puncture for problem Y}. At the removable puncture, $u$ smoothly extends over $z$ as a $J$-holomorphic sphere. The Hofer energy is given by the following lemma.
\begin{lemma}
\label{Lemma: formula Hofer energy in V}
Let $u:\cp^1\setminus \Gamma \to \widehat{V}$ be a nonconstant finite energy $J$-holomorphic curve in $\widehat{V}$. Let $\Gamma^-$ denote the set of negative punctures, and let $\gamma_z$ denote the asymptotic closed Reeb orbit at each $z\in\Gamma^-$. Then the Hofer energy \eqref{Equation: Hofer Energy in V} satisfies
\begin{equation}
    \label{Equation: Hofer energy on V in action}
    E(u) = 2\pi\cdot  (u \cdot \Sigma) - \sum_{z \in \Gamma^-} \mathcal{A}(\gamma_z),
\end{equation}
where $u \cdot \Sigma $ denotes the signed intersection number between $u$ and $\Sigma$.
\end{lemma}
\begin{proof}
    For generic $a<\epsilon$, the inverse image
    \[ u^{-1}(\{a\}\times Y) \]
    consists of finitely many circles. Moreover, when $-a$ is sufficiently large, we may assume that it consists of exactly $\#\Gamma$ circles, each of which encircles one puncture $z\in\Gamma$ once. This follows from the asymptotic formula proved in \cite{HWZ-1}.

    Let $\varphi: (-\infty,\epsilon]\to(0,1)$ be a strictly increasing smooth function which agrees with the identity function on $[\varepsilon/2,\varepsilon]$, and converge to $0$ at $-\infty$. By Stokes' theorem, the $L^2$-energy can be written as
    \begin{equation}
    \begin{aligned}
    \label{Equation: Proof on Hofer energy - 1}
        \frac{1}{2} \int_{u^{-1}(V_a)} \|du\|^2 &= \int_{u^{-1}(V_a)} u^*\omega_{\varphi}\\
        &= \int_{u^{-1}(V_\delta)} u^*\omega_{\varphi} + \int_{u^{-1}([a,\delta]\times Y)} u^*\omega_{\varphi}\\
        &= \int_{u^{-1}(V_\delta)} u^*\omega_{\varphi} + \int_{u^{-1}(\{\delta\}\times Y)} u^*(e^\varphi \alpha) - \int_{u^{-1}(\{a\}\times Y)} u^*(e^\varphi \alpha)
    \end{aligned}
    \end{equation}
    where $\delta\in[\varepsilon/2,\varepsilon]$ is chosen so that the preimage of $\{\delta\}\times Y$ consists of finitely many circles. Here, we denote by $V_a$ the subdomain
    \[
    V_a:=V\cup_Y[a,0]\times Y.
    \]
    
    As $r\to -\infty$, we have
    \[
    \lim_{a \to -\infty} \int_{u^{-1}(\{a\}\times Y)} u^*(e^\varphi \alpha) = \sum_{z \in \Gamma} \mathcal{A}(\gamma_z).
    \]
    Therefore, it remains to show that
    \begin{equation}
    \label{Equation: Proof on Hofer energy - 2}
     2\pi\cdot  (u \cdot \Sigma) = \int_{u^{-1}(V_\delta)} u^*\omega_{\varphi} + \int_{u^{-1}(\{\delta\}\times Y)} u^*(e^\varphi \alpha).
    \end{equation}

    Considering $V_\delta$ as a subdomain of $\cp^1\times\cp^1$. Since the complement
    \[ W_{\delta}:=\cp^1\times\cp^1\ \setminus \operatorname{int}(V_\delta) \]
    is contractible, each circle of $\operatorname{im}(u) \cap \{\delta\}\times Y$ can be capped by a disk $D_i$ in $W_{\delta}$. Then,
    \[
    \int_{u^{-1}(\{\delta\}\times Y)} u^*(e^\varphi \alpha)  = \sum \int_{D_i}\omega
    \]    
    holds by Stoke's Theorem, and the right hand side of equation \eqref{Equation: Proof on Hofer energy - 2} agree to integration of the symplectic form $\omega$ over the closed sphere obtained by gluing the capping disks with $u|_{u^{-1}(V_\delta)}$ in $\cp^1 \times \cp^1$.

    The integral of $\omega$ over the resulting closed sphere is equal to $2\pi$ times its signed intersection number with $\Sigma$. Since the capping disks are contained in $W_\delta$, which is disjoint from $\Sigma$, all intersections with $\Sigma$ occur on $u|_{u^{-1}(V_\delta)}$. Therefore, the equality \eqref{Equation: Proof on Hofer energy - 2} follows, completing the proof.
\end{proof}

We now focus on the case
\begin{equation}
\label{Equation: Obstruction on simple Floer cylinders in V - 1}
    u \cdot \Sigma =1 \quad \text{and} \quad u(\infty)\in \Sigma.
\end{equation}
Under the assumption (J-1), $\Sigma$ is also a $J$-holomorphic curve in $\widehat{V}$, positivity of intersections implies that $u$ and $\Sigma$ intersect transversely at the unique point $u(\infty)$. In particular, every finite energy $J$-holomorphic curve satisfying \eqref{Equation: Obstruction on simple Floer cylinders in V - 1} is simple.

Hence, for a generic choice almost complex structure SFT-like on the cylindrical end $\mathbb{R}_{\le 0}\times Y\subset\widehat{V}$, the linearised Fredholm operator associated to every such curve is surjective. Consequently, the moduli space of unparametrized $J$-holomorphic curves with the same puncture set and asymptotic conditions is a smooth manifold whose dimension at $u$ is given by the Fredholm index
\begin{equation}
\label{Equation: Fredholm index of J-holomorphic curve in V}
    \operatorname{Ind}(u) = 2 +\#\Gamma - \sum_{z \in \Gamma^-} \mu_{\rm CZ}(\gamma_z).
\end{equation}

There are two ways to see this index formula. First, following \cite[Prop.~5.2]{HWZ03}, assume that one can caps each puncture $\gamma_z$ by a smooth plane in $\widehat{W}$. The Fredholm index contribution of each cap is $\mu_{\rm CZ}(\gamma_z)-1$, and the total contribution is equal to the Fredholm index of the resulting glued unparametrised sphere map in $\cp^1 \times \cp^1$ with one marked point constrained to lie on $\Sigma$. Arguing as in \eqref{Equation: vdim constrained}, one obtains
\[
2c_1(A)+\dim M-\operatorname{codim}\Sigma -\dim(\operatorname{Aut}(\mathbb{C},i))=2.
\]

Alternatively, one can consider a one-to-one correspondence between such pseudoholomorphic curves in $\widehat{V}$ and pseudoholomorphic curves in the completion of the symplectic cobordism
\[ D^*_{g_0}S^2\setminus\operatorname{int}(W) \]
whose asymptotic Reeb orbits on $\partial D^*_{g_0}S^2$ all have multiplicity one; see \cite[Lem.~2.6]{Diogo-Lisi19}. Applying the index formula \eqref{Equation: Fredholm index of J-holomorphic curve in Y} to the latter curves yields exactly the same formula. In this setting, one should instead use a Morse--Bott formalism for $\partial D^*_{g_0}S^2$.

\medskip

We finish the subsection with the following lemma, which asserts the existence of a non-contractible Reeb orbit among the asymptotic orbits. Here, non contractibility is understood in $Y$, rather than in $V$ which is simply connected.

\begin{lemma}
\label{Lemma: Existence of a negative puncture with noncontractible asymptotic}
    For a nonconstant finite energy pseudoholomorphic curve $u$ satisfying
    \[ u \cdot \Sigma =1,\]
    there exists a negative puncture with non-contractible asymptotic Reeb orbit.
\end{lemma}
\begin{proof}
    Assume, to the contrary, that every asymptotic Reeb orbit of $u$ is contractible in $Y$. As in the proof of Lemma~\ref{Lemma: formula Hofer energy in V}, choose $-a>0$ sufficiently large so that
    \[
    \operatorname{Im}(u)\cap\{a\}\times Y
    \] 
    consists of contractible loops in the slice $\{a\}\times Y$.
    
    Capping of these slice with disks in $\{a\}\times Y$, we obtain a sphere map
    \[ v:\cp^1\longrightarrow \widehat{V} \]
    whose image is contained in
    \[ V_a:=V\cup_Y[a,0]\times Y. \]
    These capping disks are disjoint from $\Sigma$, and hence $v$ intersects $\Sigma$ transversally, as does $u$. Moreover, the intersection number of $v$ with $\Sigma$ agrees with that of $u$, namely,
    \begin{equation}
    \label{Equation: Proof of noncontractible orbit}
       v\cdot \Sigma= u \cdot \Sigma =1. 
    \end{equation}

    On the other hand, $V_a$ deformation retracts onto $\Sigma$, and hence
    \[
        H_2(V_a)\cong H_2(\Sigma)\cong\mathbb{Z},
    \]
    is generated by $[\Sigma]$. Thus, the homology class represented by $v$ is an
    integer multiple of $[\Sigma]$. Then, the intersection number of $v$ with $\Sigma$ must also be even since $[\Sigma]\circ[\Sigma]=2.$ This contradicts the intersection number \eqref{Equation: Proof of noncontractible orbit}, completing the proof.
\end{proof}

\vspace{.5cm}

\subsection{SFT-limit building of Floer spheres}\label{section: SFT-limit}

In this section, we study the possible SFT-limit building arising from a sequence of Floer spheres through the neck-stretching procedure. We first recall the geometric setup and the neck-stretching construction, and then describe the possible limit curves and their matching asymptotic conditions in the resulting building. Due to the presence of limit Floer curves for the degenerate Hamiltonian $H'$ in the limit building, we include additional argument to prove the nondegeneracy of the asymptotic ends; see Lemma~\ref{Lemma: Nondegeneracy for limit curves}. 

\medskip

Recall that $H:W\to\mathbb{R}_{\le 0}$ is a Hofer--Zehnder admissible Hamiltonian compactly supported in $\operatorname{int}(W)$. We regard $(W,\omega)$ as a subdomain of $(\operatorname{int}(D^*_{g_0}S^2),\dd \lambda)$, or equivalently as a subdomain of
\[
(\cp^1\times\cp^1\setminus\Sigma,\;2\omega_{\rm FS}\oplus2\omega_{\rm FS}),
\]
where $\Sigma$ is a symplectic surface given by the diagonal of $\cp^1 \times \cp^1$. The boundary $\partial W$ is a contact-type hypersurface separating $M=\cp^1\times\cp^1$ into $W$ and $V$, where $V$ contains $\Sigma$. We regard $(\partial W,\lambda|_{\partial W})$ as a closed contact manifold $(Y,\alpha)$.

For such a Hamiltonian $H$, fix a point $p \in W$ at which $H$ attains its minimum and is constant nearby. We also fix the following three open subsets of $\cp^1\times\cp^1$:
\begin{itemize}
    \item $U_\infty$: an open neighborhood of the symplectic divisor $\Sigma$;
    \item $U_{\rm \min}$: a Darboux ball centered at $p$ with a chart
    \[
    \iota:(B^4(r_0),\omega_0)\longrightarrow (U_{\rm min},\omega|_{U_{\rm min}});
    \]
    \item $U_{\partial W}$: collar neighborhood of $\partial W$ with the canonical identification \[(U_{\partial W},\omega|_{\partial W}) \cong \left((-\epsilon,\epsilon)_r\times\partial W, d(e^r\cdot\lambda_{\rm taut}|_{\partial W})\right).\]
\end{itemize}
For $U_{\partial W}$, we further assume that $H$ vanishes identically.

\medskip

Let $\mathcal{J}(p,Y,\Sigma)$ denote the space of almost complex structures satisfying:
\begin{itemize}
    \item[(J-1)] $J$ coincide with the complex structure $i \oplus i$ on a neighborhood $U_\infty$ of $\Sigma$.
    \item[(J-2)] Under the parametrization $\iota$, the almost complex structure $J$ agrees with the standard complex structure $i$ at the origin.
    \item[(J-3)] $J$ is SFT-like in the collar neighborhood $U_{\partial W}$.
\end{itemize}
And let $H':\cp^1\times\cp^1\to\mathbb{R}^{\le0}$ denote the Hamiltonian obtained by perturbing $H$ inside the Darboux ball $U_0$ so that $p$ becomes the unique minimum of $H'$ and is nondegenerate as a $1$-periodic orbit of the Hamiltonian flow; see \eqref{Equation: C0-perturbation inside Darboux Ball}.

\medskip

With this setup, let $(M_R,\omega_R)$ denote the symplectic manifold obtained by replacing the collar neighborhood $U_{\partial W}$ with the longer neck $[-R-\varepsilon,R+\varepsilon]\times\partial W$. Recall that there exists a symplectomorphism
\[
\Phi_R:(M_R,\omega_R)\longrightarrow
(\cp^1\times\cp^1,\;2\omega_{\rm FS}\oplus2\omega_{\rm FS}),
\]
which agrees with the identity map outside the neck. Every $J\in\mathcal{J}(p,Y,\Sigma)$ induces an almost complex structure $J_R$ on $M_R$, not as the pullback, but by extending $J$ on the neck with the translation invariance. Hamiltonian $H'$ also naturally extends to a Hamiltonian $H_R$ defined on $M_R$. The almost complex structure $J\in\mathcal{J}(p,Y,\Sigma)$ also induces SFT-like almost complex structures on the symplectisation $\mathbb{R}\times Y$, as well as on the completions $\widehat{W}$ and $\widehat{V}$, while $H'$ extends naturally to a Hamiltonian on $\widehat{W}$. By abuse of notation, we continue to denote these by $J$ and $H'$, respectively.

\medskip

By Corollary \ref{Corollary: existence of Floer spheres}, for each neck-stretched symplectic manifold $(M_n,\omega_n)$, there exists a Floer
sphere
\[
u_n \in \mathcal{M}(p,\Sigma;J_n,H_n,A_n),
\]
i.e., a solution $u_n:\mathbb{R}\times S^1\to M_n$ of
\begin{equation}
    \partial_s u_n + J_n(u_n(s,t))\partial_t u = \nabla H_n (u_n(s,t)).
\end{equation}
with asymptotic conditions
\[ \lim_{s\to-\infty}u_n(s,t)=p, \qquad \lim_{s\to+\infty}u_n(s,t)\in\Sigma, \]
and whose continuous extension $\bar u_n$ obtained by identifying
\[ (\mathbb{R}\times S^1)\cup\{-\infty,+\infty\}\cong\cp^1 \]
satisfies $[\bar u_n]=A_n$, where $A_n$ denotes the pullback $(\Phi_n^{-1})^*A$ of the line class \[ A=[\cp^1\times\{\mathrm{pt}\}]\in H_2(\cp^1\times\cp^1). \]

Here, we can still apply Corollary~\ref{Corollary: existence of Floer spheres}
since, under the symplectomorphism
\[
\Phi_n:(M_n,\omega_n)\longrightarrow (\cp^1\times\cp^1,\;2\omega_{\rm FS}\oplus2\omega_{\rm FS}),
\]
the pulled-back Hamiltonian $(\Phi_n^{-1})^*H_n$ coincides with the perturbed
Hamiltonian $H'$, and the induced almost complex structure $(\Phi_n)_*J_n$ satisfies conditions (J-1) and (J-2).

\medskip

Next, we consider the SFT limit of the sequence of unparametrised Floer spheres
\[ \{[u_n] \in \widehat{M}(p,\Sigma;J_n,H_n;A_n)\}. \]
Within  the $\mathbb C^*$-family of parametrized Floer spheres, we choose a
representative satisfying the normalization \eqref{Equation: Normalisation for bottom building}, namely,
\begin{equation}
\label{Equation: Proof of the lowest building - 1}
    u_n(0,0) \in \partial U_0
    \qquad\text{and}\qquad
    u_n(\mathbb{R}^{<0}\times S^1)\subset U_0.
\end{equation}
Here, $U_0$ is a Darboux ball contained in $U_{\rm min}$ such that $p$ is the only critical point in the closed Darboux ball $\bar{U}_0$. By construction of $H'$, the critical point $p$ is $H'$-regular.

The SFT limit is described by a finite collection of finite-energy curves modelled on a finite graph. Each vertex of the graph corresponds to a finite-energy curve, which is obtained either as the $C^\infty_{\rm loc}$-limit of a sequence of a sequence of parametrised Floer spheres
\[
    \lambda_n\cdot u_n,\qquad \lambda_n\in\mathbb{C}^*,
\]
modulo bubbling, or as a $J$-holomorphic curve appearing from the bubbling analysis. 

The limit curves obtained as the $C^\infty_{\rm loc}$-limit of the sequence $\lambda_n\cdot u_n$ corresponds to one of the following
types of nonconstant finite-energy curves:
\begin{itemize}
    \item Floer curves in $\widehat{W}$ for Floer data $(J,H')$;
    \item $J$-holomorphic curves in $\mathbb{R}\times Y$; and
    \item $J$-holomorphic curves in $\widehat{V}$.
\end{itemize}
We will refer to these limits as \textit{principal curves}. The finite Hofer energy
property follows from the fact that the $L^2$-energy of $E(\lambda_n\cdot u_n)$ is
bounded above by $2\pi$.

If a principal curve lies in either $\mathbb{R}\times Y$ or $\widehat{V}$, it satisfies a Cauchy--Riemann equation since the Hamiltonian $H_n$ vanishes on the stretched necks and on $V$. When it lies in $\widehat{W}$, it satisfies the Floer equation on $\widehat{W}$. Since the domain of each $u_n$ was the cylinder $\mathbb{R}\times S^1$, the limit curves have domains of the form $\mathbb{R}\times S^1\setminus \Gamma$, where $\Gamma$ is a possibly non-empty set of punctures corresponding to the points where bubbling occurs. Each puncture is referred to as either \textit{positive}, \textit{negative}, or \textit{removable}.

Curves obtained from the bubbling analysis will be referred to as \textit{bubble curves}. These are given by the $C^\infty_{\rm loc}$-limits of reparametrised curves
\[
    z\mapsto \lambda_n\cdot u_n(z_n+r_n\cdot z),
\]
where $z_n\in\mathbb{R}\times S^1$ converging to some point $z_\infty$, positive real numbers $r_n$ satisfying $r_n\to\infty$, and $\lambda_n\in\mathbb{C}^*$ are chosen such that
\[
    \|d(\lambda_n\cdot u_n)(z_n)\|\to\infty.
\]
As in the usual bubbling analysis of Floer cylinders, the limiting equation associated with each bubble curve is also a Cauchy--Riemann equation, since the gradient term is divided by the factor $r_n$ and vanishes as $n\to\infty$. Since the domain of each reparametrised curve is obtained by rescaling the original domain near the point $z_n$, each bubble curve is also defined on a punctured sphere.

\medskip
We now discuss principal curves in more detail, including some details of our
setup that are slightly different from the standard pseudoholomorphic curve
case. Let $\sigma_n\in\mathbb{R}$ be a real number such that there exists $\tau_n\in S^1$ satisfying
\begin{equation}
\label{Equation: Proof of the lowest building - 3}
    u_n(\sigma_n,\tau_n)\in\partial U_\infty
    \qquad\text{and}\qquad
    u_n(\mathbb{R}_{>\sigma_n}\times S^1)\subset U_\infty.
\end{equation}
The following lemma gives an obstruction on the sequence $\lambda_n\in\mathbb{C}^*$ for principal curves, and characterises the top and bottom level level of principal curves.

\newpage

\begin{lemma}
\label{Lemma: Top and Bottom level}
    Let $\lambda_n\in\mathbb{C}^*$ be a sequence of nonzero complex numbers
    such that
    \[
    \lambda_n\cdot u_n
    \]
    converges, up to a subsequence, to a non-constant curve. Then
    there exists a constant $S>0$ such that
    \[
    -S< \log |\lambda_n|<S+\sigma_n.
    \]
\end{lemma}

\begin{proof}
    Let $s_n$ be a real number defined as $\log |\lambda_n|$. Then, $v_n=\lambda_n\cdot u_n$ satisfies
    \[
        v_n|_{(-\infty,-s_n)\times S^1}\subset U_0
        \quad\text{and}\quad
        v_n|_{(\sigma_n-s_n,\infty)\times S^1}\subset U_\infty.
    \]
    Assuming the contrary, we either obtain the limit of $v_n$ as a curve contained in $\bar{U}_0$ solving the Floer equation or contained in $\bar{U}_\infty$ solving the Cauchy--Riemann equation.

    In the first case, 
    since $p$ is the only $1$-periodic orbit on $\bar{U}_0$, the limit curve converges to $p$ at both $\pm\infty$. This implies that the curve is constant, contradicting the assumption that the limit of $v_n$ is a nonconstant principal curve.

    In the second case, the limit curve defines a nonconstant punctured $i\oplus i$-holomorphic sphere contained in $\bar{U}_\infty$. Such curves must have $L^2$-energy at least $4\pi$ if they are not constant. Hence, in either case, we obtain a contradiction, which finishes the proof.
\end{proof}

Yet, a principal curve
\[
u:\mathbb{R}\times S^1  \setminus \Gamma \to \widehat{W},
\]
a priori does not necessarily have a limit as the Hamiltonian $H'$ is always degenerate. Nevertheless, the limit is always well defined in our setup:

\begin{lemma}
\label{Lemma: Nondegeneracy for limit curves}
Every principal curve in $\widehat{W}$ is asymptotic either to the minimum $p$ of the Hamiltonian $H'$ or to a closed Reeb orbit. The minimum $p$ can only appear as the negative asymptotic limit of a curve arising from a limiting sequence with bounded $\lambda_n$.
\end{lemma}
\begin{proof}
Assume that $v_\infty$ is obtained as the limit of $v_n:=\lambda_n \cdot u_n$. Since the extension $[u_n]:\cp^1 \to M_n$ only intersects the symplectic divisor $\Sigma$ at the added point $+\infty$, the image of $u_n$ lies in the complement $M_n \setminus \Sigma$. The symplectic form on the complement is exact, given by the pullback of the tautological $1$-form $\lambda_{\rm taut}$ via the diffeomorphism
\[
\Phi_n : M_n\setminus \Sigma \to \operatorname{Int}(D^*_{g_0}S^2).
\]
We denote such $1$-form by $\lambda_n$. Then, for each loop $\gamma_{n,s}:S^1\to M_n\setminus\Sigma$ defined by
\[
    \gamma_{n,s}(t)=u_n(s,t),
\]
we choose the capping $\bar{u}_{n,s}:D^2\to M_n\setminus\Sigma$ obtained
by completing the half Floer cylinder $u_n|_{(-\infty,s]\times S^1}$ with $p$.  The Hamiltonian action of the capped orbit $(\gamma_{n,s},\bar{u}_{n,s})$ is given by
\begin{equation*}
    \begin{aligned}
        \mathcal{A}_{H_n}(\gamma_{n,s},\bar{u}_{n,s})
        &=\int_{D^2} \bar{u}_{n,s}^*\omega_n -\int_{S^1}H_n(\gamma_{n,s})\,dt\\
        &=\int_{S^1} \gamma_{n,s}^*\lambda_n -\int_{S^1}H_n(\gamma_{n,s})\,dt.
    \end{aligned}
\end{equation*}
Moreover, it satisfies
\begin{equation}
\label{Equation: Proof of the lowest building - 2}
    \mathcal{A}_{H_n}(\gamma_{n,s})-\mathcal{A}_{H_n}(p)
    = \int_{(-\infty,s]\times S^1}\|\partial_su_n\|^2>0.
\end{equation}
In particular, we have
\[
    \mathcal{A}_{H_n}(\gamma_{n,s})>-H'(p).
\]

Since every asymptotic loop of a principal curve is obtained as the limit of
some sequence of loops $u_n(s_n,-)$, possibly after a rotation in the $S^1$-direction, the action of such an asymptotic loop is at least $-H'(p)$. Here, the Hamiltonian action in $\widehat{W}$ should again be calculated with respect to a capping disk in $\widehat{W}$ and a strictly increasing function $\varphi$. By Hofer--Zehnder admissibility of $H$, every $1$-periodic orbit of $H$ is constant. Hence, if the limit converges to a critical point other than $p$, the Hamiltonian action, which is given by the negative of the Hamiltonian value, gives a contradiction. Therefore, it either converges to a closed Reeb orbit or to the minimum $p$.

Combining Lemma \ref{Lemma: Minimal escape energy}, we further conclude that $p$ can precisely occur as the negative asymptotic limit for a bounded sequence $\lambda_n$, since the equation
\eqref{Equation: Proof of the lowest building - 2} can be otherwise strengthened to
\[
\mathcal{A}_{H_n}(\gamma_{n,s})-\mathcal{A}_{H_n}(p)
=
\int_{(-\infty,s]\times S^1}\|\partial_s u_n\|^2>E_-.
\]
This finishes the proof of the Lemma.
\end{proof}

We note that the argument also rules out the possibility that $\pm\infty$ are removable punctures. Recall that, we also regard sequential convergence to a $1$-periodic orbit of $X_{H'}$ as a removable puncture, even though the limit may not be well-defined. The only exceptions are $(w_\infty,+\infty)$, where the puncture is removable and the curve extends smoothly over a point in $\Sigma$, and $(u_\infty,-\infty)$, where the curve converges to $p$.

Together with these lemmas, we can describe the graph formed principal curves in the limit building. The convergence to a building of principal curves (modulo bubbling) is characterised by the existence of a finite tuple of sequences $(\lambda^0_n,\ldots,\lambda^k_n)$ such that
\begin{equation}
\begin{aligned}
    0=\lambda^0_n<&\cdots<\lambda^k_n=\sigma_n,\\
    \lim_{n\to\infty}(\lambda^i_n-\lambda^{i-1}_n)=&\infty
\qquad\text{for }1\leq i\leq k.
\end{aligned}
\end{equation}
For each $i$, the reparametrised curves
\[
v^i_n:=\lambda^i_n\cdot u_n
\]
converge to a principal curve $v^i_\infty$ on $\mathbb{R}\times S^1\setminus\Gamma^i$, where $\Gamma^i$ is the set of bubbling points.

Note that if two sequences $\lambda_n\neq\lambda'_n$ in $\mathbb{C}^*$ both induce convergent subsequences to principal curves while satisfying
\[
\log|\lambda_n/\lambda'_n|\text{ is bounded},
\]
then, after passing to a subsequence, the limiting principal curves
$v_\infty$ and $v'_\infty$, each obtained as a limit of
$v_n=\lambda_n\cdot u_n$ and $v'_n=\lambda'_n\cdot u_n$,
are related by
\[
    v_\infty=\lambda_\infty\cdot v'_\infty
\]
for some $\lambda_\infty\in\mathbb{C}^*$ obtained from the boundedness. Hence, distinct levels in the limit building can only arise from sequences whose relative reparametrisation diverges. 

Hence, we take $\lambda_n^0=0$ and $\lambda_n^k=\sigma_n$; see
Lemma \ref{Lemma: Top and Bottom level}.  For this choice, the limits
\[
    u_\infty:=v_\infty^0 \qquad\text{and}\qquad w_\infty:=v_\infty^k
\]
are nonconstant and have energy at least $E_-$ and $E_+$, respectively,
by Proposition \ref{Proposition: moduli space of Floer spheres}. Thus, $u_\infty$ and $w_\infty$ are indeed principal curves in $\widehat{W}$ and $\widehat{V}$, respectively.

By Lemma \ref{Lemma: Nondegeneracy for limit curves}, the other punctures at $\pm\infty$ are either positive or negative punctures asymptotic to closed Reeb orbits. These punctures occur in matching pairs: curves converging to the same closed Reeb orbit $\gamma_i$ at punctures $(v_\infty^i,-\infty)$ and $(v_\infty^{i-1},+\infty)$, with one puncture positive and the other negative. At this stage, we do not determine which puncture is positive or negative. We will soon show that the puncture at $+\infty$ is always positive, while the puncture at $-\infty$ is always negative.

We omit the proof of the existence of such a tuple $(\lambda^0_n,\ldots,\lambda^k_n)$, and the matching condition. These follow by inductively reparametrising $u_n$ from either the top building $u_\infty$ or the bottom building $w_\infty$, and using the fact that the total energy is bounded above by $2\pi$, together with the non-degeneracy of the contact form $\lambda$, to obtain a minimal energy threshold preventing further breaking and bubbling in the half-cylinders.

The description of the limit building is completed by the bubbling-off analysis
at each bubbling point $z_i$. The resulting limit building is modeled on a
tree graph: each vertex corresponds either to a principal curve, obtained as a
limit of reparametrisations $\mathbb{R}\times S^1 \cong \mathbb{C}^*$ (principal curves) or obtained as a bubbling off analysis for the principal cuvres (bubbling cuves). Each edges implies that there are pair of punctures matched as either a pair of (positive, negative) or (removable, removable) punctures. Since the original Floer spheres have genus zero, the limit building has genus zero, and thus its underlying graph is a finite tree. Here, the finiteness follows from the fact that each curve has a positive minimal energy, together with the finiteness of Reeb orbits with period less than $2\pi$.

\medskip

\noindent\textbf{Proof of Proposition ~\ref{Proposition: Bottom level}.} We now prove Proposition ~\ref{Proposition: Bottom level}, which establishes the existence of the Floer planes, with consideration of Fredholm index of the limit curves. For the proof it is crucial that the intersection number \[u_n\cdot \Sigma=1.\]

\noindent
\begin{proof}[Proof of Proposition \ref{Proposition: Bottom level}.]

We prove the proposition by establishing the following claims.
\begin{itemize}
    \item[Claim 1]: There is no limit curve in $\widehat{V}$ other than
    $w_\infty$.
    \item[Claim 2]: For each $0\leq i\leq k$, the set $\Gamma^i$ contains no
    positive or removable punctures.
\end{itemize}

By Claim 2, the bottom principal curve
\[
u_\infty:\mathbb{R}\times S^1\setminus\Gamma^0\to\widehat{W}
\]
has no punctures, namely $\Gamma^0=\varnothing$, and is a Floer cylinder asymptotic to the minimum $p$ at $-\infty$. By Lemma \ref{Lemma: Nondegeneracy for limit curves}, it converges to a closed Reeb orbit $\gamma_1$ at $+\infty$. Hence, it defines a Floer plane in $\widehat{W}$, i.e.,
\[
    \mathcal{M}(p,\gamma_1;J,H')\neq\varnothing.
\]
It remains to prove the two claims.
\begin{proof}[Proof of Claim 1]
    Assume that there exists another non-constant limit curve
    \[ v:\mathbb{CP}^1\setminus\Gamma\to\widehat{V} \]
    other than $w_\infty$. Since the total intersection number with $\Sigma$ is preserved under the SFT limit building, the curve $v$ must satisfy
    \[
    v\cdot\Sigma \leq
    \Sigma\cdot u_n|_{(-\infty,\sigma_n]\times S^1}.
    \]
    Indeed, all bubbling points occurs outside the half-cylinder $(\sigma_n,\infty)\times S^1$, and hence the intersection contribution of $v$ is bounded above the interesections on $u_n\left({(-\infty,\sigma_n]\times S^1}\right)$. Since each Floer sphere intersects $\Sigma$ only at $+\infty$, we have \[
    \Sigma\cdot u_n|_{(-\infty,\sigma_n]\times S^1}=0, \] and hence $v\cdot\Sigma\leq 0$. If the image of $v$ is not contained in $\Sigma$, positivity of intersections gives
$v\cdot\Sigma\geq 0$. Thus $v\cdot\Sigma=0$, and $v$ is disjoint from
$\Sigma$. Consequently, there are only two possibilities:
\[
    \operatorname{im}(v)\cap\Sigma=\varnothing
    \qquad\text{or}\qquad
    \operatorname{im}(v)\subset\Sigma.
\]
The first possibility is excluded by the energy formula
\eqref{Equation: Hofer energy on V in action}, since $E(v)>0$ implies
$v\cdot\Sigma>0$. The second possibility is also impossible, since every
limit curve has energy less than $2\pi$, whereas a nonconstant
holomorphic curve contained in $\Sigma$ has energy at least $4\pi$. This completes the proof of Claim 1.
\end{proof}

\begin{proof}[Proof of Claim 2]
Claim 1 immediately proves Claim 2 for $i=k$. Suppose that there exists a principal curve with positive or removable puncture. Following the edge associated with this puncture, we obtain a non-constant finite-energy $J$-holomorphic curve in either
\[
\widehat{W} \quad \text{or} \quad \mathbb{R}\times Y.
\]
The possibility that this curve lies in $\widehat{V}$ is excluded by Claim 1. Such a curve must have a positive puncture. If all punctures are removable, one obtains a non-constant pseudoholomorphic sphere without punctures in either $\widehat{W}$ or $\mathbb{R}\times Y$, which is impossible by the exactness of the symplectic form. The remaining possibility, where the curve lies in $\mathbb{R}\times Y$ and has a negative puncture but no positive puncture, is also excluded by the $d\lambda$-energy formula \eqref{Equation: dlambda energy on Y in action}.

Precisely, the same argument implies that the adjacent pseudoholomorphic curve matched with the positive puncture of such a curve has another positive puncture. Repeating this argument inductively, we obtain a path extending further away from the starting principal curve. However, this contradicts the fact that the SFT limit is modeled on a finite tree graph. This finishes the proof of Claim 2.
\end{proof}

Combining the two claims with the argument above, we conclude that the bottom
principal curve defines a Floer plane in $\widehat{W}$. This completes the
proof of the proposition.

\end{proof}

The proof of Proposition~\ref{Proposition: Bottom level} further gives a more
precise description of the limit building: If an edge connects two principal curves, then the puncture at $+\infty$ of the lower curve is positive, while the matching puncture at $-\infty$ of the upper curve is negative. If an edge connects a bubbling curve to a principal curve, then the puncture on the bubbling curve corresponding to this edge is positive, while all its remaining punctures are negative.

To see this, recall that the $d\lambda$-energy of $J$-holomorphic curves in $\mathbb{R}\times Y$ implies that every limit curve in $\mathbb{R}\times Y$ must have a positive puncture. By Claim 2, for a principal curve in $\mathbb{R}\times Y$, either $+\infty$ or $-\infty$ is a positive puncture. Since $v_\infty^0$ has a positive puncture at $+\infty$ by Lemma~\ref{Lemma: Nondegeneracy for limit curves}, the adjacent principal curve $v_\infty^1$ lies in $\mathbb{R}\times Y$ and has a negative puncture at $-\infty$ by the matching condition. Therefore, the puncture at $+\infty$ of $v_\infty^1$ is positive. Applying the same argument inductively, we conclude that each intermediate principal curve $v_\infty^i$ lies in $\mathbb{R}\times Y$ with a positive puncture at $+\infty$ and a negative puncture at $-\infty$. The same argument also determines the remaining edges.

\medskip

\noindent\textbf{Proof of Theorem \ref{Theorem: Upper bound from index 3 - nondegenerate W}.} Recall that Theorem~\ref{Theorem: Upper bound from index 3 - nondegenerate W} gives an upper bound for the Hofer--Zehnder capacity $c_{\rm HZ}(W)$ by the period of closed Reeb orbits in $\mathcal{P}_{\rm HZ}(W)$. Here, $\mathcal{P}_{\rm HZ}(W)$ denotes the set of closed Reeb orbits $\gamma$ of Conley--Zehnder index $3$ satisfying either
\begin{itemize}
    \item $\gamma$ is non-contractible in $\partial W$ and satisfies $\mathcal{A}(\gamma)<2 \pi$, or
    \item $\gamma$ is contractible in $\partial W$ and
    there exists a non-contractible orbit $\gamma_1$ satisfying
    \[{\rm CZ}(\gamma')=1 \quad\text{and}\quad\mathcal{A}(\gamma) + \mathcal{A}(\gamma') \le 2\pi.\]
\end{itemize}
For the proof we consider the Fredholm indices of the limit curves. Up to this point, we have used only the $L^2$-energy and intersection with $\Sigma$. To this end, we restrict to almost complex structures $J\in\mathcal{J}(p,Y,\Sigma)$ such that the linearised operators associated to all curves of energy less than $2\pi$ are surjective:

\begin{enumerate}
    \item every simple $J$-holomorphic spheres in
    \[
    \widehat{W},\quad \mathbb{R}\times Y,\quad \widehat{V},
    \]
    \item every simple Floer planes for every Hamiltonian 
    \[H'/m, \quad m\in\mathbb{N}.\]
\end{enumerate}
Again, a generic almost complex structure in $\mathcal{J}(p,Y,\Sigma)$ satisfies the assumption due to the nondegeneracy of the contact form, and since the Lemma \ref{Lemma: Somewhere injectivity of Floer spheres} can be applied to Floer plane in $\widehat{W}$ associated to $H'/m$ for arbitrary $m \in \mathbb{N}$.

\begin{proof}[Proof of Theorem~\ref{Theorem: Upper bound from index 3 - nondegenerate W}]
We prove the theorem by establishing the following claim.
\begin{itemize}
    \item[Claim 3:] Every closed Reeb orbit arising as an asymptotic limit of a
    bubbling curve has positive Conley--Zehnder index.
\end{itemize}

We first complete the proof assuming Claim~3. Let $\gamma_i$ denote the closed
Reeb orbit corresponding to the edge connecting the principal curves $v_\infty^i$ and $v_\infty^{i-1}$. We inductively estimate the Conley--Zehnder indices of $\gamma_i$ using Claim~3. The initial orbit $\gamma_1$ satisfies
\[
\mu_{\rm CZ}(\gamma_1)\ge 3,
\]
with equality only if the bottom principal curve $v_\infty^0=u_\infty$ is simple.  Indeed, write $u_\infty$ as an $m$-fold cover of a simple Floer plane
\[
    u'_\infty\in\mathcal{M}^*(p,\gamma';J,H'/m),
\]
where $m\in\mathbb{N}$ and $\gamma'$ is a closed Reeb orbit satisfying $\gamma_1=(\gamma')^m$. Since $J$ is chosen so that every simple Floer plane is $H'/m$-regular, we have
\[
    \mu_{\rm CZ}(\gamma')\ge3
\]
by \eqref{Equation: Obstruction on simple Floer cylinders in W - 1}. It follows that
\[
    \mu_{\rm CZ}(\gamma_1) \ge 2m+1.
\]
Hence $\mu_{\rm CZ}(\gamma_1) \ge 3$, and equality can occur only when $u_\infty$ is simple.

Consider the next principal curve $v_\infty^1:\mathbb{R}\times S^1\setminus\Gamma^1$. By Claim 2 and 3, every bubbling point in $\Gamma^1$ is a negative puncture that converges to a closed Reeb orbit of positive Conley--Zehnder index. Again, the curve $v_\infty^1$ is not necessarily somewhere injective, and hence it is a branched cover of a somewhere injective curve
\[
v':\mathbb{R}\times S^1\setminus\Gamma'\to\mathbb{R}\times Y.
\]
The branched covering can be chosen so that $+\infty$ and $-\infty$ are mapped to $+\infty$ and $-\infty$, respectively. In particular, $v'$ still has a unique positive puncture at $+\infty$.

We distinguish two cases.
\begin{itemize}
    \item $v'$ is a trivial cylinder, i.e., its image is contained in a
    prime closed Reeb orbit $\gamma'$ under the projection $\pi:\mathbb{R}\times Y\to Y$.
    \item $v'$ is nontrivial, i.e., $v'$ is not preserved under $\mathbb{R}$-translation of $\RR\times Y$.
\end{itemize}
In the first case, the covering multiplicity of the positive puncture is the
sum of the covering multiplicities of the remaining punctures. Since, $\mu_{\rm CZ}(\gamma_1)\geq 3$, the underlying simple closed Reeb orbit $\gamma'$ has positive Conley--Zehnder index. This implies that
\begin{equation}
\label{Equation: Proof of the GW upper bound - 3}
    \mu_{\rm CZ}(\gamma_2)\geq \mu_{\rm CZ}(\gamma_1).
\end{equation}
Hence, we again obtain $\mu_{\rm CZ}(\gamma_2)\geq 3$ in this case.

In the second case, the Fredholm index of $v'$ is at least $1$ because of the $\mathbb{R}$-translation. Hence, the Fredholm index formula \eqref{Equation: Fredholm index of J-holomorphic curve in Y - single + pucture} implies
\begin{equation}
\label{Equation: Proof of the GW upper bound - 1}
    \mu_{\rm CZ}(v'(+\infty))-1 > \mu_{\rm CZ}(v'(-\infty))-1 + \sum_{z \in \Gamma}(\mu_{\rm CZ}(v'(z))-1).
\end{equation}
Here, $v'(z)$ denotes the asymptotic Reeb orbit corresponding to the puncture $z$.  Since every puncture in $\Gamma$ other than $-\infty$ corresponds to a bubbling puncture, Claim~3 implies that
\begin{equation}
\label{Equation: Proof of the GW upper bound - 2}
    \mu_{\rm CZ}(v'(+\infty)) \ge \mu_{\rm CZ}(v'(-\infty)) +1.
\end{equation}

If $v_\infty^1$ is simple, this implies that
\[
    \mu_{\rm CZ}(\gamma_2)\geq4.
\]
If $v_\infty^1$ is multiply covered, then the negative asymptotic orbit still
has positive Conley--Zehnder index, as in the first case. Hence,
\[
    \mu_{\rm CZ}(v_\infty'(+\infty))\geq2,
\]
and therefore
\[
    \mu_{\rm CZ}(\gamma_2) = \mu_{\rm CZ}(v_\infty'(+\infty)^m) \geq4,
\]
where $m$ is the covering multiplicity. Combining both cases, we conclude that
\[
\mu_{\rm CZ}(\gamma_2)\geq 3,
\]
with equality only if $v_\infty^1$ is a branched cover of a trivial cylinder. 

Inductively, we conclude that the closed Reeb orbit $\gamma_{k}$ corresponding
to the negative puncture $(w_\infty,-\infty)$ satisfies
\[\mu_{\rm CZ}(\gamma_{k}) \ge 3,\] and the equality holds only if
\begin{equation}
\label{Equation: Proof of the GW upper bound - 5}
    \mu_{\rm CZ}(\gamma_{1})=\mu_{\rm CZ}(\gamma_{1})=\cdots=\mu_{\rm CZ}(\gamma_{k})=3    
\end{equation}
and the intermediate principal curves $v^i_\infty$ is a branch cover of a trivial cylinder for $1 \le i \le k-1$ and the Floer plane $u_\infty = v^0_\infty$ is simple.

Finally, the top principal curve $w_\infty$ cannot be multiply covered, since its extension $\bar{w}_\infty$ obtained by adding $w_\infty(\infty)\in\Sigma$ satisfies
\[
    \bar{w}_\infty \cdot \Sigma=1.
\]
Therefore, the Fredholm index formula \eqref{Equation: Fredholm index of J-holomorphic curve in V} implies that
\begin{equation}
\label{Equation: Proof of the GW upper bound - 4}
    2 - (\mu_{\rm CZ}(\gamma_k)-1) - \sum_{z \in \Gamma^k} (\mu_{\rm CZ}(w_\infty(z))-1)\ge 0.
\end{equation}
By Claim~3, we conclude that \[ \mu_{\rm CZ}(\gamma_k)\leq 3, \] which, together with the previous inductive step, gives
\[
    \mu_{\rm CZ}(\gamma_k)=3.
\]

Hence, the Floer plane $u_\infty=v_\infty^0$ is simple, and every intermediate principal curve $v_\infty^i$, $1\leq i\leq k-1$, is a branched cover of a trivial cylinder  over the same simple closed Reeb orbit $\gamma_0$. Therefore, each $\gamma_i$ is a multiple cover of $\gamma_0$, namely,
\[
    \gamma_i=(\gamma_0)^{n_i},
\]
for some natural numbers satisfying $n_1\leq n_2\leq\cdots\leq n_k$. In particular, we have
\[
\mathcal{A}(\gamma_1)\leq\mathcal{A}(\gamma_k).
\]
By the $L^2$-energy $E(u_\infty)$, the orbit $\gamma_1$ satisfies
\[
    H'(p)+\mathcal{A}(\gamma_1) >0.
\]
Together with $\mathcal{A}(\gamma_1) \le \mathcal{A}(\gamma_k)$, we conclude that
\[
    \mathcal{A}(\gamma_k) > -\min H.
\]

Hence, it remains to verify that $\gamma_k\in\mathcal{P}_{\rm HZ}(W)$. If $\gamma_k$ is contractible, then it satisfies the first criterion in the definition of $\mathcal{P}_{\rm HZ}(W)$. Suppose now that $\gamma_k$ is non-contractible. By Lemma~\ref{Lemma: Existence of a negative puncture with noncontractible asymptotic}, the top building $w_\infty$ has an additional negative puncture $z$ asymptotic to a non-contractible closed Reeb orbit $\gamma_z$. Note that all other punctures of $w_\infty$ are asymptotic to closed Reeb orbits with Conley--Zehnder index equal to $1$. This follows from Claim~3 together with inequality \eqref{Equation: Proof of the GW upper bound - 4}. The $L^2$-energy for $w_\infty$ in \eqref{Equation: Hofer energy on V in action} gives
\[
    0 < E(w_\infty)  \le 2\pi-\mathcal{A}(\gamma_k)-\mathcal{A}(\gamma_z).
\]
Therefore, $\gamma_k$ satisfies the second criterion in the definition of
$\mathcal{P}_{\rm HZ}(W)$. This completes the proof assuming Claim~3. \end{proof}

We now prove the claim.
\begin{proof}[Proof of Claim 3]

We prove the claim by induction starting from the leaves of the building. By
Claim~1, bubble curves associated to leaves are non-constant finite-energy $J$-holomorphic curves in either $\widehat{W}$ or $\mathbb{R}\times Y$ with precisely one positive puncture. Although these curves may be multiply covered, the Fredholm index formula for the underlying simple curves implies that the asymptotic orbit at the positive puncture has positive Conley--Zehnder index. Moreover, if the curve lies in $\mathbb{R}\times Y$, this index is in fact at least $2$ due to the $\mathbb{R}$-translation.

Next, consider a $J$-holomorphic curve connected to a leaf curve through its positive puncture. Such a curve is necessarily contained in $\mathbb{R}\times Y$. Applying the same argument as in the inductive step for $\gamma_i$, we conclude that the asymptotic orbit at its unique positive puncture has positive Conley--Zehnder index, and it is precisely $1$ only if the curve is a branched cover of a trivial cylinder. Inductively, the same argument applies to every curve obtained by following the positive punctures up the bubbling tree. Hence, every closed Reeb orbit appearing as the asymptotic limit of a bubbling curve has positive Conley--Zehnder index.
\end{proof}

The proof of Theorem \ref{Theorem: Upper bound from index 3 - nondegenerate W} narrows down the possible SFT limits. First, there exists a prime closed Reeb orbit $\gamma_0$ such that, each principal curve in $\mathbb{R} \times Y$ is a branched cover of the trivial cylinder $u_{\gamma_0}$. Moreover, the bubbling tree attached to a bubbling point on these principal curves consists of either branched covers of trivial cylinders or possibly multiply-covered $J$-holomorphic planes in $\widehat{W}$.

The Conley--Zehnder indices of the Reeb orbits connecting the principal curves satisfy \eqref{Equation: Proof of the GW upper bound - 5}. This automatically implies that all Reeb orbits arising from the bubbling curves have Conley--Zehnder index $1$. Consequently, whenever bubbling occurs among the principal curves in
$\mathbb{R}\times S^1$, every asymptotic limit $\gamma_i$ is multiply covered.  Moreover, if $\gamma_1$ is a prime closed Reeb orbit, then no bubbling occurs, and the middle building consists of a single level given by the trivial cylinder $u_{\gamma_0}$.

For each puncture $z\in\Gamma$, there exists a prime closed Reeb orbit $\gamma_{z,0}$ such that the corresponding negative puncture is asymptotic to $\gamma_{z,0}$. The bubbling curves attached to these punctures have the same description as above, with $\gamma_{z,0}$ replacing $\gamma_0$. In this case, the Conley--Zehnder index of $\gamma_z$ is $1$, which again forces all Reeb orbits arising from the bubbling curves to have Conley--Zehnder index $1$. One difference is that, if $\gamma_z$ is a prime closed Reeb orbit, the middle building consists of a single level given by the trivial cylinder $u_{\gamma_z}$.

\medskip

\noindent\textbf{Proof of Theorem~\ref{Theorem: Upper bound from index 3 - degenerate W}.} We finish the section by proving Theorem \ref{Theorem: Upper bound from index 3 - degenerate W}.
\begin{proof}[Proof of Theorem \ref{Theorem: Upper bound from index 3 - degenerate W}]

Let $\alpha_i$ be a sequence of nondegenerate\footnote{We say that a contact form is nondegenerate if its Reeb flow is nondegenerate.} contact forms converging to
a degenerate contact form $\alpha$ in the $C^\infty$-topology. Each $\alpha_i$ determines a fiberwise star-shaped domain $W_i\subset\operatorname{int} D^*_{g_0}S^2)$ such that the contact boundary $(\partial W_i,\lambda_{\rm taut})$ is strictly contactomorphic to $(Y,\alpha_i)$. We further assume that
\[
    W_1\supsetneq W_2\supsetneq\cdots.
\]

By Theorem \ref{Theorem: Upper bound from index 3 - nondegenerate W}, for each $i$ there exists a closed Reeb orbit $\gamma_i$ satisfying
\[
c_{\rm HZ}(W_i) \le \mathcal{A}(\gamma_i)\quad\text{and}\quad \mu_{\rm CZ}(\gamma_i)=3,
\]
and either of the following two cases holds:
\begin{itemize}
    \item $\gamma$ is non-contractible in $\partial W$ and $\mathcal{A}(\gamma_i) \le 2\pi$;
    \item $\gamma$ is contractible in $\partial W$, and there exists a
    non-contractible orbit $\gamma'$ satisfying
    \[
        {\rm CZ}(\gamma_i')=1
        \qquad\text{and}\qquad
        \mathcal{A}(\gamma_i)+\mathcal{A}(\gamma_i')\leq 2\pi.
    \]
\end{itemize}

Passing to a subsequence, we may assume that the same one of the two cases holds for every $i$ and that $\gamma_i$ converges to a closed Reeb orbit $\gamma$ for the Reeb flow $R_\alpha$. Note that the limit of a sequence of closed Reeb orbits preserves whether the orbits are contractible or non-contractible, and that actions converge to the action of the limit. Hence, $\gamma$ satisfies the corresponding one of the two cases. Note that, in the second case, $\gamma_i'$ converges, after passing to a further subsequence, to a non-contractible index-$1$ Reeb orbit $\gamma'$ for $R_\alpha$.

By monotonicity, the Hofer--Zehnder capacity of $W$ is bounded above by
that of each $W_i$. Therefore,
\[
    c_{\rm HZ}(W) \le  c_{\rm HZ}(W_i)\leq\mathcal{A}(\gamma_i).
\]
Taking the limit along the chosen subsequence, we obtain
\[
    c_{\rm HZ}(W)\leq\mathcal{A}(\gamma).
\]
Finally, since $\gamma$ is obtained as a limit of the index-$3$ orbits
$\gamma_i$, the indices satisfy
\[
    \mu^-(\gamma)\leq 3\leq\mu^+(\gamma).
\]
and
\[
\mu^-(\gamma')\leq 1\leq\mu^+(\gamma')
\]
in the contractible case.
    
\end{proof}

\subsection{Proof of Theorem~\ref{Theorem: Upper bound from index 3 - refined + geodesic flow version} and Corollary~\ref{cor:ellipsoid-index-three-upper-bound}}
\label{Section: Positive curvature}

We now translate Theorem~\ref{Theorem: Upper bound from index 3 - degenerate W}
to disk cotangent bundles of positively curved Riemannian two-spheres.

For an immersed closed curve $\gamma$ on $S^2$ with only transverse self-intersections, let $s(\gamma)\in\mathbb Z_2$ denote the parity of
the number of self-intersections. After parametrizing $\gamma$ by unit speed, its tangent lift $(\gamma,\dot\gamma)\in T^1S^2\cong\mathbb{RP}^3$ is contractible if and only if $s(\gamma)=1$. This follows by regularly homotoping $\gamma$ to a multiple cover of the equator. 

We denote by $\mu(\gamma)$ its Morse index as a critical point of the energy functional $\mathcal E_g$ and by $\nu(\gamma)$ its reduced nullity. 
Then, it is known that
\[
    \mu_-(\gamma) = \mu(\gamma),
    \quad \text{and} \quad
    \mu_+(\gamma) = \mu(\gamma)+\nu(\gamma).
\]

Finally, by the solution of Weyl's isometric embedding problem
\cite{Nir53}, every positively curved Riemannian two-sphere $(S^2,g)$ admits an isometric embedding into $\mathbb R^3$ as the boundary of a strictly convex body. We denote by $R=R(S^2,g)$ the circumradius of this convex body. After translating its circumcenter to the origin and rescaling by $R^{-1}$, the convex body is contained in the unit ball, and the induced metric on its boundary is $R^{-2}g$. Hence,
\[
    D^*_{R^{-2}g}S^2 \subset D^*_{g_0}S^2.
\]
We can therefore apply Theorem~\ref{Theorem: Upper bound from index 3 - degenerate W}. Since
\[
    \ell_{R^{-2}g}(\gamma)=R^{-1}\ell_g(\gamma),
\]
the action bound $2\pi$ in that theorem becomes $2\pi R$ after
rescaling back to $g$. Together with the above correspondence between
Reeb orbits and closed geodesics, this immediately implies
Theorem~\ref{Theorem: Upper bound from index 3 - refined + geodesic flow version}.

\begin{theorem}[Thm.~\ref{Theorem: Upper bound from index 3 - refined + geodesic flow version}]
Let $g$ be a Riemannian metric of positive curvature on $S^2$. Then
\begin{equation}
    c_{\rm HZ}(D_g^*S^2)
    \leq
    \sup_{\gamma\in\mathcal P_{\rm HZ}(g)}
    \ell_g(\gamma),
\end{equation}
where $\mathcal P_{\rm HZ}(g)$ denotes the set of closed geodesics
$\gamma$ satisfying
\[
    \operatorname{ind}(\gamma)
    \leq 3
    \leq
    \operatorname{ind}(\gamma)+\operatorname{nul}(\gamma),
\]
and one of the following two conditions:
\begin{itemize}
    \item $s(\gamma)=0$, equivalently, $\gamma$ has even number of self-intersections, and
    \[
        \ell_g(\gamma)
        \leq
        2\pi R(S^2,g);
    \]
    \item $s(\gamma)=1$, equivalently, $\gamma$ has odd
    number of self-intersection, and
    \[
        \ell_g(\gamma)+\operatorname{sys}(S^2,g)
        \leq
        2\pi R(S^2,g).
    \]
\end{itemize}
\end{theorem}

We now proceed proving Corollary \ref{cor:ellipsoid-index-three-upper-bound}.

\begin{corollary}[Cor.~\ref{cor:ellipsoid-index-three-upper-bound}]
Let $E=E(a,b,c)$ with $a\geq b\geq c>0$, and let $\eta$ be a shortest closed geodesic on $E$ of Morse index $3$. Then
\[
c_{HZ}(D_1E,\dd\lambda)\leq\ell(\eta).
\]
\end{corollary}

\begin{proof}
We apply Theorem~\ref{Theorem: Upper bound from index 3 - refined + geodesic flow version}.
Recall that $\eta$ is either the major principal ellipse $\gamma_z$ or
belongs to the family of closed geodesics intersecting $\gamma_z$ four
times characterized by $\upomega=2$. In particular,
$\eta\in\mathcal P_{\rm HZ}(g)$. Although $\eta$ need not be the only
element of $\mathcal P_{\rm HZ}(g)$. Using the independent upper bound
of Theorem~\ref{thmSH}, we may further assume that
\[
    \ell(\eta)<2\ell(\gamma_x)=2\operatorname{sys},
\]
since otherwise Theorem~\ref{thmSH} already gives the desired result.

Arguing by contradiction, suppose that the positive asymptote
$\gamma_o$ of the bottom level of the building projects to a closed
geodesic, again denoted by $\gamma_o$, different from $\eta$. As
discussed in Appendix~\ref{appendix: morse indices}, such a geodesic can
only be a multiple cover $\gamma_x^q$ of $\gamma_x$ or belong to the
$S^1$-family of closed geodesics $\gamma_{x,q}$ bifurcating from
$\gamma_x^q$. Note that
\[
    \ell(\gamma_{x,q})>\ell(\gamma_x^q)=q\operatorname{sys}.
\]
Thus, if $q\geq3$, then
\[
    3\operatorname{sys}
    \leq \ell(\gamma_o)
    \leq 2\pi a,
\]
where in the last inequality we used that the circumradius of
$E(a,b,c)$ is $a$. If $q=2$, then $s(\gamma_o)=1$, so there is an
additional negative puncture in the top level of the building.
Consequently, the action bound gives
\[
    3\operatorname{sys}
    \leq \operatorname{sys}+\ell(\gamma_o)
    \leq 2\pi a.
\]
Hence, in either case,
\[
    3\operatorname{sys}\leq2\pi a.
\]
We now show that this is incompatible with
$\ell(\eta)<2\operatorname{sys}$. Indeed,
\[
    4\sqrt{a^2+b^2}
    \leq \ell(\eta)
    <2\operatorname{sys}
    \leq4\pi b.
\]
The first inequality follows from the fact that $\eta$ connects
\[
    (a,0,0)\to(0,b,0)\to(-a,0,0)\to(0,-b,0)\to(a,0,0).
\]
It follows that
\[
    a^2\leq(\pi^2-1)b^2.
\]
On the other hand,
\[
    \ell(\eta)
    <\frac{2}{3}\,3\operatorname{sys}
    \leq\frac{4\pi a}{3}.
\]
Since again $\ell(\eta)\geq4\sqrt{a^2+b^2}$, we obtain
\[
    b^2
    <
    \left(\left(\frac{\pi}{3}\right)^2-1\right)a^2.
\]
Combining the two inequalities gives
\[
    a^2
    <
    (\pi^2-1)
    \left(\left(\frac{\pi}{3}\right)^2-1\right)a^2
    \approx0.857\,a^2,
\]
a contradiction. Hence the positive asymptote must project to $\eta$,
and the desired bound follows.

\end{proof}

\newpage

\section{Step 3: Symplectic homology}

There are several ways to obtain upper bounds for the Hofer--Zehnder
capacity of a Liouville domain $(W,\dd\lambda)$, and in particular of a disk cotangent bundle
$(D^*N,d\lambda)$, using symplectic homology. The basic mechanisms currently known are:
\begin{itemize}
  \item vanishing of $\SH_*(W;\mathcal L)$ for a suitable local system
        $\mathcal L$;
  \item the existence of a (suitable) pair-of-pants product $x * y = 1$ in $\SH_*(W)$;
  \item the existence of a dilation class $\Delta x = 1$ in $\SH_*(W)$.
  \item the existence of a higher dilation, i.e.\ the vanishing of the unit in a suitable completed $S^1$-equivariant/periodic version of symplectic cohomology.
\end{itemize}

For general Liouville domains, the fact that vanishing of symplectic homology
implies finiteness of the Hofer--Zehnder capacity is due to Irie
\cite{Irie14}. For cotangent bundles, Viterbo's isomorphism
\[
  \SH_*(D^*N;\mathbb Z_2) \cong H_{*}(\Lambda N;\mathbb Z_2)
\]
rules out vanishing with untwisted coefficients. However, results of
Albers--Oancea--Frauenfelder \cite{AFO17} show that in some cases (for example if
$\pi_2(N)\neq 0$) one can choose a local system $\sL$ on
$\Lambda N$ such that $\SH_*(D^*N;\sL)=0$, and finiteness then follows. The pair-of-pants argument is available only under additional
assumptions: it is proved when $x$ and $y$ are represented by orbits of non-contractible loops \cite{Irie14}, or in
the special case $x=\Delta(\tilde x)$ \cite{BC25}. The dilation class argument is the corresponding special case $y=1$ and $x=\Delta(\tilde x)$. Finally, Zhao introduced in her thesis \cite{Zhao16} the notion of a higher dilation, which roughly means that the unit vanishes in completed periodic symplectic cohomology. Through its relation with $S^1$-equivariant symplectic homology, this provides another route to finiteness of the Hofer--Zehnder capacity, closely related to the criterion of Frauenfelder--Pajitnov \cite{FP17}.\\

The rule of thumb seems to be that one looks for a structure on $\SH_*(W)$ that forces a form of \emph{uniruledness} of $(W,d\lambda)$, i.e.\ guarantees the existence of a pseudo-holomorphic curve through every point. In the spirit of Hofer--Viterbo \cite{HV92}, discussed in detail in the previous section, such a property is expected to yield an upper bound on the Hofer--Zehnder capacity and can be viewed as a non-compact analogue of a non-vanishing Gromov--Witten invariant. So far, these approaches have mostly been used to prove finiteness of $c_{HZ}$, but in fact they can be upgraded to quantitative upper bounds by incorporating the action filtration on symplectic homology. We now explain this more precisely in the pair-of-pants case. In the ellipsoid case, this will yield the second upper bound $2\mathrm{sys}$.

\begin{remark}
    We expect that vanishing of symplectic homology with local coefficients, as in \cite{AFO17}, can be used to recover the previous upper bound, and more generally to bound the Hofer--Zehnder capacity of $(D^*_FS^2, \dd\lambda)$ in terms of a diastolic quantity defined by a min--max over $3$-parameter families of loops. We will address this in future work.
\end{remark}

\noindent
\textbf{Symplectic homology.} We first recall the version of symplectic homology which will be used in the proof. We follow the convention of Irie \cite{Irie14}. Let $(W,\lambda)$ be a Liouville domain and let
\[
  \widehat W = W \cup_{\partial W} \bigl(\partial W \times [1,\infty)\bigr)
\]
be its completion. On the cylindrical end we write the completed Liouville form
as
\[
  \widehat\lambda = r \lambda|_{\partial W},
  \qquad r\in [1,\infty).
\]
In the case relevant to us, $W=D^*N$, this completion is canonically identified
with the full cotangent bundle $T^*N$ equipped with its standard Liouville form. We consider Hamiltonians $H \in C^\infty(S^1\times \widehat W)$, which are \emph{linear at infinity}: there exist constants $a_H>0$, $b_H\in
\mathbb R$, and $r_0\geq 1$ such that
\[
  H_t(z,r)=a_H r+b_H,
  \qquad (t,z,r)\in S^1\times \partial W\times [r_0,\infty).
\]
The number $a_H$ is called the slope of $H$. We always assume that $a_H\notin \operatorname{Spec}(W,\lambda)$, where $\operatorname{Spec}(W,\lambda)$ denotes the set of periods of Reeb orbits on $(\partial W,\lambda|_{\partial W})$. This condition ensures that no $1$-periodic Hamiltonian orbits appear on the cylindrical end. After a small compactly supported perturbation, we may also assume that all $1$ periodic orbits are nondegenerate.\\

\noindent
For such a Hamiltonian, the Floer chain complex is generated over
$\mathbb Z_2$ by the $1$-periodic Hamiltonian orbits of $H$. More precisely, for an interval $I=(a,b)\subset \mathbb R$, one sets
\[
  \CF^{(a,b)}_*(H)
  :=
  \bigoplus_{\substack{x\in \mathcal P(H)\\ \mathcal A_H(x)\in I \\ \operatorname{CZ}(x)=*}}
  \mathbb Z_2\langle x\rangle,
\]
where $\mathcal P(H)$ denotes the set of 1-periodic orbits, $\operatorname{CZ}(x)$ is the Conley--Zehnder index and the action functional is
\[
  \mathcal A_H(x)
  =
  \int_{S^1} x^*\widehat\lambda
  -
  \int_0^1 H_t(x(t))\,dt .
\]
Choosing a generic almost complex structures of contact type on the cylindrical end, counting solutions to the usual Floer equation defines a differential $\partial_H \colon \CF_*^{I}(H)\to \CF_{*-1}^{I}(H)$. Its homology will be denoted $\HF^{I}(H)$. We also write $\HF^{<a}(H):=\HF^{(-\infty,a)}(H)$ and $\HF(H)=\HF^\R(H)$. If $H_-\leq H_+$ are two admissible Hamiltonians with slopes
$a_{H_-}\leq a_{H_+}$, continuation maps give natural homomorphisms $\HF^{I}(H_-)\rightarrow \HF^{I}(H_+)$, compatible with composition. The symplectic homology of $(W,\lambda)$ is then
defined by taking the direct limit over Hamiltonians which are negative on
$W$ and whose slopes tend to infinity:
\[
  \SH^{I}(W,\lambda)
  :=
  \varinjlim_H \HF^{I}(H).
\]
In particular,
\[
  \SH^{<a}(W,\lambda)
  :=
  \varinjlim_H \HF^{<a}(H),
  \qquad
  \SH(W,\lambda)
  :=
  \varinjlim_a \SH^{<a}(W,\lambda).
\]
Irie's Lemma~2.5 shows that, if $H$ is admissible and has slope
$a_H$, then there is an isomorphism
\[
  \Psi_H\colon \HF(H)\xrightarrow{\cong} \SH^{<a_H}(W,\lambda)
\]
Thus the part of symplectic homology with action $<a_H$ may be represented by
the ordinary Floer homology of a single Hamiltonian with finite slope $a_H$.\\

\noindent
\textbf{Spectral invariants.} We next recall the spectral invariants used by Irie. Let $(W,\lambda)$ be a
Liouville domain, and let $H\in C^\infty (S^1\times \hat W)$ be linear at infinity of slope $a_H\notin \operatorname{Spec}(W,\lambda)$. For $c\in\mathbb R$, let
\[
  \iota_H^c: \HF_*^{<c}(H)\longrightarrow \HF_*(H)
\]
be the map induced by inclusion of the filtered Floer complex. If
$x\in \SH_*^{<a_H}(W,\lambda)$, define
\[
  \rho(H:x)
  :=
  \inf\left\{
    c\in\mathbb R
    \ \middle|\
    \Psi_H^{-1}(x)\in \operatorname{im}(\iota_H^c)
  \right\}.
\]
Equivalently, $\rho(H:x)$ is the smallest action level below which the class
$\Psi_H^{-1}(x)$ can be represented by a Floer cycle. The definition of $\rho(H:x)$ extends to Hamiltonians which are linear at
infinity but possibly degenerate. Namely, if $H$ is linear at infinity and $a_H\notin\operatorname{Spec}(W,\lambda)$, one chooses admissible Hamiltonians
$H_j$ such that $H_j-H$ is compactly supported and $\|H_j-H\|\to 0$. Then one defines
\[
  \rho(H:x):=\lim_{j\to\infty}\rho(H_j:x).
\]
The Lipschitz estimate below implies that this limit exists and is independent
of the chosen approximation.

\begin{lemma}[Irie, Lemma~3.2]
\label{lem:irie-spectral-invariants}
Let $H$ be a Hamiltonian on $\widehat W$ which is linear at infinity and whose
slope satisfies $a_H\notin \operatorname{Spec}(W,\lambda)$. Then the following hold.

\begin{enumerate}
  \item \emph{Spectrality.} If $0\neq x\in \SH_k^{<a_H}(W,\lambda)$, then $\rho(H:x)\in \operatorname{Spec}_k(H)$, where $\operatorname{Spec}_k(H)$ denotes the set of actions of $1$-periodic Hamiltonian orbits in the free homotopy class $\alpha$ of Conley--Zehnder index $k$.

  \item \emph{Lipschitz continuity.} Let $K$ be another Hamiltonian which is
  linear at infinity, with slope outside $\operatorname{Spec}(W,\lambda)$, and
  assume that $H-K$ is compactly supported. Then, for every nonzero
  $x\in \SH_*^{<a_H}(W,\lambda)$,
  \[
    \bigl|\rho(H:x)-\rho(K:x)\bigr|
    \leq
    \|H-K\|,
  \]
  where
  \[
    \|H-K\|
    :=
    \int_0^1
    \left(
      \max_{\widehat W}(H_t-K_t)
      -
      \min_{\widehat W}(H_t-K_t)
    \right)\,dt .
  \]
\end{enumerate}
\end{lemma}

\noindent
For sufficiently small
$\delta>0$ there is a natural identification $H_*(W,\partial W)\cong \SH_*^{<\delta}(W,\lambda)$. For every $c>0$ this gives a natural homomorphism
\[
  \iota_c:H^*(W,\Z_2)
  \longrightarrow
  \SH_*^{<c}(W,\lambda).
\]
Following Irie, we set
\[
  F_c:=\iota_c(1)\in \SH_*^{<c}(W,\lambda),
\]
where $1\in H^0(W;\mathbb Z_2)$ is the cohomological unit. We shall also need the following computation of the spectral invariant of the
class $F_c$ for Hofer--Zehnder admissible Hamiltonians. We state it in the form
used by Irie. Let $H\in C_0^\infty(\operatorname{int} W)$ and let
$\nu\in C^\infty([1,\infty))$. For $a\in\mathbb R$ define $H_{a,\nu}:S^1\times \widehat W\longrightarrow \mathbb R$ by
\begin{equation}
\label{extension}
  H_{a,\nu}(t,x)
  :=
  \begin{cases}
    aH(x), & x\in \operatorname{int} W,\\
    \nu(r), & x=(z,r)\in \partial W\times [1,\infty).
  \end{cases}
\end{equation}

\begin{proposition}[Irie, Proposition~3.3]
\label{prop:irie-fundamental-spectral-value}
Let $H\in C_0^\infty(\operatorname{int} W)$ and
$\nu\in C^\infty([1,\infty))$. Assume:
\begin{enumerate}
  \item there exists $r_0>1$ such that $\nu(r)\equiv 0$ for $r\in [1,r_0]$;

  \item there exist $r_1>1$ and $a_\nu\in (0,-\min H)\setminus \operatorname{Spec}(W,\lambda)$ such that $\nu'(r)\equiv a_\nu$ for $r\in [r_1,\infty)$;
  \item one has $ S(\nu):=\sup_{r\geq 1}\bigl(r\nu'(r)-\nu(r)\bigr)<-\min H$.
\end{enumerate}
If $H$ is Hofer--Zehnder admissible with respect to the contractible class
$c_W$, then
\[
  \rho(H_{1,\nu}:F_{a_\nu})=-\min H.
\]
\end{proposition}

We will need the following slight improvement of Irie's proposition.

\begin{lemma}
\label{lem:irie-degree-improvement}
Let $H$ and $\nu$ satisfy the assumptions of Proposition~\ref{prop:irie-fundamental-spectral-value}. If $k\neq n$ and $0\neq x\in \SH^{<a_\nu}_k(W,\lambda)$, then
$\rho(H_{1,\nu}:x)<-\min H$.
\end{lemma}

\begin{proof}
As in the proof of \cite[Proposition~3.3]{Irie14}, we may assume, by a $C^\infty$-small approximation and by the Lipschitz continuity of spectral invariants, that $\min H$ is isolated in the set of critical values of $H$, that it is attained at a unique point $p_H\in W$, and that the corresponding constant orbit is nondegenerate. We also assume that $H_{1,\nu}$ is nondegenerate. 

We first recall the relevant action estimates for the 1-periodic orbits of the Hamiltonian $H_{1,\nu}$. The $1$-periodic orbits in the interior which are constant are precisely the critical points of $H$, and their actions are $-H(q)$. Hence their actions are at most $-\min H$, with equality only for the constant orbit at the unique minimum $p_H$. By our grading convention this orbit has degree $n$. The nonconstant $1$-periodic orbits on the cylindrical end have action of the form $r\nu'(r)-\nu(r)$. By assumption, $S(\nu)=\sup_{r\geq 1}(r\nu'(r)-\nu(r))<-\min H$. Thus every orbit on the cylindrical end has action strictly smaller than $-\min H$. Finally, since $H$ is Hofer--Zehnder admissible with respect to the contractible class, there are no nonconstant contractible $1$-periodic orbits in the interior. Therefore, the only orbit of action $-\min H$ is the constant orbit at $p_H$, and this orbit has degree $n$.

Now let $0\neq x\in \SH^{<a_\nu}_k(W,\lambda)$ with $k\neq n$. By Irie's identification $\Psi_{H_{1,\nu}}:\HF(H_{1,\nu})\to \SH^{<a_\nu}(W,\lambda)$, we may regard $x$ as a nonzero class in $\HF_k(H_{1,\nu})$. By spectrality, $\rho(H_{1,\nu}:x)$ is the action of a $1$-periodic orbit of $H_{1,\nu}$ of degree $k$. By the action estimates above, every such orbit has action at most $-\min H$, and the only orbit with action exactly $-\min H$ has degree $n$. Since $k\neq n$, the spectral value cannot be equal to $-\min H$. Hence $\rho(H_{1,\nu}:x)<-\min H$, as claimed.
\end{proof}

\noindent
\textbf{Product.} We recall the pair-of-pants product in a form which allows the two input Hamiltonians to have different slopes from \cite[Ap.\ 16]{Ritter13}. Let $H\colon \widehat W\to \mathbb R$ be an admissible Hamiltonian, linear at infinity. For positive numbers
$a,b,c>0$, consider the Hamiltonians $aH, bH, cH$. The pair-of-pants product is defined by counting solutions on a three-punctured
sphere $P$ with two incoming ends and one outgoing end. Following Ritter, one
chooses a one-form $\beta\in\Omega^1(P)$ such that $\dd\beta\leq 0$, and such that, on the cylindrical ends, $\beta$ is equal to a constant multiple
of $dt$. For the product with inputs $aH$ and $bH$, and output $cH$, the
weights are chosen to be $a, b, c$. The condition for such a one-form $\beta$ to exist is $c\geq a+b$. The corresponding Floer equation is $(du-X_H\otimes \beta)^{0,1}=0$. On the cylindrical ends this becomes the usual Floer equation for the
Hamiltonians $aH$, $bH$, and $cH$, respectively. For regular Floer data, counting rigid solutions gives a chain map
\[
  CF_*(aH)\otimes CF_*(bH)
  \longrightarrow
  CF_*(cH).
\]
The energy identity implies that if an orbit $z$ of $cH$ occurs in the
product of an orbit $x$ of $aH$ and an orbit $y$ of $bH$, then $\mathcal A_{cH}(z)\leq \mathcal A_{aH}(x)+\mathcal A_{bH}(y)$. Hence the product is compatible with the action filtration:
\[
  \HF_*^{<A}(aH)\otimes \HF_*^{<B}(bH)
  \longrightarrow
  \HF_*^{<A+B}(cH).
\]
Passing to the direct limit gives the filtered
pair-of-pants product
\[
  *:
  \SH_*^{<A}(W,\lambda)\otimes \SH_*^{<B}(W,\lambda)
  \longrightarrow
  \SH_*^{<A+B}(W,\lambda).
\]
Therefore, if $x$ and $y$ can be represented below the action levels $A$ and $B$, respectively, then $x*y$ can be represented below $A+B$. Taking the infimum over all such representatives gives
\begin{equation}
\label{eq:spectral-product-estimate}
  \rho(cH:x*y)
  \leq
  \rho(aH:x)+\rho(bH:y).
\end{equation}

\noindent
\textbf{BV-operator.} We will also use the BV-operator on symplectic homology. It can be defined at chain level by counting Floer cylinders with one rotating asymptotic marker;
see Abouzaid \cite[Chapter~10]{Abouzaid15} for this construction in the cotangent bundle setting. For general Liouville domains, the same operation is
part of the standard BV structure on symplectic homology, equivalently arising from the $S^1$-equivariant formalism of Bourgeois--Oancea \cite{BO13}. With our homological grading convention it has degree $+1$, so
$$
\Delta:\SH_k(W,\lambda)\to \SH_{k+1}(W,\lambda).
$$
On the chain level, the BV-operator can be described as follows. Fix a non-degenerate Hamiltonian $H_t$ and write $H^\theta_t=H_{t+\theta}$ for the Hamiltonian obtained by rotating the time variable. For each $\theta\in S^1$, choose a homotopy of Floer data $(H^\theta_{s,t},J^\theta_{s,t})$ from $(H_t,J_t)$ as $s\to -\infty$ to $(H^\theta_t,J^\theta_t)$ as $s\to +\infty$. Thus $H^\theta_{s,t}=H_t$ for $s\ll 0$ and $H^\theta_{s,t}=H_{t+\theta}$ for $s\gg 0$. Given Hamiltonian orbits $x,y\in\mathcal P(H)$, let
$\mathcal M_\Delta(x;y)$ be the space of pairs $(\theta,u)$, with $\theta\in S^1$ and $u:\mathbb R\times S^1\to \widehat W$, solving the parametrized Floer equation $\partial_s u+J^\theta_{s,t}(u)(\partial_tu X_{H^\theta_{s,t}}(u))=0$, with asymptotic conditions $\lim_{s\to-\infty}u(s,t)=x(t)$ and $\lim_{s\to+\infty}u(s,t)=y(t+\theta)$. Here the orbit $t\mapsto y(t+\theta)$ is naturally a $1$-periodic orbit of $H^\theta$. After quotienting by the usual $\mathbb R$-translation in the $s$-variable and
choosing regular Floer data, the zero-dimensional component
$\mathcal M_\Delta^0(x;y)$ is finite. Over $\mathbb Z_2$ one defines
$$
\Delta_H:CF_k(H)\to CF_{k+1}(H)\quad \text{by}\quad
\Delta_H x=\sum_y \#_2\mathcal M_\Delta^0(x;y)\,y,
$$
where $\#_2$ denotes the mod-$2$ count. The compactification of the one-dimensional parametrized moduli spaces gives
$\partial_H\Delta_H+\Delta_H\partial_H=0$. Hence $\Delta_H$ induces a map on
Floer homology,
$\Delta_H:\HF_k(H)\to\HF_{k+1}(H)$.
This construction is compatible with continuation maps, and passing to the
direct limit gives the BV operator
$\Delta:\SH_k(W,\lambda)\to\SH_{k+1}(W,\lambda)$. The energy identity for these cylinders is the same as for the Floer
differential. Hence, whenever $y$ appears in $\Delta_H x$, one has
$\mathcal A_H(y)\leq \mathcal A_H(x)$. Thus $\Delta_H$ is action
non-increasing: $\Delta_H CF_*^{<A}(H)\subset CF_{*+1}^{<A}(H)$, and
consequently 
$$\rho(H:\Delta x)\leq \rho(H:x).$$
In the special case $W=D^*N$, Abouzaid proves that this operator agrees, under
Viterbo's isomorphism, with the BV operator on $H_{*+n}(\Lambda N)$ induced by
loop rotation \cite[Chapters~11--12]{Abouzaid15}.\\

\noindent
\textbf{Upper bound.} The following theorem is a simplified combination of
\cite[Thm.~7 and Thm.~8]{BC25}. We give a simpler prove, only using spectral invariants in spirit of Irie \cite{Irie14}. Note, that the results in \cite{BC25} are not primarily proved for the Hofer--Zehnder capacity and that our proof only works if one hits the unit (fundamental class) via PSS-map.

\begin{theorem}\label{upper bound 2}
Let $(W,\lambda)$ be a Liouville domain and denote for some $a,b>0$ and $c> a+b$, $F_{c}:=\iota_{c}(1)\in \HF^{<c}(W,\lambda)$, where $1\in H^0(W;\mathbb Z_2)$ is the canonical generator. Assume that there exist classes $x\in \SH^{<a}(W,\lambda)$, $y\in \SH^{<b}(W,\lambda)$ such that either
\[
\Delta x * y = F_{c}
\quad\text{in}\quad
\SH^{<c}(W,\lambda),
\]
or $\deg(x)<n$ and 
\[
 x * y = F_{c}
\quad\text{in}\quad
\SH^{<c}(W,\lambda).
\]
Then
\[
c_{HZ}(\operatorname{int}X,d\lambda)\leq c.
\]
\end{theorem}

\begin{proof}
We only prove the slightly more difficult first case, reading the prove it will become clear that the second case is completely analogous. Suppose, for a contradiction, that $c_{HZ}(\operatorname{int}W,d\lambda)>c$. Then there exists a Hofer--Zehnder admissible Hamiltonian $H\in C_0^\infty(\operatorname{int}W)$ such that $-\min H>c$.
Choose $\nu\in C^\infty([1,\infty))$ as in Proposition~\ref{prop:irie-fundamental-spectral-value}, with slope $a_\nu=c\notin \operatorname{Spec}(W,\lambda)$. Thus 
$$S(\nu):=\sup_{r\geq 1}(r\nu'(r)-\nu(r))<-\min H.$$
Consider the Hamiltonian $H_{1,\nu}$ defined by \eqref{extension}. By Irie's Proposition~3.3, $\rho(H_{1,\nu}:F_c)=-\min H$. Now set $c_a=\tfrac{a}{c}$ and $c_b=\tfrac bc$. Since $c>a+b$, we have $c_a+c_b<1$. By Irie's Lemma~2.5, $\HF_*(c_a H_{1,\nu})\cong \SH_*^{<a}(W,\lambda)$ and $\HF_*(c_b H_{1,\nu})\cong \SH_*^{<b}(W,\lambda)$. By assumption, there exist $x\in \HF_*(c_a H_{1,\nu})$ and $y\in \HF_*(c_b H_{1,\nu})$ such that
\[
  \Delta(x)\star y=F_c\in \HF_n(H_{1,\nu}).
\]
Note that either $x$ or $y$ has degree $k<n$, since $\deg(x)+1+\deg(y)-n=n$. By Lemma \ref{lem:irie-degree-improvement} the spectral invariant of such a class satisfies $\rho(c_aH_{1,\nu},x)<-c_a\min H$, and by the BV estimate,
\[
  \rho(c_aH_{1,\nu},\Delta x)\leq \rho(c_aH_{1,\nu},x)<-c_a\min H.
\]
Similarly, if $\deg(y)<n$, then spectrality gives
\[
  \rho(c_bH_{1,\nu},y)<-c_b\min H.
\]
In either case, one of the two summands below is strictly smaller than its top constant-orbit action, while the other is at most its top constant-orbit action. Therefore
\[
    -\min H
    =
    \rho(H_{1,\nu},F_c)
    \leq
    \rho(c_a H_{1,\nu},\Delta(x))
    +
    \rho(c_b H_{1,\nu},y)
    <
    -(c_a+c_b)\min H< -\min H.
\]
    a contradiction!
\end{proof}

\noindent
Theorem \ref{upper bound 2} will provide the second upper bound needed to finish the proof of our main Theorem \ref{ellipsoid}. Using the filtered version of Abbondandolo--Schwarz isomrphism, we will solve the equation $\Delta x\star y =F_c$ on the string topology side, using cycle representative of length no more than $c=2\operatorname{sys}$. \\

\subsection{Proof of Theorem \ref{thmSH}} Let $(N,g)$ be a Riemannian two-sphere and let $\Lambda N := W^{1,2}(S^1,N)$ be its free loop space. We write $\ell(\gamma)$ for the length of a loop $\gamma\in \Lambda N$ and
\[
\sE(\gamma)=\int_{S^1}|\dot\gamma(t)|_g^2\,dt
\]
for its energy.  We shall use the following notion of sweep-out, following \cite{CCdMOR21}. A \emph{sweep-out} of $N$ is a continuous map $\Gamma:S^1\times S^1\longrightarrow N$ of degree one. Equivalently, writing $\gamma_t(s):=\Gamma(t,s)$, a sweep-out is a one-parameter family $(\gamma_t)_{t\in S^1}$ of closed curves whose associated map $S^1\times S^1\to N$ has degree one. We denote the collection of such sweep-outs by $\mathscr S(N)$ and set

\[
\operatorname{dias}(N,g) := \inf_{\Gamma\in\mathscr S(N)} \sup_{t\in S^1}\ell(\gamma_t).
\]

\noindent
For later comparison with the classical Birkhoff construction, fix the standard parametrization of the parallels of the round sphere,

\[
p_z(s)=\bigl(\sqrt{1-z^2}\cos(2\pi s), \sqrt{1-z^2}\sin(2\pi s), z\bigr), \qquad z\in[-1,1].
\]

\vspace{.3cm}

\begin{theorem}[Thm.\ 1.3, \cite{CCdMOR21}]
\label{thm:monotone-sweepout}
Let $(N,g)$ be a Riemannian two-sphere. Suppose that $\Gamma:S^1\times S^1\to N$ is a sweep-out such that $\ell(\gamma_t)<L$, for all $t\in S^1.$ Then there exists a diffeomorphism $\Phi:S^2\to N$ such that, for every $z\in[-1,1]$, the curve $s\longmapsto \Phi(p_z(s))$ has length strictly smaller than $L$.
\end{theorem}

\begin{remark}
In particular, after this replacement the sweep-out contains constant curves, corresponding to the two poles. Cutting the parameter interval between these constant curves gives a path in $\Lambda N$ with endpoints in the subspace of constant loops. Conversely, any such path can be closed up by joining its endpoints through constant loops. Since these added loops have length zero, the maximal length is unchanged. Thus the sweep-out width agrees with the loop-space diastole of \cite{BS10} and this definition of diastole agrees with the one given in the introduction.
\end{remark}

\begin{remark}
    Abbondandolo--Mazzucchelli show in \cite[Appendix~A]{BK22} that, if $g$ has strictly positive Gaussian curvature, then $\operatorname{dias}(N,g)=\operatorname{sys}(N,g)$.
\end{remark}

The theorem shows that any such sweep-out can be replaced, without increasing the length bound, by a sweep-out obtained from the parallels of the round sphere. Thus, for all $\varepsilon>0$ there is a map
\[
    u:[-1,1]\longrightarrow \Lambda N,
    \qquad
    u(z)(s):=\Phi(p_z(s)),
\]
where $\Phi:S^2\to N$ is a diffeomorphism and $\ell(u(z))\leq \operatorname{dias}(N,g)+\varepsilon$ for all $z\in[-1,1]$, which extends to a sweep-out $\gamma_\Phi:S^1\to \Lambda N$ by connecting the constant endpoints $u(-1)$ and $u(1)$ by a path of constant loops. In particular $\gamma_\Phi$ defines a homology class $[\gamma_\Phi]\in H_1(\Lambda^{\operatorname{dias}(N,g)+\varepsilon}N,\Z_2)$.\\

\noindent
\textbf{Intermezzo: string topology}
We briefly recall the Chas--Sullivan product and the BV operator on chain level from \cite{ChasSullivan1999}. Let $N$ be a closed oriented manifold of dimension $m$ and let $\Lambda N := W^{1,2}(S^1,N)$. Let $P$ and $Q$ be closed (finite dimensional) manifolds, and $A:P\to\Lambda N$ and $B:Q\to\Lambda N$  continuous maps with
$\ev_0\circ A\pitchfork \ev_0\circ B$. Here, $ev_0: \Lambda N \to N;\ \gamma\mapsto \gamma(0)$ denotes the evaluation at $0$. Define a map
\[
A\star B := \#\circ(A\times_N B):P\times_N Q\to\Lambda N,
\]
where
\[
P\times_N Q:=\{(x,y)\in P\times Q\mid \ev_0(A(x))=\ev_0(B(y))\},
\]
and $\#$ denotes the concatenation of loops. Interpreting the maps $A,B$ and $A\star B$ as cycles in singular homology via triangulation this constructions yields the Chas--Sullivan product:
\[
\star :H_p(\Lambda N)\otimes H_q(\Lambda N)\to H_{p+q-m}(\Lambda N).
\]
The inclusion of $N$ as constant loops
\[
e: N\to \Lambda N,\qquad e(x)(t)=x,
\]
represents the unit $1=[e]$ with respect to the Chas--Sullivan product. Finally, the BV operator
$\Delta:H_k(\Lambda N)\to H_{k+1}(\Lambda N)$ is induced by the $S^1$--action
\[
\rho: S^1\times \Lambda N\to \Lambda N; \qquad \rho(\theta,\gamma)(t)=\gamma(t+\theta),
\]
via $\Delta[A]=\rho_*([S^1]\times[A])$.\\

\medskip

\noindent
\textbf{A nontrivial product.} We return to our setting, where $(N,g)$ is a Riemannian two-sphere and $\gamma_\Phi:S^1\to \Lambda N$ is a sweep-out obtained from a diffeomorphism $\Phi:S^2\to N$ as in Theorem~\ref{thm:monotone-sweepout}. We denote by $\bar\gamma_\Phi$ the reversed sweep-out, that is, the sweep-out obtained by traversing each loop of $\gamma_\Phi$ in the opposite direction.

\begin{theorem}\label{2bir}
Let $\varepsilon>0$, and let $\gamma_\Phi:S^1\to\Lambda N$ be a sweep-out obtained from Theorem~\ref{thm:monotone-sweepout} such that $\ell(\gamma_\Phi(t))\leq \operatorname{dias}(N,g)+\varepsilon$ for all $t\in S^1$. Then
\[
    \Delta[\gamma_\Phi]*\Delta[\bar\gamma_\Phi]=1
    \in H_2(\Lambda^{<2\operatorname{dias}(N,g)+2\varepsilon}N,\mathbb Z_2).
\]
\end{theorem}

\begin{proof}
We write the BV representatives explicitly. Let
\[
\rho_\Phi:S^1_\theta\times S^1_\sigma\to \Lambda N,
\qquad
\rho_\Phi(\theta,\sigma)(t)=\gamma_\Phi(\sigma)(t+\theta),
\]
and similarly
\[
\bar\rho_{\Phi}:S^1_\eta\times S^1_\tau\to \Lambda N,
\qquad
\bar \rho_{\Phi}(\eta,\tau)(t)=\gamma_\Phi(\tau)(-t-\eta).
\]
Then $\rho_\Phi$ and $\bar\rho_{\Phi}$ represent $\Delta[\gamma_\Phi]$ and $\Delta[\bar\gamma_\Phi]$, respectively.

\ \\
\noindent
\textbf{Claim 1:} The classes $\Delta[\gamma_\Phi]=[\rho_\Phi]$ and $\Delta[\bar\gamma_\Phi]=[\bar\rho_\Phi]$ can be represented by maps $B_\Phi,\bar B_{\Phi}:S^2\to \Lambda N$ such that $\ev_0\circ B_\Phi=\Phi$ and $\ev_0\circ \bar B_{\Phi}=\bar\Phi$. To see this, recall that $\gamma_\Phi:S^1\to\Lambda N$ is obtained from the path of parallel loops $u:[-1,1]\to\Lambda N$, given by $u(z)(s)=\Phi(p_z(s))$, by connecting the constant endpoint loops $u(-1)$ and $u(1)$ through a path of constant loops. On this added path of constant loops, the representative $\rho_\Phi$ is independent of the rotation parameter. Hence $\rho_\Phi$ is a pinched torus: it factors through the quotient of $S^1\times S^1$ obtained by collapsing the annular part corresponding to the path of constant loops. This quotient is naturally homeomorphic to $S^2$. We denote the resulting map by $B_\Phi:S^2\to\Lambda N$. By construction, $(B_\Phi)_*[S^2]=(\rho_\Phi)_*[S^1\times S^1]$, and hence $B_\Phi$ represents $\Delta[\gamma_\Phi]$. With the standard identification of $S^2$ with the quotient of $[-1,1]\times S^1$ obtained by collapsing the boundary circles, the map $B_\Phi$ is explicitly given by $B_\Phi(p_z(s))(t)=\Phi(p_z(s+t))$. Therefore $\ev_0(B_\Phi(p_z(s)))=\Phi(p_z(s))$, so $\ev_0\circ B_\Phi=\Phi$.

The same argument applies to the reversed sweep-out. The representative $\bar\rho_\Phi$ also collapses the annular part corresponding to the constant loops and therefore factors through a map $\bar B_{\Phi}:S^2\to\Lambda N$. Explicitly, using the convention $\bar\rho_\Phi(\eta,\tau)(t)=\gamma_\Phi(\tau)(t-\eta)$, we obtain $\bar B_{\Phi}(p_z(s))(t)=\Phi(p_z(-s-t))$. Hence $\ev_0(\bar B_{\Phi}(p_z(s)))=\Phi(p_z(-s)):=\bar\Phi(p_z(s))$, and therefore $\ev_0\circ \bar B_{\Phi}=\bar\Phi$. By construction, $\bar B_{\Phi}$ represents $\Delta[\bar\gamma_\Phi]$.

\ \\
\noindent
\textbf{Claim 2:} The maps $\ev_0\circ B_\Phi=\Phi$ and $\ev_0\circ \bar B_\Phi=\bar\Phi$ are transverse. Indeed, since $\Phi$ and $\bar\Phi$ are diffeomorphisms, the map $\Phi\times\bar\Phi:S^2\times S^2\to N\times N$ is a diffeomorphism. Hence it is transverse to the diagonal $\Delta_N\subset N\times N$. Equivalently, $\ev_0\circ B_\Phi$ and $\ev_0\circ \bar B_\Phi$ intersect transversely. Therefore the product $[B_\Phi]\star[\bar B_\Phi]$ is represented by the map $B_\Phi\star\bar B_\Phi:S^2\times_N S^2\to \Lambda N$.

\ \\ \noindent 
\textbf{Claim 3:} The map $B_\Phi\star \bar B_\Phi$ is homotopic to the inclusion of constant loops through loops of length at most $2\operatorname{dias}(N,g)+2\varepsilon$. By definition, a point of $S^2\times_N S^2$ is a pair $(x,y)$ such that $\ev_0(B_\Phi(x))=\ev_0(\bar B_\Phi(y))$. Write $x=p_z(\theta)$ and $y=p_{z'}(\eta)$. If $\bar\Phi(p_z(s)):=\Phi(p_z(-s))$, this condition becomes $\Phi(p_z(\theta))=\bar\Phi(p_{z'}(\eta))=\Phi(p_{z'}(-\eta))$. Since $\Phi:S^2\to N$ is a diffeomorphism, this implies $z=z'$ and $\eta=-\theta$. Hence $S^2\times_N S^2$ is identified with $S^2$ by $p_z(\theta)\mapsto (p_z(\theta),p_z(-\theta))$.  Under this identification, the product map sends $p_z(\theta)$ to the concatenation $B_\Phi(p_z(\theta))\# \bar B_\Phi(p_z(-\theta))$. By construction, this is the loop $s\mapsto \Phi(p_z(\theta+s))$ based at $\Phi(p_z(\theta))$, followed by the same loop with the opposite orientation. Thus it is contractible through based loops to the constant loop at $\Phi(p_z(\theta))$. The length of the concatenated loop is at most $2\ell(\gamma_\Phi(z))\leq 2\operatorname{dias}(N,g)+2\varepsilon$. During the contraction, one may use the standard cancellation homotopy, whose loops have length at most $2\ell(\gamma_\Phi(z))$. Therefore the whole homotopy is contained in $\Lambda^{\leq 2\operatorname{dias}(N,g)+2\varepsilon}N$.  Consequently, $B_\Phi\star \bar B_\Phi$ is homologous in $\Lambda^{\leq 2\operatorname{dias}(N,g)+2\varepsilon}N$ to the map $S^2\to\Lambda N$ given by $p_z(\theta)\mapsto c_{\Phi(p_z(\theta))}$. This map is the inclusion of $N$ as constant loops, composed with the diffeomorphism $\Phi:S^2\to N$. Hence it represents the fundamental class of the constant loops, which is the unit $1\in H_2(\Lambda N;\mathbb Z_2)$. It follows that $\Delta[\gamma_\Phi]\star\Delta[\bar\gamma_\Phi]=1$ in $H_2(\Lambda^{\leq 2\operatorname{dias}(N,g)+2\varepsilon}N;\mathbb Z_2)$. \end{proof}

\noindent
\textbf{Deducing Theorem \ref{thmSH}.}
We have all the pieces together to prove Theorem \ref{thmSH}, that is 
$$
c_{HZ}(D_gS^2,\dd \lambda)\leq 2\mathrm{dias} (g),
$$
for any Riemannian metric $g$.Note that it suffices to prove the result in the case where $(S^2,g)$ is bumpy, i.e., the geodesic flow is nondegenerate. As in the case of Reeb flows, a $C^\infty$-generic Riemannian metric $g$ is bumpy. Moreover, both the Hofer--Zehnder capacity $c_{\rm HZ}(D^*_gS^2)$ and the diastole $\operatorname{dias}(g)$ depend continuously on smooth variations of the metric $g$.

By Theorem \ref{2bir} we find classes $x,y\in H_*(\Lambda^{<\mathrm{dias}(g)+\varepsilon} S^2,\Z_2)$ such that 
\[
\Delta(x)*y=1
    \in H_2(\Lambda^{<2\operatorname{dias}(g)+2\varepsilon}N,\mathbb Z_2),\qquad \forall\varepsilon > 0.
\]

Using the filtered version of Viterbo isomorphism as a BV-algebra, as in
\cite{Abouzaid15},
\[
(\operatorname{SH}^{<a}_*(D^*_gS^2), \Delta , *)\cong (H_2(\Lambda^{<a+2\varepsilon}N,\mathbb Z_2), \Delta, *)
\]
we are in exactly the setup to apply Theorem \ref{upper bound 2}, which proves Theorem \ref{thmSH} taking $\varepsilon\to 0$.

\appendix

\section{Period mapping and length estimates}\label{appendix: period map}

In this appendix, we use the ellipsoidal coordinates from Klingenberg \cite{Kl95}. Since the main text follows Karney's conventions \cite{Kar25b}, we begin by recording the relation between the two sets of notation.

\medskip
\noindent
\noindent
\textbf{Dictionary between Karney's and Klingenberg's conventions.}
In the main text, following Karney \cite{Kar25b}, we write the triaxial ellipsoid as
\[
E(a,b,c)=\left\{(x,y,z)\in\mathbb{R}^3\;\middle|\;
\frac{x^2}{a^2}+\frac{y^2}{b^2}+\frac{z^2}{c^2}=1\right\},
\qquad a>b>c>0.
\]
Klingenberg \cite{Kl95} writes the same ellipsoid as
\[
\frac{x_0^2}{a_0}+\frac{x_1^2}{a_1}+\frac{x_2^2}{a_2}=1,
\qquad 0<a_0<a_1<a_2.
\]
The two conventions are related by $a_0=c^2, a_1=b^2, a_2=a^2$ and  $x_0=z, x_1=y, x_2=x$. Klingenberg introduces ellipsoidal coordinates $(u_1,u_2)\in[a_0,a_1]\times[a_1,a_2]$ by
\begin{equation}\label{eq:Klingenberg-coordinates}
\begin{aligned}
x_0(u_1,u_2)^2
&=
\frac{a_0(u_1-a_0)(u_2-a_0)}
     {(a_1-a_0)(a_2-a_0)},\\
x_1(u_1,u_2)^2
&=
\frac{a_1(u_1-a_1)(u_2-a_1)}
     {(a_0-a_1)(a_2-a_1)},\\
x_2(u_1,u_2)^2
&=
\frac{a_2(u_1-a_2)(u_2-a_2)}
     {(a_0-a_2)(a_1-a_2)}.
\end{aligned}
\end{equation}
These coordinates are related to Karney's ellipsoidal latitude $\beta$ and longitude $\omega$ by
\begin{equation*}
    \begin{aligned}
    &u_1=a_1-(a_1-a_0)\cos^2\beta
    =c^2+(b^2-c^2)\sin^2\beta \\
    &u_2=a_1+(a_2-a_1)\sin^2\omega
    =b^2+(a^2-b^2)\sin^2\omega.
\end{aligned}
\end{equation*}
The principal ellipses are therefore given by
\[
\begin{aligned}
\gamma_z=E\cap\{z=0\}
       =E\cap\{x_0=0\}
&\quad\longleftrightarrow\quad \{u_1=a_0\},\\
\gamma_y=E\cap\{y=0\}
       =E\cap\{x_1=0\}
&\quad\longleftrightarrow\quad
\{u_1=a_1\}\cup\{u_2=a_1\},\\
\gamma_x=E\cap\{x=0\}
       =E\cap\{x_2=0\}
&\quad\longleftrightarrow\quad \{u_2=a_2\}.
\end{aligned}
\]
Thus $\gamma_x$, $\gamma_y$, and $\gamma_z$ are respectively the minor, median, and major principal ellipses. Let $F\in[-k'^2,k^2]$ denote Karney's first integral and let $\gamma\in[a_0,a_2]$ denote Klingenberg's separation constant. The two parameters are related by
\[
\gamma=a_1-(a_2-a_0)F
      =b^2-(a^2-c^2)F,
\]
or equivalently by
\[
F=\frac{a_1-\gamma}{a_2-a_0}
 =\frac{b^2-\gamma}{a^2-c^2}.
\]
Their distinguished values correspond as follows:
\[
\begin{array}{c|c|c}
\text{Karney} & \text{Klingenberg} & \text{limiting geodesic}\\
\hline
F=k^2 & \gamma=a_0=c^2 & \gamma_z\\
F=0 & \gamma=a_1=b^2 & \gamma_y\\
F=-k'^2 & \gamma=a_2=a^2 & \gamma_x.
\end{array}
\]
In particular,
\[
F>0
\quad\Longleftrightarrow\quad
\gamma\in(a_0,a_1)
\quad\Longleftrightarrow\quad
\text{circumpolar},
\]
whereas
\[
F<0
\quad\Longleftrightarrow\quad
\gamma\in(a_1,a_2)
\quad\Longleftrightarrow\quad
\text{transpolar}.
\]
We denote Klingenberg's period mapping by $\omega_{\mathrm K}(\gamma)$. When expressed in terms of Karney's first integral, the same map becomes
\[
\upomega(F):=
\omega_{\mathrm K}\bigl(a_1-(a_2-a_0)F\bigr).
\]
The function $\gamma\mapsto\omega_{\mathrm K}(\gamma)$ is strictly decreasing. Since the change of variables $F\mapsto\gamma$ is also strictly decreasing, the map $F\mapsto\upomega(F)$ is strictly increasing.\\

\noindent
\textbf{Separation of variables.}
From here on, we use Klingenberg's convention. In elliptic coordinates, the Riemannian metric on the ellipsoid is given by
\begin{equation}
\label{equation: pull-back metric}
ds^2=(u_2-u_1)\bigl(U_1\,du_1^2+U_2\,du_2^2\bigr),
\end{equation}
where
\[
U_i=U_i(u_i)=\frac{(-1)^i u_i}
{4(a_0-u_i)(a_1-u_i)(a_2-u_i)}.
\]
In particular, the $u_1$- and $u_2$-coordinate curves are lines of curvature. The geodesic equations admit a separation constant
$\gamma\in(a_0,a_1)\cup(a_1,a_2)$ and can be written as
\begin{equation}
\label{equation: geodesic equation on elliptic coordinates}
\frac{\sqrt{U_1}}{\sqrt{\gamma-u_1}}\,\frac{du_1}{ds}
=
\pm
\frac{\sqrt{U_2}}{\sqrt{u_2-\gamma}}\,\frac{du_2}{ds}.
\end{equation}
Equivalently,
\[
\gamma=u_1\sin^2\alpha+u_2\cos^2\alpha,
\]
where $\alpha$ is the angle between the tangent vector of the geodesic and the $u_1$-coordinate line. The principal ellipses arise as limiting cases: the shortest principal ellipse $\{x_2=0\}$ corresponds to $\gamma=a_2$, the intermediate principal ellipse $\{x_1=0\}$ to $\gamma=a_1$, and the longest principal ellipse $\{x_0=0\}$ to $\gamma=a_0$. Since $s$ is the arc-length parameter, \eqref{equation: pull-back metric} gives
\[
1=(u_2-u_1)\left(
U_1\left(\frac{du_1}{ds}\right)^2
+
U_2\left(\frac{du_2}{ds}\right)^2
\right).
\]
On the other hand, squaring \eqref{equation: geodesic equation on elliptic coordinates} gives
\[
\frac{U_1}{\gamma-u_1}\left(\frac{du_1}{ds}\right)^2
=
\frac{U_2}{u_2-\gamma}\left(\frac{du_2}{ds}\right)^2
=: \lambda.
\]
Hence
\[
U_1\left(\frac{du_1}{ds}\right)^2
=\lambda(\gamma-u_1),
\qquad
U_2\left(\frac{du_2}{ds}\right)^2
=\lambda(u_2-\gamma).
\]
Substituting into the unit-speed condition yields
\[
1=\lambda(u_2-u_1)^2.
\]
Therefore,
\[
\frac{du_1}{ds}
=
\pm\frac{\sqrt{\gamma-u_1}}
{(u_2-u_1)\sqrt{U_1}},
\qquad
\frac{du_2}{ds}
=
\pm\frac{\sqrt{u_2-\gamma}}
{(u_2-u_1)\sqrt{U_2}}.
\]
Introducing a new parameter $\tau$ along the geodesic by
\[
d\tau=\frac{ds}{u_2-u_1},
\]
we may rewrite these equations as
\begin{equation}
\label{Equation: preferred reparametrisation}
\frac{du_1}{d\tau}
=
\pm\frac{\sqrt{\gamma-u_1}}{\sqrt{U_1}},
\qquad
\frac{du_2}{d\tau}
=
\pm\frac{\sqrt{u_2-\gamma}}{\sqrt{U_2}}.
\end{equation}
The signs change whenever the corresponding coordinate reaches a turning value. Under this parametrization the two equations are separated: each depends only on one variable and on the parameter $\gamma$. Note that the solution is not parametrised by unit speed, i.e.,
\[
    \|(\dot{u}_1,\dot{u}_2)\|^2
    = u_2-u_1.
\]

\medskip

\noindent
\textbf{Definition of the period map.}
First suppose that $\gamma\in(a_0,a_1)$. The corresponding geodesic winds around the $x_0$-axis, while $u_1$ oscillates between $a_0$ and $\gamma$. Under the parametrization \eqref{Equation: preferred reparametrisation}, the $\tau$-time required for half an oscillation is
\[
\omega_2(\gamma)
= 2\cdot \int_{a_0}^\gamma \tfrac{1}{\dot u_1} d u_1=
\int_{a_0}^{\gamma}
\frac{\sqrt{u_1}\,du_1}
{\sqrt{(\gamma-u_1)(u_1-a_0)(a_1-u_1)(a_2-u_1)}}.
\]
Similarly, the $\tau$-time required for half a winding around the $x_0$-axis is
\[
\omega_1(\gamma)
=
\int_{a_1}^{a_2}
\frac{\sqrt{u_2}\,du_2}
{\sqrt{(u_2-\gamma)(u_2-a_0)(u_2-a_1)(a_2-u_2)}}.
\]
We define their ratio by
\begin{equation}
\label{equation: Klingenberg period map}
\omega_{\mathrm K}(\gamma)
:=
\frac{\omega_1(\gamma)}{\omega_2(\gamma)}.
\end{equation}
An analogous discussion applies to the case $\gamma\in(a_1,a_2)$. In this case, the corresponding geodesic winds around the $x_2$-axis, while $u_2$ oscillates between $\gamma$ and $a_2$. Under the parametrization \eqref{Equation: preferred reparametrisation}, one similarly defines the $\tau$-times required for a half-winding and a half-oscillation. We define $\omega_{\mathrm K}(\gamma)$ to be the ratio between the time required for a half-winding around the $x_2$-axis and the time required for a half-oscillation. This defines Klingenberg's period map
\[
\omega_{\mathrm K}:(a_0,a_1)\cup(a_1,a_2)\longrightarrow\mathbb R.
\]
It measures the ratio between winding and oscillation for a geodesic with parameter $\gamma$.\\

\noindent
\textbf{Properties of period mapping.} The following proposition implies that $\omega_K$ admits a continuous extension, which we again denote by $\omega_K : [a_0,a_2] \to \mathbb{R}$, and which is strictly decreasing and satisfies $\omega_K(a_1)=1$.

\begin{proposition}
\label{Proposition: Basic property of period mapping from Klingenberg}
    The ratio $\omega_K(\gamma)$, regarded as a function on $(a_0,a_1)\cup(a_1,a_2)$, is strictly decreasing on each interval. In particular, if $\gamma \in (a_0,a_1)$, then $\omega_K(\gamma)>1$, while if $\gamma \in (a_1,a_2)$, then $\omega_K(\gamma)<1$. Moreover, it satisfies
    \[
    \lim_{\gamma \to a_1-}\omega_K(\gamma)=\lim_{\gamma \to a_1+}\omega_K(\gamma)=1,
    \]
    and limits
    \[
    \lim_{\gamma \to a_0+}\omega_K(\gamma)\quad\text{and}\quad \lim_{\gamma \to a_2-}\omega_K(\gamma)
    \]
    exist and are finite.
\end{proposition}
\begin{proof}
    The monotonicity and the behavior near $\gamma = a_1$ follow from \cite[Theorem 3.5.14]{Kl95}. Since a monotone function on an open interval admits a continuous extension whenever it is bounded, it remains only to verify the finiteness of the limits as $\gamma\to a_0+$ and $\gamma\to a_2-$. This follows from the direction computations. The limit of half winding time satisfies
    \begin{equation*}
    \begin{aligned}
       \lim_{\gamma \to a_0+ }\omega_1(\gamma) &= \int_{a_1}^{a_2} \frac{1}{u_2-a_0} \times \frac{\sqrt{u_2}}{\sqrt{(u_2-a_1)(a_2-u_2)}}~du_2\\
       &\le  \int_{a_1}^{a_2} \frac{1}{a_1-a_0} \times \frac{\sqrt{u_2}}{\sqrt{(u_2-a_1)(a_2-u_2)}}~du_2  \le \frac{L(a_1, a_2)}{2(a_1-a_0)},
    \end{aligned}
    \end{equation*}
    and
    \begin{equation*}
    \begin{aligned}
       \lim_{\gamma \to a_0+ }\omega_2(\gamma) &= \lim_{\gamma \to a_0+}\int_{a_0}^\gamma \frac{\sqrt{u_1}}{\sqrt{(\gamma-u_1)(u_1-a_0)(a_1-u_1)(a_2-u_1)}}~du_1\\
       &= \frac{\sqrt{a_0}}{\sqrt{(a_1-a_0)(a_2-a_0)}}  \times \lim_{\gamma \to a_0+}\int_{a_0}^\gamma \frac{1}{\sqrt{(\gamma-u_1)(u_1-a_0)}}~du_1\\
       &= \frac{\sqrt{a_0}}{\sqrt{(a_1-a_0)(a_2-a_0)}}  \times \lim_{\gamma \to a_0+}\int_{0}^{1} \frac{1}{\sqrt{t(1-t)}}~dt\\
       &= \frac{\sqrt{a_0}}{\sqrt{(a_1-a_0)(a_2-a_0)}}.
    \end{aligned}
    \end{equation*} 
    Hence $\omega(\gamma)$ admits a finite limit as $\gamma \to a_0+$. The case $\gamma \to a_2-$ is analogous.
\end{proof}

\noindent
\textbf{Length estimates.} We now turn to the discussion of the length of closed geodesics. Recall that closed geodesics arise precisely when $\omega(\gamma)$ is a rational number. For $\gamma \in (a_0,a_1)$, assume that $\omega(\gamma)$ is a rational number, i.e.,
\[
\omega(\gamma) = \frac{m}{n}
\]
for some natural numbers $m > n$. Then, using the parametrisation given by \eqref{Equation: preferred reparametrisation} and $\Vert \dot u\Vert=u_2-u_1$, the length of the corresponding closed geodesic can be calculated as
\begin{equation}
\begin{aligned}
\int \Vert\dot{u}(t)\Vert ~ dt &= \int u_2(t)-u_1(t)~dt\\
        &= \int (u_2(t)-\gamma) + (\gamma-u_1(t))~dt\\
        &= 4n \cdot \int_{0}^{\omega_2(\gamma)/2} (u_2(t)-\gamma)~dt + 4m \cdot \int_{0}^{\omega_1(\gamma)/2} (\gamma-u_1(t))~dt\\
        &= 2n \cdot \int_{a_1}^{a_2} \frac{\sqrt{u_2(u_2-\gamma)}}{\sqrt{(u_2-a_0)(u_2-a_1)(a_2-u_2)}}~du_2\\ & \qquad + 2m \cdot \int_{a_0}^{\gamma} \frac{\sqrt{u_1(\gamma-u_1)}}{\sqrt{(u_1-a_0)(a_1-u_1)(a_2-u_1)}}~du_1.
\end{aligned}
\end{equation}
In particular, the average time for a half-winding is given by
$$
    T(\gamma):=  \int_{a_1}^{a_2} \frac{\sqrt{u_2(u_2-\gamma)}}{\sqrt{(u_2-a_0)(u_2-a_1)(a_2-u_2)}} ~du_2 +  \frac{\omega_1(\gamma)}{\omega_2(\gamma)}\cdot \int_{a_0}^{\gamma} \frac{\sqrt{u_1(\gamma-u_1)}}{\sqrt{(u_1-a_0)(a_1-u_1)(a_2-u_1)}}~du_1,
$$
whenever $\omega(\gamma) \in \mathbb{Q}$. The right-hand side admits a continuous extension, and thus $T(\gamma)$ extends to $\gamma \in (a_0, a_1)$ with $\omega(\gamma) \notin \mathbb{Q}$. Moreover, in these cases, this quantity still represents the uniform limit of the average time required for a half-winding. An analogous statement holds for $\gamma \in (a_1, a_2)$, where the extension of $T(\gamma)$ represents the uniform limit of the average time required for a half winding around the $x_2$-axis. As in Proposition \ref{Proposition: Basic property of period mapping from Klingenberg}, we show that $T(\gamma)$ is strictly decreasing and admits a continuous extension to $[a_0,a_2]$.

\begin{lemma}
\label{Lemma: Monotonicity of winding time}
The average time for a single half-winding $T(\gamma)$, regarded as a function on
$\gamma \in (a_0,a_1)\cup(a_1,a_2)$, is strictly decreasing on each interval.
Moreover,
\[
\lim_{\gamma\to a_1-}T(\gamma)
=
\lim_{\gamma\to a_1+}T(\gamma)
=
\frac12\,\ell(\gamma_y),
\]
and
\[
\lim_{\gamma\to a_0+}T(\gamma)
=
\frac12\,\ell(\gamma_z),
\qquad
\lim_{\gamma\to a_2-}T(\gamma)
=
\frac12\,\ell(\gamma_x).
\]
\end{lemma}

\begin{proof}
    Since the proof is analogous for $\gamma \in (a_0, a_1)$ and $\gamma \in (a_1, a_2)$, we only treat the case $\gamma \in (a_0, a_1)$. In this case, the derivatives of the relevant integrals are given by
    \begin{equation*}
        \begin{aligned}
            &\left( \int_{a_1}^{a_2} \frac{\sqrt{u_2(u_2-\gamma)}}{\sqrt{(u_2-a_0)(u_2-a_1)(a_2-u_2)}} ~du_2
            \right)^\prime\\
            &=
            \int_{a_1}^{a_2} \left( \frac{\sqrt{u_2(u_2-\gamma)}}{\sqrt{(u_2-a_0)(u_2-a_1)(a_2-u_2)}}
            \right)^\prime~du_2\\
            &= -\frac{1}{2}\cdot \int_{a_1}^{a_2} \frac{\sqrt{u_2}}{\sqrt{(u_2-\gamma)(u_2-a_0)(u_2-a_1)(a_2-u_2)}} ~du_2=-\frac{1}{2}\omega_1(\gamma).
        \end{aligned}
    \end{equation*}
    and
    \begin{equation*}
        \begin{aligned}
            &\left( \int_{a_0}^{\gamma} \frac{\sqrt{u_1(\gamma-u_1)}}{\sqrt{(u_1-a_0)(a_1-u_1)(a_2-u_1)}} ~du_1
            \right)^\prime\\
            &=
            \int_{a_0}^{\gamma} \left( \frac{\sqrt{u_1(\gamma-u_1)}}{\sqrt{(u_1-a_0)(a_1-u_1)(a_2-u_1)}}
            \right)^\prime~du_1 + \left[\frac{\sqrt{u_1(\gamma-u_1)}}{\sqrt{(u_1-a_0)(a_1-u_1)(a_2-u_1)}}\right]_{u_1=\gamma}\\
            &= \frac{1}{2}\cdot \int_{a_0}^{\gamma} \frac{\sqrt{u_1}}{\sqrt{(\gamma-u_1)(u_1-a_0)(a_1-u_1))(a_2-u_1)}} ~du_1 = \frac{1}{2}\omega_2(\gamma).
        \end{aligned}
    \end{equation*}
    Hence, the derivative of $T(\gamma)$ is
    \begin{equation}
    \begin{aligned}
    \label{Equation: derivative of length}
        \frac{dT}{d\gamma} &= -\frac{1}{2}\omega_1 (\gamma) + \frac{\omega_1(\gamma)}{\omega_2(\gamma)} \left(\frac{1}{2}\omega_2(\gamma)\right) + \left(\frac{\omega_1(\gamma)}{\omega_2(\gamma)}\right)^\prime \cdot \int_{a_0}^{\gamma} \frac{\sqrt{u_1(\gamma-u_1)}}{\sqrt{(u_1-a_0)(a_1-u_1)(a_2-u_1)}} ~du_2\\
        &=\left(\frac{\omega_1(\gamma)}{\omega_2(\gamma)}\right)^\prime \cdot \int_{a_0}^{\gamma} \frac{\sqrt{u_1(\gamma-u_1)}}{\sqrt{(u_1-a_0)(a_1-u_1)(a_2-u_1)}} ~du_2.
    \end{aligned}
    \end{equation}
       Since $\upomega(\gamma)$ is a strictly decreasing function, this proves the first claim of the lemma. Next, we consider the limit as $\gamma \to a_0+$. In this case, the integral converges to an integral without the terms $(u_i-a_0)$ and $(u_i-\gamma)$. Hence, we have
    \begin{equation}
        \lim_{\gamma \to a_0+} T(\gamma) = \int_{a_1}^{a_2} \frac{\sqrt{u_2}}{\sqrt{(u_2-a_1)(a_2-u_2)}}~du_2.
    \end{equation}
    As $\gamma \to a_1-$, the integral converges to an integral without the terms $(u_i-\gamma)$ and $(u_i-a_1)$. Moreover, by Proposition \ref{Proposition: Basic property of period mapping from Klingenberg}, the ratio $\omega_1(\gamma)/\omega_2(\gamma)$ converges to $1$. Hence,
    \begin{equation}
    \begin{aligned}
        \lim_{\gamma \to a_1+} T(\gamma) &= \int_{a_1}^{a_2} \frac{\sqrt{u_2}}{\sqrt{(u_2-a_0)(a_2-u_2)}}~du_2 + 1 \cdot\int_{a_0}^{a_1} \frac{\sqrt{u_1}}{\sqrt{(u_1-a_0)(a_2-u_2)}}~du_1\\
        &= \int_{a_0}^{a_2}\frac{\sqrt{u}}{\sqrt{(u-a_0)(a_2-u)}}~du.
    \end{aligned}
    \end{equation}
    The case $\gamma \in (a_1,a_2)$ is analogous. This completes the proof.
\end{proof}

Next, we compare the length of the geodesic chord $\gamma_y\setminus L$, where $L\subset\gamma_y$ denotes the geodesic segment joining the two neighboring northern umbilical points, with that of the shortest transpolar simple closed geodesic $\eta$. By Lemma \ref{Lemma: Monotonicity of winding time}, the half length of the shortest transpolar simple closed geodesic is given by $T(\gamma_o)$, where
\[
\gamma_o =
\begin{cases}
    a_0 & \text{if}~\lim_{\gamma \to a_0} \omega(\gamma) \le 2\\
    \omega^{-1}(2) & \text{if}~\lim_{\gamma \to a_0} \omega(\gamma) > 2.
\end{cases}
\]
Note that $\gamma_y\setminus L$ can be understood as the union of three geodesic segments, one segments of $\{u_1=a_1\}$ and two segments of $\{u_2=a_1\}$

\begin{lemma}
\label{Lemma: 2-bounce orbit vs index-3 orbit}
The following inequality holds:
\[
\ell(\gamma_y \setminus L )  \ge T(\gamma_o)=\ell(\eta).
\]
\end{lemma}
\begin{proof}
    Recall that half length $T(\gamma_o)$ according to the function
    \begin{equation*}
    \begin{aligned}
    T(\gamma) =  \int_{a_1}^{a_2} &\frac{\sqrt{u_2(u_2-\gamma)}}{\sqrt{(u_2-a_0)(u_2-a_1)(a_2-u_2)}} ~du_2\\
    &+ \upomega(\gamma) \cdot \int_{a_0}^{\gamma} \frac{\sqrt{u_1(\gamma-u_1)}}{\sqrt{(u_1-a_0)(a_1-u_1)(a_2-u_1)}}~du_1.
    \end{aligned}
    \end{equation*}
    \[
    \]
    Then the derivative of a function $T:(a_0, a_1) \to \mathbb{R}$ (see \eqref{Equation: derivative of length}) satisfies
    \begin{equation*}
    \begin{aligned}
    \frac{dT}{d\gamma} = &\upomega^\prime(\gamma) \int_{a_0}^{\gamma} \frac{\sqrt{u_1(\gamma-u_1)}}{\sqrt{(u_1-a_0)(a_1-u_1)(a_2-u_1)}}~du_1 \\ &\ge \upomega^\prime(\gamma) \int_{a_0}^{a_1} \frac{\sqrt{u_1(a_1-u_1)}}{\sqrt{(u_1-a_0)(a_1-u_1)(a_2-u_1)}}~du_1\\
    &= \upomega^\prime(\gamma) \int_{a_0}^{a_1} \frac{\sqrt{u_1}}{\sqrt{(u_1-a_0)(a_2-u_1)}}~du_1.
    \end{aligned}
    \end{equation*}
    In particular, we have
    \begin{equation}
    \label{Equation: Proof of comparing length}
    \begin{aligned}
    T(a_1) - T(\gamma_o)  &\ge  (\upomega(a_1)-\upomega(\gamma_o))
    \cdot \int_{a_0}^{a_1} \frac{\sqrt{u_1}}{\sqrt{(u_1-a_0)(a_2-u_1)}}~du_1\\
    &\ge -\int_{a_0}^{a_1} \frac{\sqrt{u_1}}{\sqrt{(u_1-a_0)(a_2-u_1)}}~du_1.
    \end{aligned}
    \end{equation}
    The second inequality follows from $\omega(\gamma_o)\leq 2$ in both cases defining $\gamma_o$.

    The length of $\gamma_y\setminus L$ can be computed directly:
    \begin{equation*}
    \begin{aligned}
    \operatorname{length}(\gamma_y\setminus L)
    &= 2\cdot \int_{a_0}^{a_1} \frac{\sqrt{u_1}}{\sqrt{(u_1-a_0)(a_2-u_1)}}~du_1 + \int_{a_1}^{a_2} \frac{\sqrt{u_2}}{\sqrt{(u_2-a_0)(a_2-u_2)}}~du_2\\
    &= 1\cdot \int_{a_0}^{a_1} \frac{\sqrt{u_1}}{\sqrt{(u_1-a_0)(a_2-u_1)}}~du_1 + \int_{a_0}^{a_2} \frac{\sqrt{u}}{\sqrt{(u-a_0)(a_2-u)}}~du. 
    \end{aligned}
    \end{equation*}
    The second integral is the half-circumference of the middle principal ellipse $\gamma_y$, and hence equals $T(a_1)$ by Lemma \ref{Lemma: Monotonicity of winding time}. Therefore, together with \eqref{Equation: Proof of comparing length}, we obtain
    \begin{equation*}
    \begin{aligned}
    \operatorname{length}(\gamma_y\setminus L)=T(a_1) + \int_{a_0}^{a_1} \frac{\sqrt{u_1}}{\sqrt{(u_1-a_0)(a_2-u_1)}}~du_1 \ge T(\gamma_o).
    \end{aligned}
    \end{equation*}
    This completes the proof.    
\end{proof}

\vspace{.5cm}

\section{Morse indices}\label{appendix: morse indices}
In this section we verify the index computations in Klingenberg \cite[Ch.~3.5]{Kl95}. For periodic geodesics on $(S^2,g)$, the Conley--Zehnder index agrees with the Morse index of the energy functional of the corresponding Riemannian metric $g$. 

\begin{proposition}\label{Prop: char of low index}
Let $E=E(a,b,c)$, $a>b>c>0$, be a triaxial ellipsoid, and denote by
$\gamma_x,\gamma_y,\gamma_z$ the minor, median, and major principal ellipses,
respectively. Let $\omega:(-k'^2,k^2)\to\mathbb R$ be the period map from
Section~2.1. Then the following hold.

\begin{enumerate}
    \item The minor principal ellipse $\gamma_x$ has Morse index $1$, while
    the median principal ellipse $\gamma_y$ has Morse index $2$.

    \item The major principal ellipse $\gamma_z$ has odd Morse index at least
    $3$. Its index increases by $2$ whenever, under a deformation of the
    ellipsoid, a new Morse--Bott family of simple circumpolar closed geodesics
    bifurcates from $\gamma_z$.

    \item If $\lim_{f\to k^2}\upomega(f)<2$, then no nonprincipal simple
    circumpolar closed geodesic exists and
    $\mu(\gamma_z)=3$.

    \item If $\lim_{f\to k^2}\upomega(f)>2$, then there is a unique regular
    level $f_2\in(0,k^2)$ satisfying $\upomega(f_2)=2$. The corresponding
    critical set is a Morse--Bott $S^1$-family of simple circumpolar closed
    geodesics. Every geodesic $\eta$ in this family winds once around the
    $z$-axis, intersects $\gamma_z$ four times, and has Morse index $\mu(\eta)=3$ and nullity $\nu(\eta)=1$.
\end{enumerate}

In particular, the shortest simple closed geodesic of Morse index $3$ is
$\gamma_z$ if $\lim_{f\to k^2}\upomega(f)<2$, and otherwise belongs to the
family characterized by $\upomega(f)=2$.
\end{proposition}

To compute these indices, we follow the approach of \cite{AFFvK13} and use the following key ingredients:
\begin{itemize}
    \item We regard closed geodesics as families depending on the parameters $(a,b,c)$. The Morse index is constant along any continuous family of Morse--Bott non-degenerate closed geodesics.
    
    \item At bifurcation values we use the local invariance of Morse homology. More precisely, suppose that the index of a family of closed geodesics changes from $k$ to $k+2$ as a parameter crosses a bifurcation value and that, at the same value, a Morse--Bott family of closed geodesics is born. Then the bifurcated family has Morse index $k$ and nullity $1$. This is the mechanism used in \cite{AFFvK13} to determine the indices of the families created at successive bifurcations.
\end{itemize}

After normalizing $1=a\geq b\geq c$, there are two parameters left to vary. Our strategy is as follows:
\begin{enumerate}
    \item recall the indices for the round sphere $a=b=c=1$;
    \item determine the indices for $a=b$ and $b=c$ near the round metric using curvature bounds;
    \item determine the indices for all ellipsoids of revolution using bifurcation arguments;
    \item determine the indices for general triaxial ellipsoids by varying $b$ from $a$ to $c$ and analyzing bifurcations and collapses of closed geodesics onto the coordinate ellipses.
\end{enumerate}

\noindent
\textbf{Round sphere.}
Closed geodesics on the round sphere occur in $\rp^3$-families, one for each covering number $m\in\N$. Denote the corresponding critical manifold by $C_m$. For the round sphere of radius $1$,
\[
\mu(C_m)=2m-1.
\]

\begin{remark}
Let $\gamma$ be a closed, non-degenerate, unit-speed geodesic of period $T$ on a surface, and assume that the Gaussian curvature along $\gamma$ is constant $\kappa>0$. Then
\[
\mu(\gamma)=2\left\lfloor \frac{\sqrt{\kappa}\,T}{2\pi}\right\rfloor+1.
\]
For the round sphere closed geodesics are degenerate, so the formula does not apply directly. If one evaluates it formally, one gets $2m+1$, but this counts the two zero modes. One needs to subtract the reduced nullity $2$, and hence
\[
\mu(C_m)=2m+1-2=2m-1.
\]
\end{remark}

\noindent
\textbf{Ellipsoids of revolution near the round sphere.}
For ellipsoids of revolution sufficiently close to the round sphere, the critical manifold
$C_m\cong \rp^3$ splits into two circles $S_{m,\pm}$, corresponding to the two orientations
of the $m$-fold covered equator, and a torus $T_m$ formed by the $m$-fold covered meridians. We first determine the indices of $S_{m,\pm}$ using the previous remark. In the prolate case $E(1,c,c)$, the equator has period $2m\pi c$ and constant Gaussian curvature $\kappa=1$.
Hence
\[
\mu(S_{m,\pm}^p)
=
2\left\lfloor mc\right\rfloor+1.
\]
For $c<1$ sufficiently close to $1$, this gives
\[
\mu(S_{m,\pm}^p)=2m-1.
\]
In the oblate case $E(1,1,c)$, the equator has period $2m\pi$ and constant Gaussian curvature
$\kappa=c^{-2}$. Thus
\[
\mu(S_{m,\pm}^o)
=
2\left\lfloor \frac{m}{c}\right\rfloor+1,
\]
and for $c<1$ sufficiently close to $1$,
\[
\mu(S_{m,\pm}^o)=2m+1.
\]
The indices of the meridian families $T_m$ then follow from invariance of local Morse homology
under the splitting of the Morse--Bott family $C_m$ into $S_{m,\pm}$ and $T_m$. This yields
\[
\mu(T_m^p)=2m,
\qquad
\mu(T_m^o)=2m-1.
\]

\noindent
\textbf{General ellipsoids of revolution.} 
We now determine the indices of the closed geodesics on general ellipsoids of revolution by varying the parameter $c$ from $1$ to $0$. The $S^1$-families $S_{m,\pm}^{o/p}$ are
particularly easy to analyze, since the curvature along each such geodesic is constant:
$\kappa=c^2$ in the prolate case and $\kappa=\frac{1}{c^2}$ in the oblate case. Since these
geodesics have period $2m\pi$, the index formula for constant curvature geodesics gives
\[
\mu(S_{m,\pm}^p)=2\left\lfloor mc\right\rfloor+1,
\qquad
\mu(S_{m,\pm}^o)=2\left\lfloor \frac{m}{c}\right\rfloor+1.
\]

\paragraph{\textit{Prolate case}.} As $c$ decreases from $1$ to $0$, the index of $S_{m,\pm}^p$ drops by
$2$ precisely when $c=\frac{k}{m}$ for some $k\in\{1,\ldots,m-1\}$. At such a parameter
value the corresponding closed geodesic becomes degenerate, and a new family of closed
geodesics bifurcates. Since the geodesic flow is Liouville integrable, regular level sets of
the two first integrals are invariant tori, and hence the bifurcating families are
Morse--Bott $2$-tori. By invariance of local Morse homology, the bifurcating torus has index
\[
\mu(B^p_{k,m,\pm})=2k,
\qquad k\in\{1,\ldots,m-1\}.
\]
These tori are characterized by $\upomega=\tfrac{k}{m}$ in terms of the period map introduced in Section \ref{secbilliard}.

\paragraph{\textit{Oblate case.}}
As $c$ varies from $1$ to $0$, the index of $S_{m,\pm}^o$ increases by $2$ exactly when
$c=\frac{m}{k}$ for integers $k>m$. Again, at such a parameter value the corresponding orbit
becomes degenerate, and deformation invariance of local Morse homology implies that a new
family of closed geodesics bifurcates. These families are $2$-tori, which we denote by
$B_{k,m,\pm}^o$. More precisely, as $c$ crosses $\frac{m}{k}$, the index of $S_{m,\pm}^o$ changes from
$2k-1$ to $2k+1$. Hence the bifurcating torus must contribute in the intermediate degrees,
and therefore its Morse index is
\[
\mu(B_{k,m,\pm}^o)=2k-1,
\qquad k\in\N,\ k>m.
\]
These tori are characterized by $\upomega =\tfrac{k}{m}$ in terms of the period map introduced in Section \ref{secbilliard}.\\

\paragraph{\textit{No further bifurcations.}}
The torus families $T_m^{o/p}$ and $B_{k,m,\pm}^{o/p}$ do not undergo further
bifurcations. Indeed, on each regular invariant torus of the geodesic flow, the motion is conjugate to a linear flow with slope $\upomega(f)$, where $f$ denotes the value of the first integral labeling the torus.
A family of closed geodesics occurs precisely when $\upomega(f)\in\mathbb{Q}$. Since the period map $\upomega:(-k'^2,k^2)\to (0,a/c)$ is strictly monotonically increasing, each rational value is attained for at most one value of $F$. Therefore, a regular torus family of closed geodesics persists uniquely until it reaches a singular level of the Liouville foliation, where the torus collapses to either the major or minor principal ellipse. In particular, the families $T_m^{o/p}$ and $B_{k,m,\pm}^{o/p}$ do not
bifurcate further, and their indices remain constant as $c$ varies.\\

\noindent
\textbf{General ellipsoids.}
We now keep $a=1$ and $c$ fixed and vary $b$ from $a=1$ to $c$. For $b\approx a$,
the torus $T_m^o$ splits into two pairs of $S^1$-families, which we denote by
$H_{m,\pm}^o$ and $E_{m,\pm}^o$\footnote{$H$ for hyperbolic and $E$ for elliptic.}.
Their indices are
\[
\mu(H_{m,\pm}^o)=2m,\qquad
\mu(E_{m,\pm}^o)=2m-1.
\]
As $b\to c$, bifurcations and collapses occur in such a way that for $b\approx c$ we end up
with two $S^1$-families $H_{m,\pm}^p$ and $E_{m,\pm}^p$ which merge into the torus $T_m^p$ at $b=c$.
Their indices are
\[
\mu(H_{m,\pm}^p)=2m,\qquad \mu(E_{m,\pm}^p)=2m+1.
\]

\noindent
The families $H_{m,\pm}^o, H_{m,\pm}^p$ are characterized by $\upomega=1$ and therefore cannot meet any regular invariant torus as these are characterized by (constant) rational numbers $\upomega>1$ in the prolate case and $\upomega<1$ in the oblate case. Hence the $H_{m,\pm}^o$'s  neither bifurcate nor collapse and persist for all $b\in(c,a)$. In the limit $b\to c$, they become precisely the families
$H_{m,\pm}^p$.\\

\noindent
By contrast, the $E^o_{m,\pm}, E^p_{m,\pm}$ correspond to the limits $\lim_{f\to k'^2}\upomega(f)>1, \lim_{f\to k^2}\upomega(f)<1$ and as $b\to c$ the will take rational values. This is precisely when bifurcations and collapses occur. The limits $\lim_{f\to k'^2}\upomega(f)>1, \lim_{f\to k^2}\upomega(f)<1$ are monotone as functions in $b\in [a,c]$, c.f.\ \cite{Kl95}, therefore the $E^o_{m,\pm}$ bifurcate into the $B^p_{k,m,\pm}$, while the $B^o_{k,m,\pm}$ collapse into the $E^p_{m,\pm}$, as $b$ varies from $a$ to $c$. Schematically,
\[
H^o_{m,\pm}\rightsquigarrow H^p_{m,\pm},\qquad
E^o_{m,\pm}\rightsquigarrow B^p_{k,m,\pm},\qquad
B^o_{k,m,\pm}\rightsquigarrow E^p_{m,\pm},
\]
and no other bifurcations or collapses occur. Moreover, the intermediate indices are determined by the boundary ellipsoid-of-revolution cases. The proposition is the special case $m=1$.\\


\bibliography{literatur}

\newcommand{\etalchar}[1]{$^{#1}$}
\begin{thebibliography}{CCdM{\etalchar{+}}21}

\bibitem[ABB{\etalchar{+}}26]{OP26}
Bernhard Albach, Jean-Fran{\c{c}}ois Barraud, Misha Bialy, Johanna Bimmermann, Ana Ch{\'a}vez~C{\'a}liz, Mihai Damian, Lina Deschamps, Umberto Hryniewicz, Vincent Humili{\`e}re, Boris Khesin, Levin Maier, Agustin Moreno, Alexandru Oancea, Olga Paris-Romaskevich, Alfonso Sorrentino, and Serge Tabachnikov.
\newblock Open problems in billiards and quantitative symplectic geometry.
\newblock {\em arXiv preprint arXiv:2602.12896}, 2026.

\bibitem[Abo15]{Abouzaid15}
Mohammed Abouzaid.
\newblock Symplectic cohomology and {V}iterbo's theorem.
\newblock {\em Inventiones Mathematicae}, 199(2):533--627, 2015.

\bibitem[AEK24]{AEK24}
Alberto Abbondandolo, Oliver Edtmair, and Jungsoo Kang.
\newblock On closed characteristics of minimal action on a convex three-sphere.
\newblock {\em arXiv preprint arXiv:2412.01777}, 2024.

\bibitem[AFFvK13]{AFFvK13}
Peter Albers, Joel~W. Fish, Urs Frauenfelder, and Otto van Koert.
\newblock The {C}onley-{Z}ehnder indices of the rotating {K}epler problem.
\newblock {\em Mathematical Proceedings of the Cambridge Philosophical Society}, 154(2):243--260, 2013.

\bibitem[AFO17]{AFO17}
Peter Albers, Urs Frauenfelder, and Alexandru Oancea.
\newblock Local systems on the free loop space and finiteness of the {H}ofer--{Z}ehnder capacity.
\newblock {\em Mathematische Annalen}, 367(3):1403--1428, 2017.

\bibitem[AM11]{AM11}
Peter Albers and Marco Mazzucchelli.
\newblock Periodic bounce orbits of prescribed energy.
\newblock {\em International Mathematics Research Notices}, 2011(14):3289--3314, 2011.

\bibitem[BC25]{BC25}
Filip Bro{\'c}i{\'c} and Dylan Cant.
\newblock Parametric {G}romov width of {L}iouville domains.
\newblock {\em arXiv preprint arXiv:2504.15207}, 2025.

\bibitem[BG89]{BG89}
V.~Benci and F.~Giannoni.
\newblock Periodic bounce trajectories with a low number of bounce points.
\newblock {\em Ann. Inst. Henri Poincar{\'e}, Anal. Non Lin{\'e}aire}, 6(1):73--93, 1989.

\bibitem[Bim24]{Bim23}
Johanna Bimmermann.
\newblock Hofer--{Z}ehnder capacity of disc tangent bundles of projective spaces.
\newblock {\em Journal of the London Mathematical Society}, 110(1):e12948, 2024.

\bibitem[Bim25]{Bim25}
Johanna Bimmermann.
\newblock Symplectic crown domains and their capacities.
\newblock {\em arXiv preprint arXiv:2505.02731}, 2025.

\bibitem[BK22]{BK22}
Gabriele Benedetti and Jungsoo Kang.
\newblock Relative {H}ofer--{Z}ehnder capacity and positive symplectic homology.
\newblock {\em Journal of Fixed Point Theory and Applications}, 24(2):44, 2022.

\bibitem[BM24]{BM24}
Johanna Bimmermann and Levin Maier.
\newblock Magnetic billiards and the {H}ofer-{Z}ehnder capacity of disk tangent bundles of lens spaces.
\newblock {\em arXiv preprint arXiv:2403.06761}, 2024.

\bibitem[BO12]{BO13}
Fr{\'e}d{\'e}ric Bourgeois and Alexandru Oancea.
\newblock The {G}ysin exact sequence for {$S^1$}-equivariant symplectic homology.
\newblock {\em Journal of Topology}, 5(2):361--407, 2012.

\bibitem[BPS03]{Biran-Polterovich-Salamon03}
Paul Biran, Leonid Polterovich, and Dietmar Salamon.
\newblock Propagation in {H}amiltonian dynamics and relative symplectic homology.
\newblock {\em Duke Math. J.}, 119(1):65--118, 2003.

\bibitem[BS10]{BS10}
Florent Balacheff and St{\'e}phane Sabourau.
\newblock Diastolic and isoperimetric inequalities on surfaces.
\newblock {\em Annales scientifiques de l'{\'E}cole Normale Sup{\'e}rieure}, 43(4):579--605, 2010.

\bibitem[CCdM{\etalchar{+}}21]{CCdMOR21}
Erin~Wolf Chambers, Gregory~R. Chambers, Arnaud de~Mesmay, Tim Ophelders, and Regina Rotman.
\newblock Constructing monotone homotopies and sweepouts.
\newblock {\em Journal of Differential Geometry}, 119(3):383--401, 2021.

\bibitem[CS99]{ChasSullivan1999}
Moira Chas and Dennis Sullivan.
\newblock String topology, 1999.

\bibitem[DL19]{Diogo-Lisi19}
Lu\'is Diogo and Samuel~T. Lisi.
\newblock Symplectic homology of complements of smooth divisors.
\newblock {\em J. Topol.}, 12(3):967--1030, 2019.

\bibitem[Dui76]{Duistermaat76}
Johannes~Jisse Duistermaat.
\newblock On the morse index in variational calculus.
\newblock {\em Advances in Mathematics}, 21(2):173--195, 1976.

\bibitem[FHS95]{FHS95}
Andreas Floer, Helmut Hofer, and Dietmar Salamon.
\newblock Transversality in elliptic {M}orse theory for the symplectic action.
\newblock {\em Duke Math. J.}, 80(1):251--292, 1995.

\bibitem[FP17]{FP17}
Urs Frauenfelder and Andrei Pajitnov.
\newblock Finiteness of $\pi_1$-sensitive {H}ofer--{Z}ehnder capacity and equivariant loop space homology.
\newblock {\em Journal of Fixed Point Theory and Applications}, 19:3--15, 2017.

\bibitem[FRV23]{FRV23}
Brayan Ferreira, Vinicius~GB Ramos, and Alejandro Vicente.
\newblock Gromov width of the disk cotangent bundle of spheres of revolution.
\newblock {\em arXiv preprint arXiv:2301.08528}, 2023.

\bibitem[Har49]{Har49}
AS~Hart.
\newblock Geometrical demonstration of some properties of geodesic lines.
\newblock {\em Cambridge and Dublin Mathematical Journal}, 4:80--84, 1849.

\bibitem[HLS97]{HLS97}
Helmut Hofer, V\'eronique Lizan, and Jean-Claude Sikorav.
\newblock On genericity for holomorphic curves in four-dimensional almost-complex manifolds.
\newblock {\em J. Geom. Anal.}, 7(1):149--159, 1997.

\bibitem[Hof93]{Hofer93}
Helmut Hofer.
\newblock Pseudoholomorphic curves in symplectizations with applications to the weinstein conjecture in dimension three.
\newblock {\em Inventiones mathematicae}, 114(1):515--563, 1993.

\bibitem[HV92]{HV92}
H.~Hofer and C.~Viterbo.
\newblock The {W}einstein conjecture in the presence of holomorphic spheres.
\newblock {\em Communications on Pure and Applied Mathematics}, 45(5):583--622, 1992.

\bibitem[HWZ96]{HWZ-1}
H.~Hofer, K.~Wysocki, and E.~Zehnder.
\newblock Properties of pseudoholomorphic curves in symplectisations. {I}. {A}symptotics.
\newblock {\em Ann. Inst. H. Poincar\'e{} C Anal. Non Lin\'eaire}, 13(3):337--379, 1996.

\bibitem[HWZ99]{HWZ-3}
H.~Hofer, K.~Wysocki, and E.~Zehnder.
\newblock Properties of pseudoholomorphic curves in symplectizations. {III}. {F}redholm theory.
\newblock In {\em Topics in nonlinear analysis}, volume~35 of {\em Progr. Nonlinear Differential Equations Appl.}, pages 381--475. Birkh\"auser, Basel, 1999.

\bibitem[HWZ03]{HWZ03}
H.~Hofer, K.~Wysocki, and E.~Zehnder.
\newblock Finite energy foliations of tight three-spheres and {H}amiltonian dynamics.
\newblock {\em Ann. of Math. (2)}, 157(1):125--255, 2003.

\bibitem[HZ90]{HZ90}
Helmut Hofer and Eduard Zehnder.
\newblock A new capacity for symplectic manifolds.
\newblock In {\em Analysis, et cetera}, pages 405--427. Elsevier, 1990.

\bibitem[HZ11]{HZ94}
Helmut Hofer and Eduard Zehnder.
\newblock {\em Symplectic invariants and {H}amiltonian dynamics}.
\newblock Birkh{\"a}user, 2011.

\bibitem[Iri14]{Irie14}
Kei Irie.
\newblock Hofer--{Z}ehnder capacity of unit disk cotangent bundles and the loop product.
\newblock {\em Journal of the European Mathematical Society}, 016:2477--2497, 2014.

\bibitem[Jac39]{Jac39}
Carl Gustav~Jakob Jacobi.
\newblock Note von der geodätischen {L}inie auf einem {E}llipsoid und den verschiedenen {A}nwendungen einer merkwürdigen analytischen {S}ubstitution.
\newblock {\em Journal für die Reine und Angewandte Mathematik}, 19:309--313, 1839.
\newblock Letter to Bessel, Dec. 28, 1838. French translation (1841). S2CID 121670851.

\bibitem[Kar25]{Kar25}
Charles F.~F. Karney.
\newblock Geographiclib.
\newblock \url{https://geographiclib.sourceforge.io/C++/2.7}, 2025.
\newblock Version 2.7, 2025-11-06.

\bibitem[Kar26]{Kar25b}
Charles F.~F. Karney.
\newblock Jacobi's solution for geodesics on a triaxial ellipsoid.
\newblock {\em Journal of Geodesy}, 100(2):17, 2026.

\bibitem[Kli95]{Kl95}
Wilhelm Klingenberg.
\newblock {\em Riemannian geometry}, volume~1.
\newblock Walter de Gruyter, 1995.

\bibitem[Ler95]{ler}
Eugene Lerman.
\newblock Symplectic cuts.
\newblock {\em Math. Res. Lett.}, 2(3):247--258, 1995.

\bibitem[Lu06]{Lu06}
Guangcun Lu.
\newblock Gromov-{W}itten invariants and pseudo symplectic capacities.
\newblock {\em Israel Journal of Mathematics}, 156(1):1--63, 2006.

\bibitem[Mos80]{Mos80}
J{\"u}rgen Moser.
\newblock Various aspects of integrable {Hamiltonian} systems.
\newblock Dynamical systems, {C}.{I}.{M}.{E}. {Lect}., {Bressanone} 1978, {Prog}. {Math}. 8, 233-290 (1980)., 1980.

\bibitem[MS12]{MS12}
Dusa McDuff and Dietmar Salamon.
\newblock {\em {$J$}-holomorphic curves and symplectic topology}, volume~52 of {\em American Mathematical Society Colloquium Publications}.
\newblock American Mathematical Society, Providence, RI, second edition, 2012.

\bibitem[MS17]{MS17}
Dusa McDuff and Dietmar Salamon.
\newblock {\em Introduction to symplectic topology}.
\newblock Oxford Graduate Texts in Mathematics. Oxford University Press, Oxford, third edition, 2017.

\bibitem[Nir53]{Nir53}
Louis Nirenberg.
\newblock The {W}eyl and {M}inkowski problems in differential geometry in the large.
\newblock {\em Communications on Pure and Applied Mathematics}, 6(3):337--394, 1953.

\bibitem[Rit13]{Ritter13}
Alexander~F. Ritter.
\newblock Topological quantum field theory structure on symplectic cohomology.
\newblock {\em Journal of Topology}, 6(2):391--489, 2013.

\bibitem[She22]{Shelukhin22}
Egor Shelukhin.
\newblock Viterbo conjecture for {Z}oll symmetric spaces.
\newblock {\em Invent. Math.}, 230(1):321--373, 2022.

\bibitem[Sie17]{Siefring17}
Richard Siefring.
\newblock Finite-energy pseudoholomorphic planes with multiple asymptotic limits.
\newblock {\em Mathematische Annalen}, 368(1-2):367--390, 2017.

\bibitem[SZ92]{Salamon-Zehnder92}
Dietmar Salamon and Eduard Zehnder.
\newblock Morse theory for periodic solutions of {H}amiltonian systems and the {M}aslov index.
\newblock {\em Comm. Pure Appl. Math.}, 45(10):1303--1360, 1992.

\bibitem[Ush12]{U12}
Michael Usher.
\newblock Many closed symplectic manifolds have infinite {H}ofer--{Z}ehnder capacity.
\newblock {\em Transactions of the American Mathematical Society}, 364(11):5913--5943, 2012.

\bibitem[Voc21]{V21}
Anna-Maria Vocke.
\newblock {\em Periodic bounce orbits in magnetic billiard systems}.
\newblock PhD thesis, 2021.

\bibitem[Web06]{Weber06}
Joa Weber.
\newblock Noncontractible periodic orbits in cotangent bundles and {F}loer homology.
\newblock {\em Duke Math. J.}, 133(3):527--568, 2006.

\bibitem[Zha16]{Zhao16}
Jingyu Zhao.
\newblock {\em Periodic symplectic cohomologies and obstructions to exact Lagrangian immersions}.
\newblock PhD thesis, Columbia University, 2016.
\newblock PhD thesis.

\end{thebibliography}
\bibliographystyle{alpha}
\end{document}